\documentclass[final]{siamltex}

\usepackage{algorithm}
\usepackage{algorithmic}
\usepackage{amsmath}
\usepackage{amsfonts}
\usepackage{multirow}
\usepackage{graphicx,epsf,subfigure,epstopdf}
\usepackage{color}
\usepackage[top=1in, bottom=1in, left=1in, right=1in, dvips,letterpaper]{geometry}
\usepackage{setspace}
\newcommand\argmin{\mathop{\textrm{argmin}}}
\newcommand\crule[1][5cm]{%
  \par
  \nointerlineskip
  \centerline{\hbox to #1{\hrulefill}}%
  \nointerlineskip}

\numberwithin{equation}{section}
\numberwithin{algorithm}{section}
\newtheorem{thm}{{\sc Theorem}}[section]
\newtheorem{lem}{{\sc Lemma}}[section]
\newtheorem{cor}[thm]{Corollary}
\newtheorem{rem}{Remark}[section]
\newtheorem{prop}{Proposition}[section]
\newtheorem{defi}{{\sc Definition}}[section]

\newcommand{\R}{\mathbb{R}}
\newcommand{\E}{\mathbb{E}}

\newcommand{\SFO}{\mathcal{SFO}}

\newcommand{\C}[1]{{\cal {#1}}}

\newcommand{\be}{\begin{equation}}
\newcommand{\ee}{\end{equation}}

\newcommand{\etal}{{\it et al.\ }}

\title{Stochastic Quasi-Newton Methods for \\ Nonconvex Stochastic Optimization}

\author{Xiao Wang
\thanks{School of Mathematical Sciences, University of Chinese Academy of Sciences; Key Laboratory of Big Data Mining and Knowledge Management, Chinese Academy of Sciences, China. Email: wangxiao@ucas.ac.cn. Research of this author was supported in part by UCAS President Grant Y35101AY00 and NSFC Grant 11301505.}
\and Shiqian Ma
\thanks{Department of Systems Engineering and Engineering Management, The Chinese University of Hong Kong, Shatin, N. T., Hong Kong, China. Email: sqma@se.cuhk.edu.hk. Research of this author was supported in part by the Hong Kong Research Grants Council General Research Fund (Grant 14205314).}
\and Donald Goldfarb
\thanks{Department of Industrial Engineering and Operations Research, Columbia University, New York, NY, USA. Email: goldfarb@columbia.edu. Research of this author was supported in part by NSF Grant CCF-1527809.}
\and Wei Liu
\thanks{Tencent AI Lab, Shenzhen, China. Email: wliu@ee.columbia.edu.}
}

\begin{document}

\maketitle


\begin{abstract}
In this paper we study stochastic quasi-Newton methods for nonconvex stochastic optimization, where we assume that noisy information about the gradients of the objective function is available via a stochastic first-order oracle ($\SFO$). We propose a general framework for such methods, for which we prove almost sure convergence to stationary points and analyze its worst-case iteration complexity. When a randomly chosen iterate is returned as the output of such an algorithm, we prove that in the worst-case, the $\SFO$-calls complexity is $O(\epsilon^{-2})$ to ensure that the expectation of the squared norm of the gradient is smaller than the given accuracy tolerance $\epsilon$. We also propose a specific algorithm, namely a stochastic damped L-BFGS (SdLBFGS) method, that falls under the proposed framework. {Moreover, we incorporate the SVRG variance reduction technique into the proposed SdLBFGS method, and analyze its $\SFO$-calls complexity. Numerical results on a nonconvex binary classification problem using SVM, and a multiclass classification problem using neural networks are reported.}

\vspace{0.8cm}

\noindent {\bf Keywords:} {Nonconvex Stochastic Optimization, Stochastic Approximation, Quasi-Newton Method, Damped L-BFGS Method, Variance Reduction}

\vspace{0.5cm}

\noindent {\bf Mathematics Subject Classification 2010:} 90C15; 90C30; 62L20; 90C60

\end{abstract}

\section{Introduction}\label{sec:intro}

In this paper, we consider the following stochastic optimization problem:
\begin{equation}\label{orig-prob}
\min_{x\in\R^n}\quad  f(x) = \E[F(x,\xi)],
\end{equation}
where $F:\R^{n}\times \R^{d} \to \R$ is continuously differentiable and possibly nonconvex, $\xi\in\R^d$ denotes a random variable with distribution function $P$, and $\E[\cdot]$ denotes the expectation taken with respect to $\xi$.
In many cases the function $F(\cdot,\xi)$ is not given explicitly and/or the distribution function $P$ is unknown, or the function values and gradients of $f$ cannot be easily obtained and only noisy information about the gradient of $f$ is available. In this paper we assume that noisy gradients of $f$ can be obtained via calls to a {\it stochastic first-order oracle} ($\SFO$). Problem \eqref{orig-prob} arises in many applications in statistics and machine learning \cite{mbps09,ShaiBen14}, mixed logit modeling problems in economics and transportation \cite{bbt00,bct06,hg03} as well as many other areas. {A special case of \eqref{orig-prob} that arises frequently in machine learning is the empirical risk minimization problem
\be\label{sum-f-i}
\min_{x\in\R^n}\quad f(x)=\frac{1}{T}\sum_{i=1}^T f_i(x),
\ee
where $f_i:\R^{n}\to \R$ is the loss function corresponds to the $i$-th sample data, and $T$ denotes the number of sample data and is assumed to be extremely large.
}

The idea of employing stochastic approximation (SA) to solve stochastic programming problems can be traced back to the seminal work by Robbins and Monro \cite{rm51}. The classical SA method, also referred to as stochastic gradient descent (SGD), mimics the steepest gradient descent method, i.e., it updates iterate $x_k$ via
$x_{k+1} = x_k - \alpha_k g_k$,
where the stochastic gradient $g_k$ is an unbiased estimate of the gradient $\nabla f(x_k)$ of $f$ at $x_k$, and $\alpha_k$ is the stepsize. The SA method has been studied extensively in \cite{C54,e83,g78,p90,pj92,rs86,s58}, where the main focus has been the convergence of SA in different settings. Recently, there has been a lot of interest in analyzing
the worst-case complexity of SA methods, stimulated by the complexity theory developed by Nesterov for first-order methods for solving convex optimization problems \cite{Nesterov-1983,NesterovConvexBook2004}.
Nemirovski \etal \cite{njls09} proposed a mirror descent SA method for solving the convex stochastic programming problem $x^*:=\argmin \{ f(x)\mid x\in X\}$, where $f$ is nonsmooth and convex and $X$ is a convex set, and proved that for any given $\epsilon>0$, the method needs $O(\epsilon^{-2})$ iterations to obtain an $\bar{x}$ such that $\E[f(\bar{x})-f(x^*)]\le\epsilon$. Other SA methods with provable complexities for solving convex stochastic optimization problems have also been studied in \cite{gl12,jntv05,jrt08,l12,lns12,bach-nips-2012,bach-nips-2013,SAGA-2014,Bach-sgd-2014,shamir-zhang-icml-2013,xiao-zhang-siopt-2014}.

Recently there has been a lot of interest in SA methods for stochastic optimization problem \eqref{orig-prob} in which $f$ is a nonconvex function. In \cite{B98}, an SA method to minimize a general cost function was proposed by Bottou and proved to be convergent to stationary points. Ghadimi and Lan \cite{gl13} proposed a randomized stochastic gradient (RSG) method that returns an iterate from a randomly chosen iteration as an approximate solution. It is shown in \cite{gl13} that to return a solution $\bar{x}$ such that $\E[\|\nabla f(\bar{x})\|^2]\le\epsilon$, where $\|\cdot\|$ denotes the Euclidean norm, the total number of $\SFO$-calls needed by RSG is $O(\epsilon^{-2})$. Ghadimi and Lan \cite{gl132} also studied an accelerated SA method for solving \eqref{orig-prob} based on Nesterov's accelerated gradient method \cite{Nesterov-1983,NesterovConvexBook2004}, which improved the $\SFO$-call complexity for convex cases from $O(\epsilon^{-2})$ to $O(\epsilon^{-4/3})$. In \cite{glz13}, Ghadimi, Lan and Zhang proposed a mini-batch SA method {for solving} problems in which the objective function is a composition of a nonconvex smooth function $f$ and a convex nonsmooth function,
and analyzed its worst-case $\SFO$-call complexity. In \cite{dl13}, a method that incorporates a block-coordinate decomposition scheme into stochastic mirror-descent methodology, was proposed by Dang and Lan for a nonconvex stochastic optimization problem $x^*=\argmin\{f(x):x\in X\}$ in which the convex set $X$ has a block structure. More recently, Wang, Ma and Yuan \cite{WangMaYuan13} proposed a penalty method for nonconvex stochastic optimization problems with nonconvex constraints, and analyzed its $\SFO$-call complexity.

In this paper, we study stochastic quasi-Newton (SQN) methods for solving the nonconvex stochastic optimization problem \eqref{orig-prob}.
In the deterministic optimization setting, quasi-Newton methods are more robust and achieve higher accuracy than gradient methods, because {they use approximate second-order derivative information}. Quasi-Newton methods usually employ the following updates for solving \eqref{orig-prob}:
\be\label{quasi}
x_{k+1} = x_k - \alpha_k B_k^{-1}\nabla f(x_k),\quad \mbox{or}\quad x_{k+1} = x_k - \alpha_k H_k\nabla f(x_k)
\ee
where $B_k$ is an approximation to the Hessian matrix $\nabla^2 f(x_k)$ at $x_k$, or $H_k$ is an approximation to $[\nabla^2 f(x_k)]^{-1}$. The most widely-used quasi-Newton method, the BFGS method \cite{broyden-bfgs-1970,Fletcher-bfgs-1970,Goldfarb-bfgs-1970,Shanno-bfgs-1970} updates $B_k$ via
\be\label{bfgs}
 B_{k} =  B_{k-1} + \frac{y_{k-1}y_{k-1}^\top }{s_{k-1}^\top y_{k-1}} - \frac{B_{k-1}s_{k-1}s_{k-1}^\top B_{k-1}}{s_{k-1}^\top B_{k-1}s_{k-1}},
\ee
where $s_{k-1} := x_{k}-x_{k-1}$ and $y_{k-1}:=\nabla f(x_{k}) - \nabla f(x_{k-1})$.
By using the Sherman-Morrison-Woodbury formula, it is easy to derive that the equivalent update to $H_k=B_k^{-1}$ is
\be\label{bfgs1}
H_{k} = (I-\rho_{k-1} s_{k-1} {y}_{k-1}^\top) H_{k-1} (I -\rho_{k-1} {y}_{k-1} s_{k-1}^\top) + \rho_{k-1} s_{k-1} s_{k-1}^\top,
\ee
where $\rho_{k-1}:=1/(s_{k-1}^\top y_{k-1})$.
For stochastic optimization, there has been some work in designing stochastic quasi-Newton methods that update the iterates via \eqref{quasi} using the stochastic gradient $g_k$ in place of $\nabla f(x_k)$.
Specific examples include the following. The adaptive subgradient (AdaGrad) method proposed by Duchi, Hazan and Singer \cite{DHS2011}, which takes $B_k$ to be a diagonal matrix that estimates the diagonal of the square root of the uncentered covariance matrix of the gradients, has been proven to be quite efficient in practice. In \cite{BBG-09}, Bordes, Bottou and Gallinari studied SGD with a diagonal rescaling matrix based on the secant condition associated with quasi-Newton methods. Roux and Fitzgibbon \cite{RF-2010} discussed the necessity of including both Hessian and covariance matrix information in a stochastic Newton type method.
Byrd \etal \cite{BCNN} proposed a quasi-Newton method that uses the sample average approximation (SAA) approach to estimate Hessian-vector multiplications. In \cite{BHNS-14}, Byrd \etal proposed a stochastic limited-memory BFGS (L-BFGS) \cite{Liu-Nocedal-89} method based on SA, and proved its convergence for strongly convex problems. Stochastic BFGS and L-BFGS methods were also studied for online convex optimization by Schraudolph, Yu and G\"{u}nter in \cite{SchYuGue07}. For strongly convex problems,
Mokhtari and Ribeiro proposed a regularized stochastic BFGS method (RES) and analyzed its convergence in \cite{mr10} and studied an online L-BFGS method in \cite{mr14-2}. Recently, Moritz, Nishihara and Jordan \cite{MNJ2015} proposed a linearly convergent method that integrates the L-BFGS method in \cite{BHNS-14} with the variance reduction technique (SVRG) proposed by Johnson and Zhang in \cite{SPVR} to alleviate the effect of noisy gradients. A related method that incorporates SVRG into a quasi-Newton method was studied by Lucchi, McWilliams and Hofmann in \cite{LMH2015}. In \cite{Gower-icml-2016}, Gower, Goldfarb and Richt\'arik proposed a variance reduced block L-BFGS method that converges linearly for convex functions.
It should be noted that all of the above stochastic quasi-Newton methods are designed for solving convex or even strongly convex problems.

{\bf Challenges.}
The key challenge in designing stochastic quasi-Newton methods for nonconvex problem lies in the difficulty in preserving the positive-definiteness of $B_k$ (and $H_k$), due to the non-convexity of the problem and the presence of noise in estimating the gradient.
It is known that the BFGS update \eqref{bfgs} preserves the positive-definiteness of $B_k$ as long as the curvature condition
\be\label{curvature} s_{k-1}^\top y_{k-1} > 0 \ee
holds, which can be guaranteed for strongly convex problem. For nonconvex problem, the curvature condition \eqref{curvature} can be satisfied by performing a line search. However, doing this is no longer feasible for \eqref{orig-prob} in the stochastic setting, because exact function values and gradient information are not available. As a result, an important issue in designing stochastic quasi-Newton methods for nonconvex problems is how to preserve the positive-definiteness of $B_k$ (or $H_k$) without line search.

{\bf Our contributions.} Our contributions (and where they appear) in this paper are as follows.
\begin{itemize}
\item[1.] We propose a general framework for stochastic quasi-Newton methods (SQN) for solving the nonconvex stochastic optimization problem \eqref{orig-prob}, {and prove its almost sure convergence to a stationary point when the step size $\alpha_k$ is diminishing. We also prove that the number of iterations $N$ needed to obtain $\frac{1}{N}\sum_{k=1}^N\E[\|\nabla f(x_k)\|^2]\le \epsilon$, is $N=O(\epsilon^{-\frac{1}{1-\beta}})$, for $\alpha_k$ chosen proportional to $k^{-\beta}$, where $\beta\in(0.5,1)$ is a constant.} (See Section \ref{sec:RSSA})
\item[2.] When a randomly chosen iterate $x_R$ is returned as the output of SQN, we prove that the worst-case $\SFO$-calls complexity is $O(\epsilon^{-2})$ to guarantee $\E[\|\nabla f(x_R)\|^2]\le \epsilon$. (See Section \ref{sec:RSSA-2})
\item[3.] We propose a stochastic damped L-BFGS (SdLBFGS) method that fits into the proposed framework. This method adaptively generates a positive definite matrix $H_k$ that approximates the inverse Hessian matrix at the current iterate $x_k$. Convergence and complexity results for this method are provided. Moreover, our method does not generate $H_k$ explicitly, and only its multiplication with vectors is computed directly. (See Section \ref{sec:LBFGS})
\item[4.] {Motivated by the recent advance of SVRG for nonconvex minimization \cite{SVR-nonconvex,AllenHazan2016-nonconvex}, we propose a variance reduced variant of SdLBFGS and analyze its $\SFO$-calls complexity.} (See Section \ref{sec:sdlbfgs-vr})
\end{itemize}

\section{A general framework for stochastic quasi-Newton methods for nonconvex optimization}\label{sec:RSSA}

In this section, we study SQN methods for the (possibly nonconvex) stochastic optimization problem \eqref{orig-prob}.
We assume that an $\SFO$ outputs a stochastic gradient $g(x,\xi)$ of $f$ for a given $x$, where $\xi$ is a random variable whose distribution is supported on $\Xi\subseteq \R^d$. Here we assume that $\Xi$ does not depend on $x$.

We now give some assumptions that are required throughout this paper.
\begin{itemize}
\item[\textbf{AS.1}]\quad $f: \R^n\to \R$ is continuously differentiable. $f(x)$ is lower bounded by a real number $f^{low}$ for any $x\in\R^n$. $\nabla f$ is globally Lipschitz continuous with Lipschitz constant $L$; namely for any $x,y\in\R^n$,
    \[
    \|\nabla f(x) - \nabla f(y)\| \le L\|x-y\|.
    \]
\item[\textbf{AS.2}]\quad For any iteration $k$, we have
\begin{align}
  a)\quad & \E_{\xi_k}\left[g(x_k,\xi_k)\right]=\nabla f(x_k),\\
  b)\quad & \E_{\xi_k}\left[\|g(x_k,\xi_k)-\nabla f(x_k)\|^2\right]\le \sigma^2, \label{c}
\end{align}
where $\sigma>0$ is the noise level of the gradient estimation, and $\xi_k$, $k=1,2,\ldots$, are independent samples, and for a given $k$ the random variable $\xi_k$ is independent of $\{x_j\}_{j=1}^k$.
\end{itemize}
\begin{rem}
Note that the stochastic BFGS methods studied in \cite{mr10,BHNS-14,mr14-2} require that the noisy gradient is bounded, i.e.,
\be\label{grad-bound}\E_{\xi_k}\left[\|g(x_k,\xi_k)\|^2\right]\le M_g, \ee
where $M_g>0$ is a constant. Our assumption \eqref{c} is weaker than \eqref{grad-bound}.
\end{rem}

Analogous to deterministic quasi-Newton methods, our SQN method takes steps
\be \label{x-k-1}
x_{k+1} = x_k - \alpha_k H_k g_k,
\ee
where $g_k$ is defined as a mini-batch estimate of the gradient:
\begin{equation}\label{G-k}
    g_k = \frac{1}{m_k}\sum_{i=1}^{m_k}g(x_k,\xi_{k,i}),
\end{equation}
and $\xi_{k,i}$ denotes the random variable generated by the $i$-th sampling in the $k$-th iteration. From \textbf{AS.2} we can see that $g_k$ has the following properties:
\be\label{Exp-G-k}
\E[g_k|x_k] = \nabla f(x_k), \quad \E[\|g_k-\nabla f(x_k)\|^2|x_k]\le \frac{\sigma^2}{m_k}.
\ee

\begin{itemize}
\item[{\bf AS.3}]\quad There exist two positive constants $C_{l}, C_{u}$ such that
    \[
    \underline{\kappa} I  \preceq  H_k  \preceq  \bar{\kappa} I, \quad \mbox{ for all } k,
    \]
    where the notation $A\succeq B$ with $A,B\in\R^{n\times n}$ means that $A-B$ is positive semidefinite.
\end{itemize}

We denote by $\xi_k = (\xi_{k,1},\ldots,\xi_{k,m_k})$, the random samplings in the $k$-th iteration, and denote by $\xi_{[k]}:=(\xi_1,\ldots,\xi_k)$, the random samplings in the first $k$ iterations.
Since $H_k$ is generated iteratively based on historical gradient information by a random process, we make the following assumption on $H_k (k\ge 2)$ to control the randomness (note that $H_1$ is given in the initialization step).
\begin{itemize}
\item[{\bf AS.4}]\quad For any $k\ge 2$, the random variable $H_k$ depends only on $\xi_{[k-1]}$.
\end{itemize}
It then follows directly from {\bf AS.4} and \eqref{Exp-G-k} that
\be\label{as-4-equality}
\E[H_k g_k|\xi_{[k-1]}] = H_k\nabla f(x_k),
\ee
where the expectation is taken with respect to $\xi_k$ generated in the computation of $g_k$.

We will not specify how to compute $H_k$ until Section \ref{sec:LBFGS}, where a specific updating scheme for $H_k$ satisfying both assumptions {\bf AS.3} and {\bf AS.4} will be proposed.

We now present our SQN method for solving \eqref{orig-prob} as Algorithm \ref{sso-uncons}.
\begin{algorithm}[ht]
\caption{{\bf SQN: Stochastic quasi-Newton method for nonconvex stochastic optimization}}
\label{sso-uncons}
\begin{algorithmic}[1]
\REQUIRE {Given $x_1\in\R^n$, a positive definite matrix $H_1\in\R^{n\times n}$, batch sizes $\{m_k\}_{k\ge 1}$,and stepsizes $\{\alpha_k\}_{k\ge 1}$
}
\FOR {$k=1,2,\ldots$}
    \STATE    Calculate
        $
            g_k = \frac{1}{m_k}\sum_{i=1}^{m_k}g(x_k,\xi_{k,i}).
        $
    \STATE Generate a positive definite Hessian inverse approximation $H_{k}$.
    \STATE Calculate
        $
        x_{k+1} = x_k - \alpha_k H_k g_k.
        $
\ENDFOR
\end{algorithmic}
\end{algorithm}

\subsection{Convergence and complexity of SQN with diminishing step size}\label{sec:RSSA-1}
In this subsection, we analyze the convergence and complexity of SQN under the condition that the step size $\alpha_k$ in \eqref{x-k-1} is diminishing. Specifically, in this subsection we assume $\alpha_k$ satisfies the following condition:
\be\label{alpha-inf}
   \sum_{k=1}^{+\infty}\alpha_k = +\infty, \qquad \sum_{k=1}^{+\infty}\alpha_k^2 < +\infty,
\ee
which is a standard assumption in stochastic approximation algorithms (see, e.g., \cite{BHNS-14,mr14-2,njls09}). One very simple choice of $\alpha_k$ that satisfies \eqref{alpha-inf} is $\alpha_k=O(1/k)$.

The following lemma shows that a descent property in terms of the expected objective value holds for SQN. Our analysis is similar to analyses that have been used in \cite{B98,mr10}.
\begin{lem}\label{lem3.1}
Suppose that $\{x_k\}$ is generated by SQN and assumptions {\bf AS.1-4} hold. Further assume that \eqref{alpha-inf} holds, and $\alpha_k \leq \frac{\underline{\kappa}}{L\bar{\kappa}^2}$ for all $k$. (Note that this can be satisfied if $\alpha_k$ is non-increasing and the initial step size $\alpha_1 \leq \frac{\underline{\kappa}}{L\bar{\kappa}^2}$). Then the following inequality holds
\be \label{exp-red}
\E[f(x_{k+1})|x_k]\le  f(x_k) - \frac{1}{2}\alpha_k \underline{\kappa}  \|\nabla f(x_k)\|^2  + \frac{L\sigma^2 \bar{\kappa}^2}{2m_k}\alpha_k^2, \quad \forall k\geq 1,
\ee
where the conditional expectation is taken with respect to $\xi_k$.
\end{lem}

\begin{proof}
Define $\delta_k = g_k - \nabla f(x_k)$. From \eqref{x-k-1}, and assumptions {\bf AS.1} and {\bf AS.3}, we have
\begin{align}
f(x_{k+1})
 \le & \,f(x_k) + \langle \nabla f(x_k), x_{k+1} - x_k \rangle + \frac{L}{2}\|x_{k+1}-x_{k}\|^2  \notag \\
 = & \,f(x_k) -\alpha_k\langle \nabla f(x_k),  H_kg_k \rangle + \frac{L}{2}\alpha_k^2\| H_kg_k\|^2  \notag \\
 \leq &\, f(x_k)  -\alpha_k\langle \nabla f(x_k),  H_k\nabla f(x_k)\rangle - \alpha_k\langle \nabla f(x_k),  H_k\delta_k\rangle + \frac{L}{2}\alpha_k^2 \bar{\kappa}^2\|g_k\|^2. \label{red-ori}
\end{align}
Taking expectation with respect to $\xi_k$ on both sides of \eqref{red-ori} conditioned on $x_k$, we obtain,
\be \label{red}
\E[f(x_{k+1})|x_k]\le  f(x_k) -\alpha_k\langle \nabla f(x_k), H_k\nabla f(x_k) \rangle + \frac{L}{2}\alpha_k^2 \bar{\kappa}^2\E[\|g_k\|^2|x_k],
\ee
where we used \eqref{as-4-equality} and the fact that $\E[\delta_k|x_k]=0$.
From \eqref{Exp-G-k} and $\E[\delta_k|x_k]=0$, it follows that
\begin{align}
\E[\|g_k\|^2|x_k]& =  \E[\|g_k-\nabla f(x_k) + \nabla f(x_k)\|^2|x_k] \notag \\
                 & =  \E[\|\nabla f(x_k)\|^2|x_k] + \E[\|g_k - \nabla f(x_k)\|^2|x_k] +  2\E[\langle\delta_k,\nabla f(x_k)\rangle |x_k] \notag \\
                 & = \|\nabla f(x_k)\|^2 + \E[\|g_k - \nabla f(x_k)\|^2|x_k]  \le \|\nabla f(x_k)\|^2 +  \sigma^2/m_k, \label{exp-G2}
\end{align}
which together with \eqref{red} and {\bf AS.3} yields that
\be\label{lemma2.1-proof-1}
\E[f(x_{k+1})|x_k]\le  f(x_k) - \left(\alpha_k \underline{\kappa} - \frac{L}{2}\alpha_k^2 \bar{\kappa}^2\right) \|\nabla f(x_k)\|^2  + \frac{L\sigma^2 \bar{\kappa}^2}{2m_k}\alpha_k^2.
\ee
Then \eqref{lemma2.1-proof-1} combined with the assumption $\alpha_k \leq \frac{\underline{\kappa}}{LC_u^2}$ implies \eqref{exp-red}.
\end{proof}

Before proceeding further, we introduce the definition of a {\it supermartingale} (see \cite{Durrett-10} for more details).
\begin{defi}\label{def2.1}
Let $\{\mathcal{F}_k\}$ be an increasing sequence of $\sigma$-algebras. If $\{X_k\}$ is a stochastic process satisfying
(i) $\E[|X_k|]<\infty$; (ii) $X_k\in\mathcal{F}_k$ for all $k$; and (iii) $\E[X_{k+1}|\mathcal{F}_k]\le X_k$ for all $k$,
then $\{X_k\}$ is called a supermartingale.
\end{defi}

\begin{prop}[see, e.g., Theorem 5.2.9 in \cite{Durrett-10}]\label{prop2.1}
If $\{X_k\}$ is a nonnegative supermartingale, then $\lim_{k\to\infty}X_k\to X$ almost surely and $\E[X]\le\E[X_0]$.
\end{prop}

We are now ready to give convergence results for SQN (Algorithm \ref{sso-uncons}).
\begin{thm}\label{thm3.1}
Suppose that assumptions {\bf AS.1-4} hold for $\{x_k\}$ generated by SQN with batch size $m_k = {m}$ for all $k$. If the stepsize $\alpha_k$ satisfies \eqref{alpha-inf} and $\alpha_k \leq \frac{\underline{\kappa}}{L\bar{\kappa}^2}$ for all $k$, then it holds that
\be \label{liminf_g}
\liminf_{k\to\infty}\, \|\nabla f(x_k)\| = 0, \quad \mbox{ with probability 1}.
\ee
Moreover, there exists a positive constant $M_f$ such that
\be\label{M_f} \E[f(x_k)] \leq M_f, \quad \forall k. \ee
\end{thm}
\begin{proof}
Define $\beta_k := \frac{\alpha_k \underline{\kappa}}{2}\|\nabla f(x_k)\|^2$ and
$\gamma_k := f(x_k)+ \frac{L\sigma^2 \bar{\kappa}^2}{2m}\sum_{i=k}^\infty \alpha_i^2$.
Let $\C{F}_k$ be the $\sigma$-algebra measuring $\beta_k$, $\gamma_k$ and $x_k$. From \eqref{exp-red} we know that for any $k$, it holds that
\begin{align}
\E[\gamma_{k+1}|\C{F}_k] & = \E[f(x_{k+1})|\C{F}_k] + \frac{L\sigma^2 \bar{\kappa}^2}{2m}\sum_{i=k+1}^\infty \alpha_i^2 \notag \\
                         & \le f(x_k) -\frac{\alpha_k \underline{\kappa}}{2}\|\nabla f(x_k)\|^2 + \frac{L\sigma^2 \bar{\kappa}^2}{2m}\sum_{i=k}^\infty \alpha_i^2  =\gamma_k - \beta_k, \label{diff}
\end{align}
which implies that
$
\E[\gamma_{k+1}-f^{low}|\C{F}_k] \le \gamma_k - f^{low} - \beta_k.
$
Since $\beta_k\ge0$, we have $0\le\E[\gamma_k-f^{low}]\le \gamma_1-f^{low}<+\infty$, which implies \eqref{M_f}. According to Definition \ref{def2.1}, $\{\gamma_k-f^{low}\}$ is a supermartingale. Therefore, Proposition \ref{prop2.1} shows that there exists a $\gamma$ such that $\lim_{k\to\infty}\gamma_k=\gamma$ with probability 1, and $\E[\gamma]\le \E[\gamma_1]$.
Note that from \eqref{diff} we have $\E[\beta_k] \le \E[\gamma_k]-\E[\gamma_{k+1}]$. Thus,
\[
\E[\sum_{k=1}^{\infty}\beta_k] \le \sum_{k=1}^{\infty}(\E[\gamma_{k}]-\E[\gamma_{k+1}]) < +\infty,
\]
which further yields that
\be\label{alpha-g}
\sum_{k=1}^{\infty}\beta_k = \frac{\underline{\kappa}}{2}\sum_{k=1}^\infty \alpha_k\|\nabla f(x_k)\|^2 < +\infty \quad \mbox{  with probability 1}.
\ee
Since $\sum_{k=1}^{\infty}\alpha_k =+\infty$, it follows that \eqref{liminf_g} holds.
\end{proof}

Under the assumption \eqref{grad-bound} used in \cite{mr10,BHNS-14,mr14-2}, we now prove a stronger convergence result showing that any limit point of $\{x_k\}$ generated by SQN is a stationary point of \eqref{orig-prob} with probability 1.
\begin{thm}\label{thm2.2}
Assume the same assumptions hold as in Theorem \ref{thm3.1}, and that \eqref{grad-bound} holds. Then
\be \label{lim_g}
\lim_{k\to\infty}\, \|\nabla f(x_k)\| = 0, \quad \mbox{ with probability 1}.
\ee
\end{thm}
\begin{proof}
For any given $\epsilon>0$, according to \eqref{liminf_g}, there exist infinitely many iterates $x_k$ such that $\|\nabla f(x_k)\|< \epsilon$. 
Then if \eqref{lim_g} does not hold, there must exist two infinite sequences of indices $\{m_i\}$, $\{n_i\}$ with $n_i > m_i$, such that for $i=1,2,\ldots,$
\be\label{x_m_n}
\|\nabla f(x_{m_i})\| \ge 2\epsilon, \quad \|\nabla f(x_{n_i})\|<\epsilon, \quad  \|\nabla f(x_k)\|\ge \epsilon, \quad k=m_i+1,\ldots,n_i-1.
\ee
Then from \eqref{alpha-g} it follows that
\begin{align*}
+\infty >\, \sum_{k=1}^{+\infty} \alpha_k\|\nabla f(x_k)\|^2
        \ge\, \sum_{i=1}^{+\infty}\sum_{k=m_i}^{n_i-1} \alpha_k \|\nabla f(x_k)\|^2
        \ge\,  \epsilon^2 \sum_{i=1}^{+\infty}\sum_{k=m_i}^{n_i-1} \alpha_k, \quad \mbox{with probability 1,}
\end{align*}
which implies that
\be\label{m_n_i}
\sum_{k=m_i}^{n_i-1} \alpha_k \to 0, \mbox{ with probability 1,
}\quad\mbox{ as } i\to +\infty.
\ee
According to \eqref{exp-G2}, we have that
\begin{equation}\label{thm2.2-proof-1}
\E[\|x_{k+1}-x_k\||x_k] = \alpha_k \E[\|H_k g_k\||x_k] \le \alpha_k \bar{\kappa}\E[\|g_k\||x_k] \le \alpha_k \bar{\kappa} (\E[\|g_k\|^2|x_k])^{\frac{1}{2}} \le \alpha_k \bar{\kappa} (M_g/m)^{\frac{1}{2}},
\end{equation}
where the last inequality is due to \eqref{grad-bound} and the convexity of $\|\cdot\|^2$.
Then it follows from \eqref{thm2.2-proof-1} that
\[
\E[\|x_{n_i} - x_{m_i}\|] \le \bar{\kappa} (M_g/m)^{\frac{1}{2}}\sum_{k=m_i}^{n_i-1} \alpha_k,
\]
which together with \eqref{m_n_i} implies that $\|x_{n_i} - x_{m_i}\|\to 0$ with probability 1, as $i\to +\infty$. Hence, from the Lipschitz continuity of $\nabla f$, it follows that $\|\nabla f(x_{n_i})-\nabla f(x_{m_i})\|\to 0$ with probability 1 as $i\to+\infty$. However, this contradicts \eqref{x_m_n}. Therefore, the assumption that \eqref{lim_g} does not hold is not true.
\end{proof}

\begin{rem}
Note that our result in Theorem \ref{thm2.2} is stronger than the ones given in existing works such as \cite{mr10} and \cite{mr14-2}.
Moreover, although Bottou \cite{B98} also proves that the SA method for nonconvex stochastic optimization with diminishing stepsize is almost surely convergent to stationary point, our analysis requires weaker assumptions. For example, \cite{B98} assumes that the objective function is three times continuously differentiable, while our analysis does not require this. Furthermore, we are able to analyze the iteration complexity of SQN, for a specifically chosen step size $\alpha_k$ (see Theorem \ref{thm2.3} below), which is not provided in \cite{B98}.
\end{rem}

We now analyze the iteration complexity of SQN. 

\begin{thm}\label{thm2.3}
Suppose that assumptions {\bf AS.1-4} hold for $\{x_k\}$ generated by SQN with batch size $m_k = {m}$ for all $k$. We also assume that $\alpha_k$ is specifically chosen as
\be\label{alpha-spec}
\alpha_k = \frac{\underline{\kappa}}{L \bar{\kappa}^2}k^{-\beta}
\ee
with $\beta\in(0.5,1)$. Note that this choice satisfies \eqref{alpha-inf} and $\alpha_k \leq \frac{\underline{\kappa}}{L \bar{\kappa}^2}$ for all $k$.
Then
\be\label{conv-rate}
\frac{1}{N}\sum_{k=1}^N \E[\|\nabla f(x_k)\|^2] \le  \frac{2L(M_f-f^{low}) \bar{\kappa}^2}{\underline{\kappa}^2}N^{\beta-1} + \frac{\sigma^2 }{(1-\beta)m}(N^{-\beta}-N^{-1}),
\ee
where $N$ denotes the iteration number.
Moreover, for a given $\epsilon\in(0,1)$, to guarantee that $\frac{1}{N}\sum_{k=1}^N \E[\|\nabla f(x_k)\|^2] < \epsilon$, the number of iterations $N$ needed is at most
$O\left(\epsilon^{-\frac{1}{1-\beta}}\right)$.
\end{thm}

\begin{proof}
Taking expectation on both sides of \eqref{exp-red} and summing over $k=1,\ldots,N$ yields
\begin{align*}
\frac{1}{2}\underline{\kappa}\sum_{k=1}^N \E[\|\nabla f(x_k)\|^2] & \le \sum_{k=1}^N\frac{1}{\alpha_k}(\E[f(x_k)]-\E[f(x_{k+1})]) + \frac{L\sigma^2 \bar{\kappa}^2}{2m}\sum_{k=1}^N \alpha_k \\
& = \frac{1}{\alpha_1}f(x_1) + \sum_{k=2}^N \left(\frac{1}{ \alpha_k}-\frac{1}{\alpha_{k-1}}\right)\E[f(x_k)] - \frac{\E[f(x_{N+1})]}{\alpha_N} + \frac{L\sigma^2 \bar{\kappa}^2}{2m}\sum_{k=1}^N \alpha_k \\
& \le \frac{M_f}{\alpha_1} + M_f \sum_{k=2}^N \left(\frac{1}{ \alpha_k}-\frac{1}{\alpha_{k-1}}\right) - \frac{f^{low}}{\alpha_N} + \frac{L\sigma^2 \bar{\kappa}^2}{2m}\sum_{k=1}^N \alpha_k \\
& =\frac{M_f-f^{low}}{\alpha_N} + \frac{L\sigma^2 \bar{\kappa}^2}{2m}\sum_{k=1}^N \alpha_k  \\
& \le \frac{L(M_f-f^{low}) \bar{\kappa}^2}{\underline{\kappa}}N^\beta + \frac{\sigma^2 \underline{\kappa}}{2(1-\beta)m}(N^{1-\beta}-1),
\end{align*}
which results in \eqref{conv-rate},
where the second inequality is due to \eqref{M_f} and the last inequality is due to \eqref{alpha-spec}. Then for a given $\epsilon>0$,  to guarantee that $\frac{1}{N}\sum_{k=1}^N \E[\|\nabla f(x_k)\|^2] \le \epsilon$, it suffices to require that
\[
 \frac{2L(M_f-f^{low}) \bar{\kappa}^2}{\underline{\kappa}^2}N^{\beta-1} + \frac{\sigma^2 }{(1-\beta)m}(N^{-\beta}-N^{-1})< \epsilon.
\]
Since $\beta\in(0.5,1)$, it follows that the number of iterations $N$ needed is at most $O(\epsilon^{-\frac{1}{1-\beta}})$.
\end{proof}

\begin{rem}
Note that Theorem \ref{thm2.3} also provides iteration complexity analysis for the classic SGD method, which can be regarded as a special case of SQN with $H_k = I$. To the best of our knowledge, our complexity result in Theorem \ref{thm2.3} is new for both SGD and stochastic quasi-Newton methods.
\end{rem}

\subsection{Complexity of SQN with random output and constant step size}\label{sec:RSSA-2}

We analyze the $\SFO$-calls complexity of SQN when the output is randomly chosen from $\{x_i\}_{i=1}^N$, where $N$ is the maximum iteration number. Our results in this subsection are motivated by the randomized stochastic gradient (RSG) method proposed by Ghadimi and Lan \cite{gl13}. RSG runs SGD 
for $R$ iterations, where $R$ is a randomly chosen integer from $\{1,\ldots,N\}$ with a specifically defined probability mass function $P_R$. In \cite{gl13} it is proved that under certain conditions on the step size and $P_R$, $O(1/\epsilon^2)$ $\SFO$-calls are needed by SGD to guarantee $\E[\|\nabla f(x_R)\|^2] \leq \epsilon$. We show below that under similar conditions, the same complexity holds for our SQN.

\begin{thm} \label{thm3.2}
Suppose that assumptions {\bf AS.1-4} hold, and that $\alpha_k$ in SQN (Algorithm \ref{sso-uncons}) is chosen such that $0<\alpha_k\le 2\underline{\kappa}/(L\bar{\kappa}^2)$ for all $k$ with $\alpha_k<2\underline{\kappa}/(L\bar{\kappa}^2)$ for at least one $k$. Moreover, for a given integer $N$, let $R$ be a random variable with the probability mass function
\be \label{P_R}
P_R(k):=\mathrm{Prob}\{R=k\}= \frac{ \alpha_k \underline{\kappa} - \alpha_k^2 L \bar{\kappa}^2/2}{\sum_{k=1}^N
\left ( \alpha_k \underline{\kappa} - \alpha_k^2 L \bar{\kappa}^2/2\right)}, \quad k=1,\ldots,N.
\ee
Then we have
\be \label{exp-gra}
\E[\|\nabla f(x_R)\|^2]\le \frac{D_f + (\sigma^2 L \bar{\kappa}^2)/2\sum_{k=1}^N(\alpha_k^2/m_k)}{\sum_{k=1}^N
\left ( \alpha_k \underline{\kappa} - \alpha_k^2 L \bar{\kappa}^2/2\right)},
\ee
where $D_f:=f(x_1)-f^{low}$ and the expectation is taken with respect to $R$ and $\xi_{[N]}$. Moreover, if we choose $\alpha_k = \underline{\kappa}/(LC_{u}^2)$ and $m_k = m$ for all $k=1,\ldots,N$, then \eqref{exp-gra} reduces to
\be \label{exp-g}
\E[\|\nabla f(x_R)\|^2] \le \frac{2L \bar{\kappa}^2D_f}{N \underline{\kappa}^2} + \frac{\sigma^2}{{m}}.
\ee
\end{thm}
\begin{proof}
From \eqref{red-ori} it follows that
\begin{align*}
f(x_{k+1})  \le &\, f(x_k)  -\alpha_k\langle \nabla f(x_k),  H_k\,\nabla f(x_k)\rangle - \alpha_k\langle \nabla f(x_k),  H_k\,\delta_k\rangle + \\
&\, \frac{L}{2}\alpha_k^2\bar{\kappa}^2 [\|\nabla f(x_k)\|^2 + 2\langle \nabla f(x_k), \delta_k\rangle + \|\delta_k\|^2]\\
\le & \,f(x_k) - \left( \alpha_k \underline{\kappa} - \frac{1 }{2}\alpha_k^2 L \bar{\kappa}^2\right) \|\nabla f(x_k)\|^2 + \frac{1 }{2}\alpha_k^2 L \bar{\kappa}^2\|\delta_k\|^2 + \alpha_k^2 L\bar{\kappa}^2 \langle \nabla f(x_k),\delta_k\rangle \\
&\, - \alpha_k\langle \nabla f(x_k),  H_k\,\delta_k \rangle,
\end{align*}
where $\delta_k=g_k-\nabla f(x_k)$. Now summing $k=1,\ldots,N$ and noticing that $\alpha_k\le 2\underline{\kappa}/(L \bar{\kappa}^2)$, yields
\begin{align}
&\sum_{k=1}^N
\left ( \alpha_k \underline{\kappa} - \frac{L \bar{\kappa}^2}{2}\alpha_k^2\right) \|\nabla f(x_k)\|^2 \notag\\
\le & f(x_1) - f^{low} + \frac{L\bar{\kappa}^2}{2}\sum_{k=1}^N\alpha_k^2\|\delta_k\|^2 + \sum_{k=1}^N(\alpha_k^2 L\bar{\kappa}^2\langle \nabla f(x_k),\delta_k\rangle - \alpha_k\langle \nabla f(x_k), H_k\delta_k \rangle).  \label{sum-gra}
\end{align}
By {\bf AS.2} and {\bf AS.4} we have that
\[
\E_{\xi_k}[\langle \nabla f(x_k), \delta_k \rangle|\xi_{[k-1]}]=0, \quad \mbox{ and } \quad \E_{\xi_k}[\langle \nabla f(x_k),  H_k\,\delta_k \rangle|\xi_{[k-1]}]=0.
\]
Moreover, from \eqref{Exp-G-k} it follows that $\E_{\xi_k}[\|\delta_k\|^2|\xi_{[k-1]}]\le\sigma^2/m_k$. Therefore, taking the expectation on both sides of \eqref{sum-gra} with respect to $\xi_{[N]}$ yields
\be\label{int-exp}
\sum_{k=1}^N
 ( \alpha_k \underline{\kappa} - \alpha_k^2 L \bar{\kappa}^2/2)\E_{\xi_{[N]}}[\|\nabla f(x_k)\|^2] \le f(x_1) - f^{low} + \frac{L \bar{\kappa}^2\sigma^2}{2}\sum_{k=1}^N\frac{\alpha_k^2}{m_k}.
\ee
It follows from the definition of $P_R$ in \eqref{P_R} that
\be
\E[\|\nabla f(x_R)\|^2] = \E_{R,\xi_{[N]}}[\|\nabla f(x_R)\|^2] = \frac{\sum_{k=1}^N
\left ( \alpha_k \underline{\kappa} - \alpha_k^2 L \bar{\kappa}^2/2\right)\E_{\xi_{[N]}}[\|\nabla f(x_k)\|^2]}{\sum_{k=1}^N
\left ( \alpha_k \underline{\kappa} - \alpha_k^2 L \bar{\kappa}^2/2\right)},
\ee
which together with \eqref{int-exp} implies \eqref{exp-gra}.
\end{proof}

\begin{rem}
Note that in Theorem \ref{thm3.2}, $\alpha_k$'s are not required to be diminishing, and they can be constant as long as they are upper bounded by $2\underline{\kappa}/(L\bar{\kappa}^2)$.
\end{rem}

We now show that the $\SFO$ complexity of SQN with random output and constant step size is $O(\epsilon^{-2})$.

\begin{cor}\label{cor3.4}
Assume the conditions in Theorem \ref{thm3.2} hold, and $\alpha_k = \underline{\kappa}/(LC_{u}^2)$ and $m_k = m$ for all $k=1,\ldots,N$. Let $\bar{N}$ be the total number of $\SFO$-calls needed to calculate stochastic gradients $g_k$ in SQN (Algorithm \ref{sso-uncons}). For a given accuracy tolerance $\epsilon>0$, we assume that
\be \label{bar-N}
\bar{N}\ge\max\left\{\frac{C_1^2}{\epsilon^2} + \frac{4C_2}{\epsilon},\frac{\sigma^2}{L^2\tilde{D}}\right\}, \quad \mbox{ where } \quad C_1=\frac{4\sigma  \bar{\kappa}^2D_f}{\underline{\kappa}^2\sqrt{\tilde{D}}} + \sigma L\sqrt{\tilde{D}}, \quad C_2=\frac{4L \bar{\kappa}^2D_f}{\underline{\kappa}^2},
\ee
where $\tilde{D}$ is a problem-independent positive constant. Moreover, we assume that the batch size satisfies
\be \label{batch-size-m}
m_k = {m} := \left\lceil \min \left\{ \bar{N}, \max\left\{1, \frac{\sigma}{L}\sqrt{\frac{\bar{N}}{\tilde{D}}}\right\}\right\}\right\rceil.
\ee
Then it holds that $\E[\|\nabla f(x_R)\|^2]\le\epsilon$,
where the expectation is taken with respect to $R$ and $\xi_{[N]}$. 
\end{cor}

\begin{proof}
Note that the number of iterations of SQN is at most $N=\lceil \bar{N}/{m}\rceil$. Obviously, $N\ge \bar{N}/(2{m})$. From \eqref{exp-g} we have that
\begin{align}
\E[\|\nabla f(x_R)\|^2] & \le \frac{2L  \bar{\kappa}^2D_{f}}{N \underline{\kappa}^2} + \frac{\sigma^2}{ {m}}  \le  \frac{4L \bar{\kappa}^2 D_{f}}{\bar{N} \underline{\kappa}^2}{m} + \frac{\sigma^2}{ {m}} \notag \\
& \le  \frac{4L \bar{\kappa}^2 D_{f} }{\bar{N} \underline{\kappa}^2}\left(1+\frac{\sigma}{L}\sqrt{\frac{\bar{N}}{\tilde{D}}}\right) +  \max\left\{\frac{ \sigma^2}{\bar{N}}, \frac{ \sigma L\sqrt{\tilde{D}}}{\sqrt{\bar{N}}}\right\}. \label{exp-1}
\end{align}
\eqref{bar-N} indicates that
\[ \sqrt{\bar{N}} \ge\frac{\sqrt{C_1^2 + 4\epsilon C_2}}{\epsilon } \ge \frac{\sqrt{C_1^2 + 4\epsilon C_2} + C_1}{2\epsilon}. \]
\eqref{bar-N} also implies that $\sigma^2/\bar{N}\le \sigma L\sqrt{\tilde{D}}/\sqrt{\bar{N}}$. Then \eqref{exp-1} yields that
\[
\E[\|\nabla f(x_R)\|^2]  \le \frac{4L \bar{\kappa}^2 D_{f} }{\bar{N}\underline{\kappa}^2}\left(1+\frac{\sigma}{L}\sqrt{\frac{\bar{N}}{\tilde{D}}}\right) +  \frac{ \sigma L\sqrt{\tilde{D}}}{\sqrt{\bar{N}}}
 = \frac{C_1}{\sqrt{\bar{N}}} + \frac{C_2}{\bar{N}}
 \le \epsilon,
\]
which completes the proof.
\end{proof}

\begin{rem}\label{rem3.2}
In Corollary \ref{cor3.4} we did not consider the $\SFO$-calls that are involved in updating $H_k$ in line 3 of SQN. In the next section, we consider a specific updating scheme to generate $H_k$, and analyze the total $\SFO$-calls complexity of SQN including the generation of the $H_k$.
\end{rem}

\section{Stochastic damped L-BFGS method}\label{sec:LBFGS}

In this section, we propose a specific way, namely a damped L-BFGS method (SdLBFGS), to generate $H_k$ in SQN (Algorithm \ref{sso-uncons}) that satisfies assumptions {\bf AS.3} and {\bf AS.4}. We also provide an efficient way to compute $H_kg_k$ without generating $H_k$ explicitly.

Before doing this, we first describe a stochastic damped BFGS method as follows. We generate an auxiliary stochastic gradient at $x_k$ using the samplings from the $(k-1)$-st iteration:
\[\bar{g}_k := \frac{1}{m_{k-1}}\sum_{i=1}^{m_{k-1}} g(x_{k},\xi_{k-1,i}).\]
{Note that we assume that our $\SFO$ can separate two arguments $x_k$ and $\xi_k$ in the stochastic gradient $g(x_{k},\xi_{k-1})$ and generate an output $g(x_k; \xi_{k-1,i})$.}
The stochastic gradient difference is defined as
\be\label{y-k}
y_{k-1} : = \bar{g}_{k} - g_{k-1} = \frac{\sum_{i=1}^{m_{k-1}}g(x_{k},\xi_{k-1,i}) - g(x_{k-1},\xi_{k-1,i})}{m_{k-1}}.
\ee
The iterate difference is still defined as $s_{k-1}=x_{k}-x_{k-1}$. We then define
\be\label{y-k-ini}
\bar{y}_{k-1} = \hat{\theta}_{k-1} y_{k-1} + (1-\hat{\theta}_{k-1})B_{k-1}s_{k-1},
\ee
where
\be\label{def-theta}
\hat{\theta}_{k-1} = \begin{cases}
\frac{0.75s_{k-1}^\top B_{k-1}s_{k-1}}{s_{k-1}^\top B_{k-1}s_{k-1} - s_{k-1}^\top y_{k-1}}, & \mbox{if } s_{k-1}^\top y_{k-1} < 0.25s_{k-1}^\top B_{k-1}s_{k-1},\\
1, & \mbox{otherwise. }
\end{cases}
\ee
Note that if $B_{k-1}\succ 0$, then $0<\hat{\theta}_{k-1}\le 1$. Our stochastic damped BFGS approach updates $B_{k-1}$ as
\be \label{damp-bfgs}
B_{k} = B_{k-1} + \frac{\bar{y}_{k-1} \bar{y}_{k-1}^\top }{s_{k-1}^\top \bar{y}_{k-1}} - \frac{B_{k-1}s_{k-1}s_{k-1}^\top B_{k-1}}{s_{k-1}^\top B_{k-1}s_{k-1}}.
\ee
According to the Sherman-Morrison-Woodbury formula, this corresponds to updating  $H_{k}=B_{k}^{-1}$ as
\be\label{damp-H}
H_{k} = (I-\rho_{k-1} s_{k-1} \bar{y}_{k-1}^\top) H_{k-1} (I -\rho_{k-1} \bar{y}_{k-1} s_{k-1}^\top) + \rho_{k-1} s_{k-1} s_{k-1}^\top,
\ee
where $\rho_{k-1} = (s_{k-1}^\top \bar{y}_{k-1})^{-1}$.
The following lemma shows that the damped BFGS updates \eqref{damp-bfgs} and \eqref{damp-H} preserve the positive definiteness of $B_k$ and $H_k$.
\begin{lem}\label{per-posi}
For $\bar{y}_{k-1}$ defined in \eqref{y-k-ini}, $s_{k-1}^\top \bar{y}_{k-1}\ge 0.25 s_{k-1}^\top B_{k-1} s_{k-1}$. Moreover, if $B_{k-1}=H_{k-1}^{-1}\succ0$, then $B_{k}$ and $H_{k}$ generated by the damped BFGS updates \eqref{damp-bfgs} and \eqref{damp-H} are both positive definite.
\end{lem}
\begin{proof}
From \eqref{y-k-ini} and \eqref{def-theta} we have that
\begin{align*}
s_{k-1}^\top \bar{y}_{k-1} = & \, \hat{\theta}_{k-1}(s_{k-1}^\top {y}_{k-1} - s_{k-1}^\top B_{k-1}s_{k-1}) + s_{k-1}^\top B_{k-1}s_{k-1}\\
               = & \, \begin{cases}
               0.25s_{k-1}^\top B_{k-1}s_{k-1},\quad & \mbox{if }s_{k-1}^\top {y}_{k-1} < 0.25 s_{k-1}^\top B_{k-1}s_{k-1},\\
               s_{k-1}^\top {y}_{k-1},\quad & \mbox{otherwise},
               \end{cases}
\end{align*}
which implies $s_{k-1}^\top\bar{y}_{k-1}\geq 0.25s_{k-1}^\top B_{k-1}s_{k-1}$. Therefore, if $B_{k-1}\succ 0$, it follows that $\rho_{k-1}>0$. As a result, for $H_k$ defined in \eqref{damp-H} and any nonzero vector $z\in\R^n$, we have
\[
z^\top H_{k} z = z^\top(I-\rho_{k-1} s_{k-1} \bar{y}_{k-1}^\top) H_{k-1} (I -\rho_{k-1} \bar{y}_{k-1} s_{k-1}^\top)z + \rho_{k-1} (s_{k-1}^\top z)^2>0,
\]
given that $H_{k-1}\succ 0$. Therefore, both $H_{k}$ and $B_{k}$ defined in \eqref{damp-H} and \eqref{damp-bfgs} are positive definite.
\end{proof}

Computing $H_k$ by the stochastic damped BFGS update \eqref{damp-H}, and computing the step direction $H_kg_k$ requires $O(n^2)$ multiplications. This is costly if $n$ is large.
The L-BFGS method originally proposed by Liu and Nocedal \cite{Liu-Nocedal-89} can be adopted here to reduce this computational cost.
The L-BFGS method can be described as follows for deterministic optimization problems. Given an initial estimate $H_{k,0}\in\R^{n\times n}$ of the inverse Hessian at the current iterate $x_k$ and two sequences $\{s_j\}$, $\{y_j\}$, $j=k-p,\ldots,k-1$, where $p$ is the memory size, the L-BFGS method updates $H_{k,i}$ recursively as
\be\label{L-BFGS}
H_{k,i} = (I-\rho_{j} s_{j} y_{j}^\top)H_{k,i-1}(I-\rho_{j} y_{j} s_{j}^\top) + \rho_{j} s_{j} s_{j}^\top, \quad j=k-(p-i+1); \, i=1,\ldots,p,
\ee
where $\rho_{j}=(s_j^\top y_j)^{-1}$. The output $H_{k,p}$ is then used as the estimate of the inverse Hessian at $x_k$ to compute the search direction at the $k$-th iteration. It can be shown that if the sequence of pairs $\{s_j,y_j\}$ satisfy the curvature condition $s_j^\top y_j>0$, $j=k-1,\ldots,k-p$, then $H_{k,p}$ is positive definite provided that $H_{k,0}$ is positive definite.
Recently, stochastic L-BFGS methods have been proposed for solving strongly convex problems in \cite{BHNS-14,mr14-2,MNJ2015,Gower-icml-2016}. However, the theoretical convergence analyses in these papers do not apply to nonconvex problems. We now show how to design a  stochastic damped L-BFGS formula for nonconvex problems.

Suppose that in the past iterations the algorithm generated $s_j$ and $\bar{y}_j$ that satisfy
\[
s_j^\top \bar{y}_j\ge 0.25 s_j^\top H_{j+1,0}^{-1}s_j,\quad j=k-p,\ldots,k-2.
\]
Then at the current iterate, we compute $s_{k-1}=x_k-x_{k-1}$ and $y_{k-1}$ by \eqref{y-k}. Since $s_{k-1}^\top y_{k-1}$ may not be positive, motivated by the stochastic damped BFGS update \eqref{y-k-ini}-\eqref{damp-H}, we define a new vector $\{\bar{y}_{k-1}\}$ as
\be\label{bar-y}
\bar{y}_{k-1} = \theta_{k-1} y_{k-1} + (1-\theta_{k-1})H_{k,0}^{-1}s_{k-1},
\ee
where
\be\label{theta_k-1}
\theta_{k-1} = \begin{cases}
\frac{0.75 s_{k-1}^\top H_{k,0}^{-1}s_{k-1}}{s_{k-1}^\top H_{k,0}^{-1}s_{k-1} - s_{k-1}^\top y_{k-1}}, & \mbox{if } s_{k-1}^\top y_{k-1} < 0.25s_{k-1}^\top H_{k,0}^{-1}s_{k-1},\\
1, & \mbox{otherwise. }
\end{cases}
\ee
Similar to Lemma \ref{per-posi}, we can prove that
\[
s_{k-1}^\top \bar{y}_{k-1}\ge 0.25 s_{k-1}^\top H_{k,0}^{-1} s_{k-1}.
\]
Using $s_j$ and $\bar{y}_j$, $j=k-p,\ldots,k-1$, we define the stochastic damped L-BFGS formula as
\be\label{mod-L-BFGS}
H_{k,i} = (I-\rho_{j} s_{j} \bar{y}_{j}^\top)H_{k,i-1}(I-\rho_{j} \bar{y}_{j} s_{j}^\top) + \rho_{j} s_{j} s_{j}^\top, \quad j=k-(p-i+1); \, i=1,\ldots,p,
\ee
where $\rho_j = (s_j^\top \bar{y}_j)^{-1}$.
As in the analysis in Lemma \ref{per-posi}, by induction we can show that $H_{k,i}\succ0$, $i=1,\ldots,p$. Note that when $k< p$, we use $s_j$ and $\bar{y}_j$, $j=1,\ldots,k$ to execute the stochastic damped L-BFGS update.

We next discuss the choice of $H_{k,0}$. A popular choice in the standard L-BFGS method is to set $H_{k,0} = \frac{s_{k-1}^\top y_{k-1}}{y_{k-1}^\top y_{k-1}} I$. Since $s_{k-1}^\top y_{k-1}$ may not be positive for nonconvex problems, we set
\be\label{gama-k}
H_{k,0}=\gamma_k^{-1} I, \quad \mbox{ where } \gamma_k = \max\left\{
 \frac{y_{k-1}^\top y_{k-1}}{s_{k-1}^\top y_{k-1}} , \delta\right\}\geq \delta,
\ee
where $\delta>0$ is a given constant.

To prove that $H_k=H_{k,p}$ generated by \eqref{mod-L-BFGS}-\eqref{gama-k} satisfies assumptions {\bf AS.3} and {\bf AS.4}, we need to make the following assumption.
\begin{itemize}
\item[{\bf AS.5}]\quad The function $F(x,\xi)$ is twice continuously differentiable with respect to $x$. The stochastic gradient $g(x,\xi)$ is computed as $g(x,\xi)=\nabla_x F(x,\xi)$, and there exists a positive constant $\kappa$ such that $\|\nabla_{xx}^2 F(x,\xi)\|\le \kappa$, for any $x,\xi$.
\end{itemize}
Note that {\bf AS.5} is equivalent to requiring that $-\kappa I\preceq \nabla_{xx}^2 F(x,\xi) \preceq\kappa I$, rather than the strong convexity assumption $0 \prec \underline{\kappa} I \preceq \nabla_{xx}^2 F(x,\xi) \preceq\kappa I$ required in \cite{BHNS-14,mr14-2}.
The following lemma shows that the eigenvalues of $H_k$ are bounded below away from zero under assumption {\bf AS.5}.
\begin{lem}\label{low}
Suppose that {\bf AS.5} holds. Given $H_{k,0}$ defined in \eqref{gama-k}, suppose that $H_k=H_{k,p}$ is updated through the stochastic damped L-BFGS formula \eqref{mod-L-BFGS}. Then all the eigenvalues of $H_k$ satisfy
\be\label{low-H}
\lambda(H_k) \ge \left(\frac{4p\kappa^2}{\delta} +(4p+1)(\kappa+\delta)\right)^{-1}.
\ee
\end{lem}
\begin{proof}
According to Lemma \ref{per-posi}, $H_{k,i}\succ 0$, $i=1,\ldots,p$. To prove that the eigenvalues of $H_k$ are bounded below away from zero, it suffices to prove that the eigenvalues of $B_k=H_k^{-1}$ are bounded from above. From the damped L-BFGS formula \eqref{mod-L-BFGS}, $B_k=B_{k,p}$ can be computed recursively as
\[
B_{k,i} = B_{k,i-1} + \frac{\bar{y}_j\bar{y}_j^\top}{s_j^\top \bar{y}_j} - \frac{B_{k,i-1}s_j s_j^\top B_{k,i-1}}{s_j^\top B_{k,i-1}s_j}, \quad j=k-(p-i+1);i=1,\ldots,p,
\]
starting from $B_{k,0}=H_{k,0}^{-1}=\gamma_k I$.
Since $B_{k,0}\succ0$, Lemma \ref{per-posi} indicates that $B_{k,i}\succ0$ for $i=1,\ldots,p$. Moreover, the following inequalities hold:
{\be\label{tr-B}
\|B_{k,i}\| \le \left\|B_{k,i-1} - \frac{B_{k,i-1}s_j s_j^\top B_{k,i-1}}{s_j^\top B_{k,i-1}s_j}\right\| + \left\|\frac{\bar{y}_j\bar{y}_j^\top}{s_j^\top \bar{y}_j}\right\| \le \|B_{k,i-1}\| + \left\|\frac{\bar{y}_j\bar{y}_j^\top}{s_j^\top \bar{y}_j}\right\| =\|B_{k,i-1}\| + \frac{\bar{y}_j^\top\bar{y}_j}{s_j^\top \bar{y}_j}.
\ee}
From the definition of $\bar{y}_j$ in \eqref{bar-y} and the facts that $s_j^\top \bar{y}_j\ge 0.25 s_j^\top B_{j+1,0}s_j$ and $B_{j+1,0}=\gamma_{j+1}I$ from \eqref{gama-k}, we have that for any $j=k-1,\ldots,k-p$
\begin{align}
\frac{\bar{y}_j^\top\bar{y}_j}{s_j^\top \bar{y}_j} \le  \, 4 \frac{\|\theta_j y_j + (1-\theta_j)B_{j+1,0}s_j\|^2}{s_j^\top B_{j+1,0}s_j} 
=   \, 4\theta_j^2\frac{y_j^\top y_j}{\gamma_{j+1}s_j^\top s_j} + 8\theta_j(1-\theta_j)\frac{y_j^\top s_j}{s_j^\top s_j} + 4(1-\theta_j)^2 \gamma_{j+1}. \label{proof-lemma-low-1}
\end{align}
Note that from \eqref{y-k} we have
\[
y_j = \frac{\sum_{l=1}^{m_j}g(x_{j+1},\xi_{j,l})-g(x_j,\xi_{j,l})}{m_j} = \frac{1}{m_j}\left(\sum_{l=1}^{m_j}\overline{\nabla^2_{xx} F}(x_j,\xi_{j,l},s_j)\right)s_j,
\]
where $\overline{\nabla^2_{xx} F}(x_j,\xi_{j,l},s_j) = \int_0^1\nabla_{xx}^2F(x_j+ts_j,\xi_{j,l})dt$, because $g(x_{j+1},\xi_{j,l})-g(x_j,\xi_{j,l})=\int_0^1\frac{dg}{dt}(x_j+ts_j,\xi_{j,l})dt = \int_0^1\nabla_{xx}^2 F(x_j+ts_j,\xi_{j,l})s_jdt$. Therefore, for any $j=k-1,\ldots,k-p$, from \eqref{proof-lemma-low-1}, and the facts that $0<\theta_j\le1$ and $\delta\le\gamma_{j+1} \le\kappa+\delta$, and the assumption {\bf AS.5} it follows that
\begin{align}
\frac{\bar{y}_j^\top\bar{y}_j}{s_j^\top \bar{y}_j} \le  \, \frac{4\theta_j^2\kappa^2}{\gamma_{j+1}} + 8\theta_j(1-\theta_j)\kappa + 4(1-\theta_j)^2\gamma_{j+1} 
\le \, \frac{4\theta_j^2\kappa^2}{\delta} + 4[(1-\theta_j^2)\kappa + (1-\theta_j)^2\delta] 
\le \, \frac{4\kappa^2}{\delta} + 4(\kappa+\delta). \label{proof-lemma-low-2}
\end{align}
Combining \eqref{tr-B} and \eqref{proof-lemma-low-2} yields
\[
\|B_{k,i}\|\le \|B_{k,i-1}\| + 4\left(\frac{\kappa^2}{\delta}+\kappa+\delta\right).
\]
By induction, we have that
\begin{align*}
\|B_k\|=\|B_{k,p}\| \le \, \|B_{k,0}\| + 4p\left(\frac{\kappa^2}{\delta} + \kappa+\delta\right)
\le \, \frac{4p\kappa^2}{\delta} +(4p+1)(\kappa+\delta),
\end{align*}
which implies \eqref{low-H}.
\end{proof}

We now prove that $H_k$ is uniformly bounded above.
\begin{lem}\label{upp}
Suppose that the assumption {\bf AS.5} holds. Given $H_{k,0}$ defined in \eqref{gama-k}, suppose that $H_k=H_{k,p}$ is updated through the stochastic damped L-BFGS formula \eqref{mod-L-BFGS}. Then $H_k$ satisfies
\be\label{up-H}
\lambda_{\max}(H_k) = \|H_k\| \le \left(\frac{\alpha^{2p}-1}{\alpha^{2}-1}\right)\frac{4}{\delta} + \frac{\alpha^{2p}}{\delta},
\ee
where $\alpha = (4\kappa + 5\delta)/\delta$, and $\lambda_{\max}(H_k)$ and $\|H_k\|$ denote, respectively, the maximum eigenvalue and operator norm $\|\cdot\|$ of $H_k$.
\end{lem}
\begin{proof}
For notational simplicity, let $H=H_{k,i-1}$, $H^+=H_{k,i}$, $s=s_j$, $\bar{y}=\bar{y}_j$, $\rho=(s_j^\top\bar{y}_j)^{-1}=(s^\top\bar{y})^{-1}$. Now \eqref{damp-H} can be written as
\[H^+ = H -\rho(H\bar{y}s^\top+s\bar{y}^\top H) + \rho ss^\top + \rho^2(\bar{y}^\top H \bar{y})ss^\top.\]
Using the facts that $\|uv^\top\|=\|u\|\cdot\|v\|$ for any vectors $u$ and $v$, $\rho s^\top s = \rho\|s\|^2 = \frac{s^\top s}{s^\top\bar{y}}\leq \frac{4}{\delta}$, and $\frac{\|\bar{y}\|^2}{s^\top\bar{y}}\leq 4\left(\frac{\kappa^2}{\delta}+\kappa+\delta\right)<\frac{4}{\delta}(\kappa+\delta)^2$, which follows from \eqref{proof-lemma-low-2},
we have that
\[\|H^+\|\leq \|H\| + \frac{2\|H\|\cdot\|\bar{y}\|\cdot\|s\|}{s^\top\bar{y}}+\frac{s^\top s}{s^\top\bar{y}}+\frac{s^\top s}{s^\top\bar{y}}\cdot\frac{\|H\|\cdot\|\bar{y}\|^2}{s^\top\bar{y}}.\]
Noting that $\frac{\|\bar{y}\|\|s\|}{s^\top\bar{y}}=\left[\frac{\|\bar{y}\|^2}{s^\top\bar{y}}\cdot\frac{\|s\|^2}{s^\top\bar{y}}\right]^{1/2}$,
it follows that
\begin{align*}\|H^+\|\leq \, \left(1+2\cdot\frac{4}{\delta}(\kappa+\delta)+\left(\frac{4}{\delta}(\kappa+\delta)\right)^2\right)\|H\|+\frac{4}{\delta}   =   \, (1+(4\kappa+4\delta)/\delta)^2\|H\|+\frac{4}{\delta}.\end{align*}
Hence, by induction we obtain \eqref{up-H}.
\end{proof}


Lemmas \ref{low} and \ref{upp} indicate that $H_k$ generated by \eqref{bar-y}-\eqref{mod-L-BFGS} satisfies assumption {\bf AS.3}. Moreover, since $y_{k-1}$ defined in \eqref{y-k} does not depend on random samplings in the $k$-th iteration, it follows that $H_k$ depends only on $\xi_{[k-1]}$ and assumption {\bf AS.4} is satisfied.

To analyze the cost of computing the step direction $H_kg_k$, note that from \eqref{mod-L-BFGS}, $H_k$ can be represented as
\be\label{mod-LBFGS-1}
H_{k,i} = (I-\rho_{j} s_{j} \bar{y}_{j}^\top)H_{k,i-1}(I-\rho_{j} \bar{y}_{j} s_{j}^\top) + \rho_{j} s_{j} s_{j}^\top, \quad j=k-(p-i+1); \, i=1,\ldots,p,
\ee
which is the same as the classical L-BFGS formula in \eqref{L-BFGS}, except that $y_j$ is replaced by  $\bar{y}_j$. Hence, we can compute the step direction by the two-loop recursion, implemented in the following procedure. 

{\floatname{algorithm}{Procedure}
\begin{algorithm}[ht]
\caption{{\bf Step computation using stochastic damped L-BFGS}}\label{comp-mod-LBFGS}
\begin{algorithmic}[1]
\REQUIRE {Let $x_k$ be current iterate. Given the stochastic gradient $g_{k-1}$ at iterate $x_{k-1}$, the random variable $\xi_{k-1}$, the batch size $m_{k-1}$, $s_{j}$, $\bar{y}_{j}$ and $\rho_j$, $j=k-p,\ldots,k-2$ and $u_0=g_k$,.}
\ENSURE {$H_kg_k=v_p$.}
\STATE Set $s_{k-1} = x_k - x_{k-1}$ and calculate $y_{k-1}$ through \eqref{y-k}
\STATE Calculate $\gamma_k$ through \eqref{gama-k}
\STATE Calculate $\bar{y}_{k-1}$ through \eqref{bar-y} and $\rho_{k-1}=(s_{k-1}^\top \bar{y}_{k-1})^{-1}$
\FOR {$i=0,\ldots,\min\{p,k-1\}-1$}
  \STATE Calculate $\mu_i=\rho_{k-i-1}u_i^\top s_{k-i-1}$
  \STATE Calculate $u_{i+1} = u_i - \mu_i \bar{y}_{k-i-1}$
  \ENDFOR
  \STATE Calculate $v_0=\gamma_k^{-1}u_p$
  \FOR {$i=0,\ldots,\min\{p,k-1\}-1$}
  \STATE Calculate $\nu_i=\rho_{k-p+i}v_i^\top \bar{y}_{k-p+i}$
  \STATE Calculate $
  v_{i+1} = v_i + (\mu_{p-i-1}-\nu_i)s_{k-p+i}$.
  \ENDFOR
\end{algorithmic}
\end{algorithm}}

We now analyze the computational cost of Procedure \ref{comp-mod-LBFGS}. In Step 2, the computation of $\gamma_k$ involves $y_{k-1}^\top y_{k-1}$ and $s_{k-1}^\top y_{k-1}$, which take $2n$ multiplications. In Step 3, from the definition of $\bar{y}_k$ in \eqref{bar-y}, since $s_{k-1}^\top y_{k-1}$ has been obtained in a previous step, one only needs to compute $s_{k-1}^\top s_{k-1}$ and some scalar-vector products, thus the computation of $\bar{y}_{k-1}$ takes $3n$ multiplications. Due to the fact that
\[
\rho_{k-1}^{-1} = s_{k-1}^\top \bar{y}_{k-1} = \begin{cases}
0.25\gamma_k s_{k-1}^\top s_{k-1},\quad & \mbox{if }s_{k-1}^\top {y}_{k-1} < 0.25 \gamma_k s_{k-1}^\top s_{k-1},\\
               s_{k-1}^\top {y}_{k-1},\quad & \mbox{otherwise},
\end{cases}
\]
all involved computations have been done for $\rho_{k-1}$. Furthermore, the first loop Steps 4-7 involves $p$ scalar-vector multiplications and $p$ vector inner products. So does the second loop Steps 9-12. Including the product $\gamma_k^{-1}u_p$, the whole procedure takes $(4p+6)n$ multiplications. 

Notice that in Step 1 of Procedure \ref{comp-mod-LBFGS}, the computation of $y_{k-1}$ involves the evaluation of $\sum_{i=1}^{m_{k-1}}g(x_{k},\xi_{k-1,i})$, which requires $m_{k-1}$ $\SFO$-calls. As a result, when Procedure \ref{comp-mod-LBFGS} is plugged into SQN (Algorithm \ref{sso-uncons}), the total number of $\SFO$-calls needed in the $k$-th iteration becomes $m_k+m_{k-1}$. This leads to the following overall $\SFO$-calls complexity result for our stochastic damped L-BFGS method.

\begin{thm}
Suppose that {\bf AS.1}, {\bf AS.2} and {\bf AS.5} hold. Let $N_{sfo}$ denote the total number of $\SFO$-calls in SQN (Algorithm \ref{sso-uncons}) in which Procedure \ref{comp-mod-LBFGS} is used to compute $H_{k}g_k$. Under the same conditions as in Corollary \ref{cor3.4}, to achieve  $\E[\|\nabla f(x_R)\|^2]\le \epsilon$, $N_{sfo}\leq 2\bar{N}$, where $\bar{N}$ satisfies \eqref{bar-N}, i.e., is $O(\epsilon^{-2})$.
\end{thm}

{
\section{SdLBFGS with a Variance Reduction Technique}\label{sec:sdlbfgs-vr}



Motivated by the recent advance of SVRG for nonconvex minimization proposed in \cite{SVR-nonconvex} and \cite{AllenHazan2016-nonconvex}, we now present a variance reduced SdLBFGS method, which we call SdLBFGS-VR, for solving \eqref{sum-f-i}. Here, the mini-batch stochastic gradient is defined as $g(x) = \frac{1}{|\C{K}|}\sum_{i\in\C{K}}\nabla f_i(x)$, where the subsample set $\C{K}\subseteq [T]$ is randomly chosen from $\{1,\ldots,T\}$. SdLBFGS-VR allows a constant step size, and thus can accelerate the convergence speed of SdLBFGS. SdLBFGS-VR is summarized in Algorithm \ref{L-BFGS-VR}.

\begin{algorithm}[H]
\caption{\quad \bf SdLBFGS with variance reduction (SdLBFGS-VR)}
\label{L-BFGS-VR}
\begin{algorithmic}[1]
\REQUIRE $\tilde{x}_0\in\R^d$, $\{\alpha_t\}_{t=0}^{m-1}$
\ENSURE {Iterate $x$ chosen uniformly random from $\{x_t^{k+1}: t=0,\ldots,q-1; k=0,\ldots,N-1\}$}
\FOR {$k=0,\ldots,N-1$}
   \STATE $x_0^{k+1}=\tilde{x}_k$
   \STATE compute $\nabla f(\tilde{x}_k)$
   \FOR {$t=0,\ldots,q-1$}
     \STATE Sample a set $\C{K}\subseteq [T]$ with $|\C{K}|=m$
     \STATE Compute $g_t^{k+1}=\nabla f_{\C{K}}(x_t^{k+1}) - \nabla f_{\C{K}}(\tilde{x}_k) + \nabla f(\tilde{x}_k)$ where $\nabla f_{\C{K}}(x_t^{k+1})=\frac{1}{m}\sum_{i\in\mathcal{K}}\nabla f_i(x_t^{k+1})$
     \STATE Compute $d_t^{k+1}=-H_t^{k+1}g_t^{k+1}$ through Procedure \ref{comp-mod-LBFGS}
     \STATE Set $x_{t+1}^{k+1}=x_t^{k+1}+\alpha_t^{k+1}d_t^{k+1}$
   \ENDFOR
   \STATE Set $\tilde{x}_{k+1}=x_q^{k+1}$
\ENDFOR
\end{algorithmic}
\end{algorithm}

We now analyze the $\SFO$-calls complexity of Algorithm \ref{L-BFGS-VR}. 
We first analyze the convergence rate of SdLBFGS-VR, essentially following \cite{SVR-nonconvex}.

\begin{lem}
Suppose assumptions {\bf AS.1, AS.2} and {\bf AS.5} hold. Set $c_t^{k+1}=c_{t+1}^{k+1}(1 + \alpha_t^{k+1}\beta_t + 2L^2(\alpha_t^{k+1})^2\bar{\kappa}^2/m) + (\alpha_t^{k+1})^2L^3 \bar{\kappa}^2/m$.
It holds that
\be\label{lem-4-1-conclusion}
\Gamma_t^{k+1}\E[\|\nabla f(x_t^{k+1})\|^2]\le R_t^{k+1}-R_{t+1}^{k+1},
\ee
where $R_t^{k+1} = \E[f(x_t^{k+1})+ c_t^{k+1}\|x_t^{k+1}-\tilde{x}_k\|^2]$ and $\Gamma_t^{k+1}=\alpha_t^{k+1}\underline{\kappa}- c_{t+1}^{k+1}\alpha_t^{k+1}\bar{\kappa}^2/\beta_t - (\alpha_t^{k+1})^2L\bar{\kappa}^2-2c_{t+1}^{k+1}(\alpha_t^{k+1})^2\bar{\kappa}^2$
for any $\beta_t>0$.
\end{lem}

\begin{proof}
It follows from {\bf AS.1} and the fact that $\underline{\kappa}I\preceq H_t^{k+1} \preceq \bar{\kappa}I$, for any $t=0,\ldots,q-1; k=0,\ldots,N-1$, which follows under assumption {\bf AS.5}, from Lemmas \ref{low} and \ref{upp}, that,
\begin{align}
\E[f(x_{t+1}^{k+1})] \le\, & \E[f(x_t^{k+1}) + \langle \nabla f(x_t^{k+1}), x_{t+1}^{k+1} - x_t^{k+1}\rangle + \frac{L}{2}\|x_{t+1}^{k+1}-x_t^{k+1}\|^2] \nonumber \\
=\, & \E[f(x_t^{k+1}) - \alpha_t^{k+1} \langle \nabla f(x_t^{k+1}), H_t^{k+1}g_t^{k+1}\rangle + \frac{(\alpha_t^{k+1})^2 L }{2} \|H_t^{k+1}g_{t}^{k+1}\|^2] \nonumber \\
\le \, & \E[f(x_t^{k+1}) - \alpha_t^{k+1}\underline{\kappa}\|\nabla f(x_t^{k+1})\|^2 + \frac{(\alpha_t^{k+1})^2 L\bar{\kappa}^2}{2}\|g_t^{k+1}\|^2].\label{lem-4-1-inequality-1}
\end{align}
Moreover, for any $\beta_t>0$, since $g_t^{k+1}$ is an unbiased estimate of $\nabla f(x_t^{k+1})$, it holds that,
\begin{align}
\E[\|x_{t+1}^{k+1} - \tilde{x}_k\|^2] = \, & \E[\| x_{t+1}^{k+1}  - x_t^{k+1}\|^2 + \|x_t^{k+1}- \tilde{x}_k\|^2 + 2\langle x_{t+1}^{k+1} - x_t^{k+1}, x_t^{k+1}- \tilde{x}_k\rangle] \label{lem-4-1-inequality-2} \\
=\, & \E[(\alpha_t^{k+1})^2\|H_t^{k+1}g_t^{k+1}\|^2 + \|x_t^{k+1}- \tilde{x}_k\|^2 ] - 2\alpha_t^{k+1}\langle H_t^{k+1}\nabla f(x_t^{k+1}),x_t^{k+1}-\tilde{x}_k\rangle \nonumber\\
\le\, & \E[(\alpha_t^{k+1})^2\bar{\kappa}^2\|g_t^{k+1}\|^2 + \|x_t^{k+1}- \tilde{x}_k\|^2 ] + \alpha_t^{k+1}\E[\beta_t^{-1}{\|H_t^{k+1}\nabla f(x_t^{k+1})\|^2} + \beta_t\|x_t^{k+1}- \tilde{x}_k\|^2.] \nonumber
\end{align}
Furthermore, the following inequality holds:
\begin{align}
\E[\|g_t^{k+1}\|^2] =& \, \E[\|\nabla f_{\C{K}}(x_t^{k+1}) -\nabla f(x_t^{k+1}) + \nabla f(x_t^{k+1}) -\nabla f_{\C{K}}(\tilde{x}_k)  + \nabla f(\tilde{x}_k)\|^2] \nonumber\\
\le &\,  2\E[\|\nabla f(x_t^{k+1})\|^2] + 2\E[\|\nabla f_{\C{K}}(x_t^{k+1})  -\nabla f_{\C{K}}(\tilde{x}_k) - (\nabla f(x_t^{k+1})  - \nabla f(\tilde{x}_k))\|^2]\nonumber\\
= & \, 2\E[\|\nabla f(x_t^{k+1})\|^2] + 2\E[\|\nabla f_{\C{K}}(x_t^{k+1})  -\nabla f_{\C{K}}(\tilde{x}_k) - \E[\nabla f_{\C{K}}(x_t^{k+1})  -\nabla f_{\C{K}}(\tilde{x}_k)]\|^2]\nonumber\\
\le &\,  2\E[\|\nabla f(x_t^{k+1})\|^2] +  2\E[\|\nabla f_{\C{K}}(x_t^{k+1})  -\nabla f_{\C{K}}(\tilde{x}_k)\|^2]\nonumber\\
= & \, 2\E[\|\nabla f(x_t^{k+1})\|^2] + \frac{2}{m}\E[\|\nabla f_i(x_t^{k+1}) - \nabla f_i(\tilde{x}_k)\|^2]\nonumber\\
\le &\, 2\E[\|\nabla f(x_t^{k+1})\|^2] + \frac{2L^2}{m}\E[\|x_t^{k+1} - \tilde{x}_k\|^2].\label{lem-4-1-inequality-3}
\end{align}

Combining \eqref{lem-4-1-inequality-1}, \eqref{lem-4-1-inequality-2} and \eqref{lem-4-1-inequality-3} yields that,
\begin{align*}
R_{t+1}^{k+1} \le & \,\E[f(x_{t+1}^{k+1}) + c_{t+1}^{k+1}\|x_{t+1}^{k+1}-\tilde{x}_k\|^2]\\
\le &\, \E[f(x_t^{k+1}) - \alpha_t^{k+1}\underline{\kappa}\|\nabla f(x_t^{k+1})\|^2 + \frac{(\alpha_t^{k+1})^2 L\bar{\kappa}^2}{2}\|g_t^{k+1}\|^2 + c_{t+1}^{k+1}(\alpha_t^{k+1})^2\bar{\kappa}^2\E[\|g_t^{k+1}\|^2] + c_{t+1}^{k+1}\E[ \|x_t^{k+1}- \tilde{x}_k\|^2 ] \\
&\, + \alpha_t^{k+1}c_{t+1}^{k+1}\E[\beta_t^{-1}{\bar{\kappa}^2\|\nabla f(x_t^{k+1})\|^2} + \beta_t\|x_t^{k+1}- \tilde{x}_k\|^2]\\
= &\, \E[f(x_t^{k+1})] - (\alpha_t^{k+1}\underline{\kappa}- c_{t+1}^{k+1}\alpha_t^{k+1}\bar{\kappa}^2/\beta_t)\E[\|\nabla f(x_t^{k+1})\|^2] + ((\alpha_t^{k+1})^2 L\bar{\kappa}^2/2+c_{t+1}^{k+1}(\alpha_t^{k+1})^2\bar{\kappa}^2)\E[\|g_t^{k+1}\|^2] \\
&\, + c_{t+1}^{k+1}(1+\alpha_t^{k+1}\beta_t)\E[ \|x_t^{k+1}- \tilde{x}_k\|^2 ] \\
\le &\,  \E[f(x_{t}^{k+1})] + (c_{t+1}^{k+1}1 + c_{t+1}^{k+1} \alpha_t^{k+1}\beta_t + 2c_{t+1}^{k+1}L^2(\alpha_t^{k+1})^2\bar{\kappa}^2/m + (\alpha_t^{k+1})^2L^3\bar{\kappa}^2/m)\E[\|x_t^{k+1}-\tilde{x}_k\|^2]\\
&\, - (\alpha_t^{k+1}\underline{\kappa} - c_{t+1}^{k+1}\alpha_t^{k+1}\bar{\kappa}^2/\beta_t- (\alpha_t^{k+1})^2 L\bar{\kappa}^2 - 2c_{t+1}^{k+1}(\alpha_t^{k+1})^2\bar{\kappa}^2 )\E[\|\nabla f(x_t^{k+1})\|^2] \\
= & \, R_t^{k+1} - \Gamma_t^{k+1} \E[\|\nabla f(x_t^{k+1})\|^2],
\end{align*}
which further implies \eqref{lem-4-1-conclusion}.
\end{proof}

\begin{thm}\label{thm4.1}
Suppose assumptions {\bf AS.1, AS.2} and {\bf AS.5} hold. Set $\beta_t=\beta=\frac{L\bar{\kappa}}{T^{1/3}}$, $c_q^{k+1}=c_q=0$. Suppose that there exist two positive constants $\nu, \mu_0 \in (0,1)$ such that
\be\label{thm4.1-mu0}(1-\frac{\nu\bar{\kappa}}{\mu_0\underline{\kappa}}) \underline{\kappa} \geq \mu_0\bar{\kappa}(e-1)+\mu_0m\bar{\kappa}+2\mu_0^2\bar{\kappa}m(e-1)\ee
holds. Set $\alpha_t^{k+1}=\alpha=\frac{\mu_0m}{L\bar{\kappa}T^{2/3}}$, and $q=\lfloor\frac{T}{3\mu_0m}\rfloor$.
\be\label{thm4.1-conclusion}
\E[\|\nabla f(x)\|^2] \le \frac{T^{2/3}L[f(x_0)-f(x_*)]}{qNm\nu}.
\ee
\end{thm}
\begin{proof}
Denote $\theta=\alpha\beta+2\alpha^2 L^2\bar{\kappa}^2/m$. It then follows that
$\theta = \mu_0m/T + 2\mu_0^2m/T^{4/3} \le 3\mu_0m/T$, and $(1+\theta)^q\leq e$, where $e$ is the Euler's number.
Because $c_q=0$, for any $k\geq 0$, we have
\[
c_0:=c_0^{k+1}=\frac{\alpha^2 L^3\bar{\kappa}^2}{m}\cdot\frac{(1+\theta)^q-1}{\theta} = \frac{\mu_0^2mL((1+\theta)^q-1)}{ T^{4/3}\theta} = \frac{\mu_0^2L((1+\theta)^q-1)}{\mu_0T^{1/3}+2\mu_0^2}\le \frac{\mu_0L((1+\theta)^q-1)}{ T^{1/3}}\le \frac{\mu_0L(e-1)}{T^{1/3}}.
\]
Therefore, it follows that
\[\min_t\Gamma_t^{k+1}\ge \frac{\nu m}{LT^{2/3}}.\]
As a result, we have
\[
\sum_{t=0}^{q-1} \E[\|\nabla f(x_t^{k+1})\|^2] \le \frac{R_0^{k+1}-R_q^{k+1}}{\min_t\Gamma_t^{k+1} } = \frac{\E[f(\tilde{x}_{k})- f(\tilde{x}_{k+1})]}{\min_t\Gamma_t^{k+1} },
\]
which further yields that
\[
\E[\|\nabla f(x)\|^2] = \frac{1}{qN}\sum_{k=0}^{N-1}\sum_{t=0}^{q-1} \E[\|\nabla f(x_t^{k+1})\|^2] \le\frac{f(x_0)-f(x_*)}{qN\min_t\Gamma_t^{k+1}} \le \frac{T^{2/3}L[f(x_0)-f(x_*)]}{qNm\nu }.
\]
\end{proof}

\begin{cor}
Under the same conditions as Theorem \ref{thm4.1}, to achieve $\E[\|\nabla f(x)\|^2] \leq \epsilon$, the total number of component gradient evaluations required in Algorithm \ref{L-BFGS-VR} is
$
O\left({T^{2/3}}/{\epsilon}\right).
$
\end{cor}

\begin{proof}
From Theorem \ref{thm4.1}, it follows that to obtain an $\epsilon$-solution, the outer iteration number $N$ of Algorithm \ref{L-BFGS-VR} should be in the order of
$
O(\frac{T^{2/3}}{qm\epsilon}) = O(\frac{T^{-1/3}}{\epsilon}),
$
which is due to the fact that $qm=O(T)$. As a result, the total number of component gradient evaluations is $(T+qm)N$, which is $O(T^{2/3}/\epsilon)$.
\end{proof}
}

\renewcommand{\thefigure}{5.\arabic{figure}}

\setcounter{figure}{0}

\section{Numerical Experiments}\label{sec:num}
{
In this section, we empirically study the performance of the proposed SdLBFGS and SdLBFGS-VR methods. We compare SdLBFGS with SGD with $g_k$ given by \eqref{G-k} using a diminishing step size $\alpha_k=\beta/k$ in both methods, for solving the following nonconvex support vector machine (SVM) problem with a sigmoid loss function, which has been considered in \cite{gl13,Mason-nips-1999}:
\be\label{prob_SVM}
\min_{x\in\R^n}\quad f(x):=\E_{u,v}[1-\mathrm{tanh}(v\langle x,u\rangle)] + \lambda \|x\|_2^2,
\ee
where $\lambda>0$ is a regularization parameter, $u\in\R^n$ denotes the feature vector, and $v\in\{-1,1\}$ refers to the corresponding label. In our experiments, $\lambda$ was chosen as $10^{-4}$.
We also compare SdLBFGS-VR with SVRG \cite{SVR-nonconvex}, using a constant step size in both methods, for solving
\be\label{prob_SVM-FS}
\min_{x\in\R^n} \quad \frac{1}{T}\sum_{i=1}^T f_i(x) + \lambda\|x\|^2,
\ee
where $f_i(x) = 1-\mathrm{tanh}(v_i\langle x,u_i\rangle)$, $i=1,\ldots,T$. In all of our tests, we used the same mini-batch size $m$ at every iteration.} All the algorithms were implemented in Matlab R2013a on a PC with a 2.60 GHz Intel microprocessor and 8GB of memory.

\subsection{Numerical results for SdLBFGS on synthetic data}\label{sec:synthetic}

In this subsection, we report numerical results for SdLBFGS and SGD for solving \eqref{prob_SVM} on synthetic data for problems of dimension $n=500$.
We set the initial point for both methods to $x_1=5*\bar{x}_1$, where $\bar{x}_1$ was drawn from the uniform distribution over $[0,1]^n$. We generated training and testing points $(u,v)$ in the following manner. We first generated a sparse vector $u$ with 5\% nonzero components following the uniform distribution on $[0,1]^n$, and then set $v=\mbox{sign}(\langle \bar{x},u \rangle)$ for some $\bar{x}\in\R^n$ drawn from the uniform distribution on $[-1,1]^n$. Every time we computed a stochastic gradient of the sigmoid loss function in \eqref{prob_SVM}, we accessed $m$ of these data points $(u,v)$ without replacement. {Drawing data points without replacement is a common practice in testing the performance of stochastic algorithms, although in this case the random samples are not necessarily independent.}

In Figure \ref{fig1} we compare the performance of SGD and SdLBFGS with various memory sizes $p$. The batch size was set to $m=100$ and the stepsize to both $\alpha_k = 10/k$ and $\alpha_k = 20/k$ for SGD and $10/k$ for SdLBFGS. In the left figure we plot the squared norm of the gradient (SNG) versus the number of iterations, up to a total of $1000$. The SNG was computed using $N=5000$ randomly generated testing points $(u_i,v_i), i=1,\ldots,N$ as:
\be\label{sng}\mbox{ SNG } = \left\|\frac{1}{N}\sum_{i=1}^N \nabla_x F(\tilde{x};u_i,v_i) + 2\lambda \tilde{x}\right\|^2, \ee
where $\tilde{x}$ is the output of the algorithm and $F(x;u,v) = 1-\mathrm{tanh}(v\langle x,u\rangle)$.
In the right figure we plot the percentage of correctly classified testing data. From Figure \ref{fig1} we can see that increasing the memory size improves the performance of SdLBFGS. When the memory size $p=0$, we have $H_k=\gamma_k^{-1}I$ and SdLBFGS reduces to SGD with an adaptive stepsize. SdLBFGS with memory size $p=1, 3$ oscillates quite a lot. The variants of SdLBFGS with memory size $p=5$, $10$, $20$ all oscillate less and perform quite similarly, and they all significantly outperform SdLBFGS with $p=0$ and $p=1$.


{In Figure \ref{fig-delta}, we report the performance of SdLBFGS with different $\delta$ used in \eqref{gama-k}. From Figure \ref{fig-delta} we see that SdLBFGS performs best with small $\delta$ such as $\delta = 0.01, 0.1$ and $1$.
}


In Figure \ref{fig1-batch} we report the effect of the batch size $m$ on the performance of SGD and SdLBFGS with memory size $p=20$.
For SdLBFGS, the left figure shows that $m=500$ gives the best performance among the three choices 50, 100 and 500, tested, with respect to the total number of iterations taken. This is because a larger batch size leads to gradient estimation with lower variance. The right figure shows that if the total number of $\SFO$-calls is fixed, then because of the tradeoff between the number of iterations and the batch size, i.e., because the number of iterations is proportional to the reciprocal of the batch size, the SdLBFGS variant corresponding to $m=100$ slightly outperforms the $m=500$ variant.


In Figure \ref{fig4-cls}, we report the percentage of correctly classified data for $5000$ randomly generated testing points. The results are consistent with the one shown in the left figure of Figure \ref{fig1-batch}, i.e., the ones with a lower squared norm of the gradient give a higher percentage of correctly classified data.


Moreover, we also counted the number of steps taken by SdLBFGS in which $s_{k-1}^\top y_{k-1} < 0$. We set the total number of iterations to 1000 and tested the effect of the memory size and batch size of SdLBFGS on the number of such steps. For fixed batch size $m=50$, the average numbers of such steps over 10 runs of SdLBFGS were respectively equal to $(178,50,36,42,15)$ when the memory sizes were $(1,3,5,10,20)$. For fixed memory size $p=20$, the average numbers of such steps over 10 runs of SdLBFGS were respectively equal to $(15,6,1)$ when the batch sizes are $(50,100,500)$. Therefore, the number of such steps roughly decreases as the memory size $p$ and the batch size $m$ increase. This is to be expected because as $p$ increases, there is less negative effect caused by ``limited-memory''; and as $m$ increases, the gradient estimation has lower variance.

\subsection{Numerical results for SdLBFGS on the RCV1 dataset}\label{sec:rcv1}

In this subsection, we compare SGD and SdLBFGS for solving \eqref{prob_SVM} on a real dataset: RCV1 \cite{LYRL2004}, which is a collection of newswire articles produced by Reuters in 1996-1997. In our tests, we used a subset \footnote{downloaded from http://www.cad.zju.edu.cn/home/dengcai/Data/TextData.html} of RCV1 used in \cite{CH2011} that contains 9625 articles with 29992 distinct words. The articles are classified into four categories ``C15'', ``ECAT'', ``GCAT'' and ``MCAT'', each with 2022, 2064, 2901 and 2638 articles respectively. We consider the binary classification problem of predicting whether or not an article is in the second and fourth category, i.e., the entry of each label vector is 1 if a given article appears in category ``MCAT'' or ``ECAT'', and -1 otherwise. We used 60\% of the articles (5776) as training data and the remaining 40\% (3849) as testing data.

In Figure \ref{fig2-RCV}, we compare SdLBFGS with various memory sizes and SGD on the RCV1 dataset. For SGD and SdLBFGS, we use the stepsize: $\alpha_k=10/k$ and the batch size $m=100$. We also used a second stepsize of $20/k$ for SGD. Note that the SNG computed via \eqref{sng} uses $N=3849$ testing data. The left figure shows that for the RCV1 data set, increasing the memory size improves the performance of SdLBFGS. The performance of SdLBFGS with memory sizes $p=10,20$ was similar, although for $p=10$ it was slightly better. The right figure also shows that larger memory sizes can achieve higher correct classification percentages.


{In Figure \ref{RCV-fig-delta}, we report the performance of SdLBFGS on RCV1 dataset with different $\delta$ used in \eqref{gama-k}. Similar to Figure \ref{fig-delta}, we see from Figure \ref{RCV-fig-delta} that SdLBFGS works best for small $\delta$ such as $\delta = 0.01, 0.1$ and $1$.
}


Figure \ref{fig3-RCV} compares SGD and SdLBFGS with different batch sizes. The stepsize of SGD and SdLBFGS was set to $\alpha_k=20/k$ and $\alpha_k=10/k$, respectively. The memory size of SdLBFGS was chosen as $p=10$. We tested SGD with batch size $m=1,50,100$, and SdLBFGS with batch size $m=1,50,75,100$. From Figure \ref{fig3-RCV} we can see that SGD performs worse than SdLBFGS. For SdLBFGS, from Figure \ref{fig3-RCV} (a) we observe that larger batch sizes give better results in terms of SNG. If we fix the total number of $\SFO$-calls to $2*10^4$, SdLBFGS with $m=1$ performs the worst among the different batch sizes and exhibits dramatic oscillation. The performance gets much better when the batch size becomes larger. In this set of tests, the performance with $m=50,75$ was slightly better than $m=100$. One possible reason is that for the same number of $\SFO$-calls, a smaller batch size leads to larger number of iterations and thus gives better results.


In Figure \ref{fig4-RCV}, we report the percentage of correctly classified data points for both SGD and SdLBFGS with different batch sizes. These results are consistent with the ones in Figure \ref{fig3-RCV}. Roughly speaking, the algorithm that gives a lower SNG leads to a higher percentage of correctly classified data points.


We also counted the number of steps taken by SdLBFGS in which $s_{k-1}^\top y_{k-1} < 0$. We again set the total number of iterations of SdLBFGS to 1000. For fixed batch size $m=50$, the average numbers of such steps over 10 runs of SdLBFGS were respectively equal to $(3,5,8,6,1)$ when the memory sizes were $(1,3,5,10,20)$. For fixed memory size $p=10$, the average numbers of such steps over 10 runs of SdLBFGS were respectively equal to $(308,6,2)$ when the batch sizes were $(1,50,75)$. This is qualitatively similar to our observations in Section \ref{sec:synthetic}, except that for the fixed batch size $m=50$, a fewer number of such steps were required by the SdLBFGS variants with memory sizes $p=1$ and $p=3$ compared with $p=5$ and $p=10$.

{
\subsection{Numerical results for SdLBFGS-VR on the RCV1 dataset}

In this subsection, we compare SdLBFGS-VR, SVRG \cite{SVR-nonconvex} and SdLBFGS for solving \eqref{prob_SVM-FS} on the RCV1 dataset with $T=5776$.
We here follow the same strategy as suggested in \cite{SVR-nonconvex} to initialize SdLBFGS-VR and SVRG. In particular, we run SGD with step size $1/t$ and batch size $20$ for $T$ iterations to get the initial point for SdLBFGS-VR and SVRG, where $t$ denotes the iteration counter. In all the tests, we use a constant step size $\alpha$ for both methods. The comparison results are shown in Figures \ref{RCV_diff-p}-\ref{RCV-VR-diff-beta}. In these figures, the ``SFO-calls'' in the $x$-axis includes both the number of stochastic gradients and $T$ gradient evaluations of the individual component functions when computing the full gradient $\nabla f(\tilde{x}_k)$ in each outer loop.

Figure \ref{RCV_diff-p} compares the performance of SdLBFGS-VR with different memory size $p$. It shows that the limited-memory BFGS improves performance, even when $p=1$. Moreover, larger memory size usually provides better performance, but the difference is not very significant.


Figure \ref{RCV-VR-diff-m} compares the performance of SdLBFGS-VR with different batch sizes $m$ and shows that SdLBFGS-VR is not very sensitive to $m$. In these tests, we always set $q=\lfloor T/m\rfloor$.


The impact of step size on SdLBFGS-VR and SVRG is shown in Figure \ref{RCV_diff-beta2} for three step sizes: $\alpha=0.1, 0.01$ and $0.001$. Clearly, for the same step size, SdLBFGS-VR gives better result than SVRG. From our numerical tests, we also observed that neither SdLBFGS-VR nor SVRG is stable when $\alpha\geq 1$.


In Figure \ref{RCV-VR-diff-beta}, we report the performance of SdLBFGS-VR with different constant step sizes $\alpha$, for $\alpha=0.1,0.01,0.001$ and SdLBFGS with different diminishing step sizes $\beta/k$ for $\beta = 10,1,0.1$, since SdLBFGS needs a diminishing step size to guarantee convergence. We see there that SdLBFGS-VR usually performs better than SdLBFGS. The performance of SdLBFGS with $\beta=10$ is in fact already very good, but still inferior to SdLBFGS-VR. This indicates that the variance reduction technique is indeed helpful.


}

{
\subsection{Numerical results of SdLBFGS-VR on MNIST dataset}\label{sec:mnist}

In this section, we report the numerical results of SdLBFGS-VR for solving a multiclass classification problem using neural networks on a standard testing data set MNIST\footnote{http://yann.lecun.com/exdb/mnist/}. All the experimental settings are the same as in \cite{SVR-nonconvex}. In particular, $T=60000$. The numerical results are reported in Figures \ref{MNIST_diff-p}-\ref{MNIST-VR-diff-beta}, and their purposes are the same as Figures \ref{RCV_diff-p}-\ref{RCV-VR-diff-beta}. From these figures, we have similar observations as those from Figures \ref{RCV_diff-p}-\ref{RCV-VR-diff-beta}.
Note that in Figures \ref{MNIST_diff-p}-\ref{MNIST-VR-diff-beta}, the ``SFO-calls'' in the $x$-axis again includes both the number of stochastic gradients and $T$ gradient evaluations of the individual component functions when computing the full gradient $\nabla f(\tilde{x}_k)$ in each outer loop.
}

%
%

%

\section{Conclusions}\label{sec:conclusions}

In this paper we proposed a general framework for stochastic quasi-Newton methods for nonconvex stochastic optimization. Global convergence, iteration complexity, and $\SFO$-calls complexity were analyzed under different conditions on the step size and the output of the algorithm. Specifically, a stochastic damped limited memory BFGS method was proposed, which falls under the proposed framework and does not generate $H_k$ explicitly. The damping technique was used to preserve the positive definiteness of $H_k$, without requiring the original problem to be convex. A variance reduced stochastic L-BFGS method was also proposed for solving the empirical risk minimization problem. Encouraging numerical results were reported for solving nonconvex classification problems using SVM and neural networks.

\section*{Acknowledgement}

The authors are grateful to two anonymous referees for their insightful comments and constructive suggestions that have improved the presentation of this paper greatly. The authors also thank Conghui Tan for helping conduct the numerical tests in Section \ref{sec:mnist}.


\bibliographystyle{plain}

\bibliography{All}

\begin{thebibliography}{10}

\bibitem{AllenHazan2016-nonconvex}
Zeyuan {Allen-Zhu} and Elad Hazan.
\newblock {Variance Reduction for Faster Non-Convex Optimization}.
\newblock In {\em ICML}, 2016.

\bibitem{Bach-sgd-2014}
F.~Bach.
\newblock Adaptivity of averaged stochastic gradient descent to local strong
  convexity for logistic regression.
\newblock {\em Journal of Machine Learning Research}, 15:595--627, 2014.

\bibitem{bach-nips-2013}
F.~Bach and E.~Moulines.
\newblock Non-strongly-convex smooth stochastic approximation with convergence
  rate $\mathcal{O}(1/n)$.
\newblock In {\em NIPS}, 2013.

\bibitem{bct06}
F.~Bastin, C.~Cirillo, and P.~L. Toint.
\newblock Convergence theory for nonconvex stochastic programming with an
  application to mixed logit.
\newblock {\em Math. Program.}, 108:207--234, 2006.

\bibitem{BBG-09}
A.~Bordes, L.~Bottou, and P.~Gallinari.
\newblock {SGD-QN}: Careful quasi-{Newton} stochastic gradient descent.
\newblock {\em J. Mach. Learn. Res.}, 10:1737--1754, 2009.

\bibitem{B98}
L.~Bottou.
\newblock Online algorithms and stochastic approximations.
\newblock {\em Online Learning and Neural Networks}, Edited by David Saad,
  Cambridge University Press, Cambridge, UK, 1998.

\bibitem{bbt00}
D.~Brownstone, D.~S. Bunch, and K.~Train.
\newblock Joint mixed logit models of stated and revealed preferences for
  alternative-fuel vehicles.
\newblock {\em Transport. Res. B}, 34(5):315--338, 2000.

\bibitem{broyden-bfgs-1970}
C.~G. Broyden.
\newblock The convergence of a calss of double-rank minimization algorithms.
\newblock {\em J. Inst. Math. Appl.}, 6(1):76--90, 1970.

\bibitem{BCNN}
R.H. Byrd, G.~Chin, W.~Neveitt, and J.~Nocedal.
\newblock On the use of stochastic hessian information in optimization methods
  for machine learning.
\newblock {\em SIAM J. Optim.}, 21(3):977--995, 2011.

\bibitem{BHNS-14}
R.H. Byrd, S.L. Hansen, J.~Nocedal, and Y.~Singer.
\newblock A stochastic quasi-{Newton} method for large-scale optimization.
\newblock {\em SIAM J. Optim.}, 26(2):1008--1031, 2016.

\bibitem{CH2011}
D.~Cai and X.~He.
\newblock Manifold adaptive experimental design for text categorization.
\newblock {\em IEEE Transactions on Knowledge and Data Engineering},
  4:707--719, 2012.

\bibitem{C54}
K.~L. Chung.
\newblock On a stochastic approximation method.
\newblock {\em Annals of Math. Stat.}, pages 463--483, 1954.

\bibitem{dl13}
C.~D. Dang and G.~Lan.
\newblock Stochastic block mirror descent methods for nonsmooth and stochastic
  optimization.
\newblock {\em SIAM Journal on Optimization}, 25(2):856--881, 2015.

\bibitem{SAGA-2014}
A.~Defazio, F.~Bach, and S.~Lacoste-Julien.
\newblock {SAGA}: A fast incremental gradient method with support for
  non-strongly convex composite objectives.
\newblock In {\em NIPS}, 2014.

\bibitem{DHS2011}
J.~C. Duchi, E.~Hazan, and Y.~Singer.
\newblock Adaptive subgradient methods for online learning and stochastic
  optimization.
\newblock {\em J. Mach. Learn. Res.}, 999999:2121--2159, 2011.

\bibitem{Durrett-10}
R.~Durrett.
\newblock {\em Probability: Theory and Examples}.
\newblock Cambridge University Press, London, 2010.

\bibitem{e83}
Y.~Ermoliev.
\newblock Stochastic quasigradient methods and their application to system
  optimization.
\newblock {\em Stochastics}, 9:1--36, 1983.

\bibitem{Fletcher-bfgs-1970}
R.~Fletcher.
\newblock A new approach to variable metric algorithms.
\newblock {\em The Computer Journal}, 13(3):317--322, 1970.

\bibitem{g78}
A.~A. Gaivoronski.
\newblock Nonstationary stochastic programming problems.
\newblock {\em Kibernetika}, 4:89--92, 1978.

\bibitem{gl12}
S.~Ghadimi and G.~Lan.
\newblock Optimal stochastic approximation algorithms for strongly convex
  stochastic composite optimization, {I}: a generic algorithmic framework.
\newblock {\em SIAM J. Optim.}, 22:1469--1492, 2012.

\bibitem{gl13}
S.~Ghadimi and G.~Lan.
\newblock Stochastic first- and zeroth-order methods for nonconvex stochastic
  programming.
\newblock {\em SIAM J. Optim.}, 15(6):2341--2368, 2013.

\bibitem{gl132}
S.~Ghadimi and G.~Lan.
\newblock Accelerated gradient methods for nonconvex nonlinear and stochastic
  programming.
\newblock {\em Mathematical Programming}, 156(1):59--99, 2016.

\bibitem{glz13}
S.~Ghadimi, G.~Lan, and H.~Zhang.
\newblock Mini-batch stochastic approximation methods for nonconvex stochastic
  composite optimization.
\newblock {\em Math. Program.}, 155(1):267--305, 2016.

\bibitem{Goldfarb-bfgs-1970}
D.~Goldfarb.
\newblock A family of variable metric updates derived by variational means.
\newblock {\em Math. Comput.}, 24(109):23--26, 1970.

\bibitem{Gower-icml-2016}
R.~M. Gower, D.~Goldfarb, and P.~Richt\'arik.
\newblock Stochastic block {BFGS}: squeezing more curvature out of data.
\newblock In {\em ICML}, 2016.

\bibitem{hg03}
D.~A. Hensher and W.~H. Greene.
\newblock The mixed logit model: The state of practice.
\newblock {\em Transportation}, 30(2):133--176, 2003.

\bibitem{SPVR}
R.~Johnson and T.~Zhang.
\newblock Accelerating stochastic gradient descent using predictive variance
  reduction.
\newblock {\em NIPS}, 2013.

\bibitem{jntv05}
A.~Juditsky, A.~Nazin, A.~B. Tsybakov, and N.~Vayatis.
\newblock Recursive aggregation of estimators via the mirror descent algorithm
  with average.
\newblock {\em Problems of Information Transmission}, 41(4):368--384, 2005.

\bibitem{jrt08}
A.~Juditsky, P.~Rigollet, and A.~B. Tsybakov.
\newblock Learning by mirror averaging.
\newblock {\em Annals of Stat.}, 36:2183--2206, 2008.

\bibitem{l12}
G.~Lan.
\newblock An optimal method for stochastic composite optimization.
\newblock {\em Math. Program.}, 133(1):365--397, 2012.

\bibitem{lns12}
G.~Lan, A.~S. Nemirovski, and A.~Shapiro.
\newblock Validation analysis of mirror descent stochastic approximation
  method.
\newblock {\em Math. Pogram.}, 134:425--458, 2012.

\bibitem{bach-nips-2012}
N.~Le~Roux, M.~Schmidt, and F.~Bach.
\newblock A stochastic gradient method with an exponential convergence rate for
  strongly-convex optimization with finite training sets.
\newblock In {\em NIPS}, 2012.

\bibitem{LYRL2004}
D.~Lewis, Y.~Yang, T.~Rose, and F.~Li.
\newblock {RCV1}: A new benchmark collection for text categorization research.
\newblock {\em Journal of Mach. Learn. Res.}, 5(361-397), 2004.

\bibitem{Liu-Nocedal-89}
D.~C. Liu and J.~Nocedal.
\newblock On the limited memory {BFGS} method for large scale optimization.
\newblock {\em Math. Program., Ser. B}, 45(3):503--528, 1989.

\bibitem{LMH2015}
A.~Lucchi, B.~McWilliams, and T.~Hofmann.
\newblock A variance reduced stochastic {N}ewton method.
\newblock {\em preprint available at http://arxiv.org/abs/1503.08316}, 2015.

\bibitem{mbps09}
J.~Mairal, F.~Bach, J.~Ponce, and G.~Sapiro.
\newblock Online dictionary learning for sparse coding.
\newblock In {\em ICML}, 2009.

\bibitem{Mason-nips-1999}
L.~Mason, J.~Baxter, P.~Bartlett, and M.~Frean.
\newblock Boosting algorithms as gradient descent in function space.
\newblock In {\em NIPS}, volume~12, pages 512--518, 1999.

\bibitem{mr10}
A.~Mokhtari and A.~Ribeiro.
\newblock {RES}: {R}egularized stochastic {BFGS} algorithm.
\newblock {\em IEEE Trans. Signal Process.}, 62(23):6089--6104, 2014.

\bibitem{mr14-2}
A.~Mokhtari and A.~Ribeiro.
\newblock Global convergence of online limited memory {BFGS}.
\newblock {\em J. Mach. Learn. Res.}, 16:3151--3181, 2015.

\bibitem{MNJ2015}
P.~Moritz, R.~Nishihara, and M.I. Jordan.
\newblock A linearly-convergent stochasitic {L}-{BFGS} algorithm.
\newblock In {\em AISTATS}, pages 249--258, 2016.

\bibitem{njls09}
A.~S. Nemirovski, A.~Juditsky, G.~Lan, and A.~Shapiro.
\newblock Robust stochastic approximation approach to stochastic programming.
\newblock {\em SIAM J. Optim.}, 19:1574--1609, 2009.

\bibitem{Nesterov-1983}
Y.~E. Nesterov.
\newblock A method for unconstrained convex minimization problem with the rate
  of convergence $\mathcal{O}(1/k^2)$.
\newblock {\em Dokl. Akad. Nauk SSSR}, 269:543--547, 1983.

\bibitem{NesterovConvexBook2004}
Y.~E. Nesterov.
\newblock {\em Introductory lectures on convex optimization: A basic course}.
\newblock Applied Optimization. Kluwer Academic Publishers, Boston, MA, 2004.

\bibitem{p90}
B.~T. Polyak.
\newblock New stochastic approximation type procedures.
\newblock {\em Automat. i Telemekh.}, 7:98--107, 1990.

\bibitem{pj92}
B.~T. Polyak and A.~B. Juditsky.
\newblock Acceleration of stochastic approximation by averaging.
\newblock {\em SIAM J. Control and Optim.}, 30:838--855, 1992.

\bibitem{SVR-nonconvex}
S.J. Reddi, A.~Hefny, S.~Sra, B.~P\'{o}cz\'{o}s, and A.~Smola.
\newblock Stochastic variance reduction for nonconvex optimization.
\newblock In {\em ICML}, 2016.

\bibitem{rm51}
H.~Robbins and S.~Monro.
\newblock A stochastic approximatin method.
\newblock {\em Annals of Math. Stat.}, 22:400--407, 1951.

\bibitem{RF-2010}
N.L. Roux and A.W. Fitzgibbon.
\newblock A fast natural {N}ewton method.
\newblock In {\em ICML}, pages 623--630, 2010.

\bibitem{rs86}
A.~Ruszczynski and W.~Syski.
\newblock A method of aggregate stochastic subgradients with on-line stepsize
  rules for convex stochastic programming problems.
\newblock {\em Math. Prog. Stud.}, 28:113--131, 1986.

\bibitem{s58}
J.~Sacks.
\newblock Asymptotic distribution of stochastic approximation.
\newblock {\em Annals of Math. Stat.}, 29:373--409, 1958.

\bibitem{SchYuGue07}
N.~N. Schraudolph, J.~Yu, and S.~G\"unter.
\newblock A stochastic quasi-{N}ewton method for online convex optimization.
\newblock In {\em AISTATS}, pages 436--443, 2007.

\bibitem{ShaiBen14}
S.~Shalev-Shwartz and S.~Ben-David.
\newblock {\em Understanding Machine Learning: From Theory to Algorithms}.
\newblock Cambridge University Press, 2014.

\bibitem{shamir-zhang-icml-2013}
O.~Shamir and T.~Zhang.
\newblock Stochastic gradient descent for non-smooth optimization: Convergence
  results and optimal averaging schemes.
\newblock In {\em ICML}, 2013.

\bibitem{Shanno-bfgs-1970}
D.~F. Shanno.
\newblock Conditioning of quasi-{N}ewton methods for function minimization.
\newblock {\em Math. Comput.}, 24(111):647--656, 1970.

\bibitem{WangMaYuan13}
X.~Wang, S.~Ma, and Y.~Yuan.
\newblock Penalty methods with stochastic approximation for stochastic
  nonlinear programming.
\newblock {\em Mathematics of Computation}, 2016.

\bibitem{xiao-zhang-siopt-2014}
L.~Xiao and T.~Zhang.
\newblock A proximal stochastic gradient method with progressive variance
  reduction.
\newblock {\em SIAM Journal on Optimization}, 24:2057--2075, 2014.

\end{thebibliography}

\newpage

\begin{figure}[H]
\centering{\vspace{-8mm}
\includegraphics[scale=0.4]{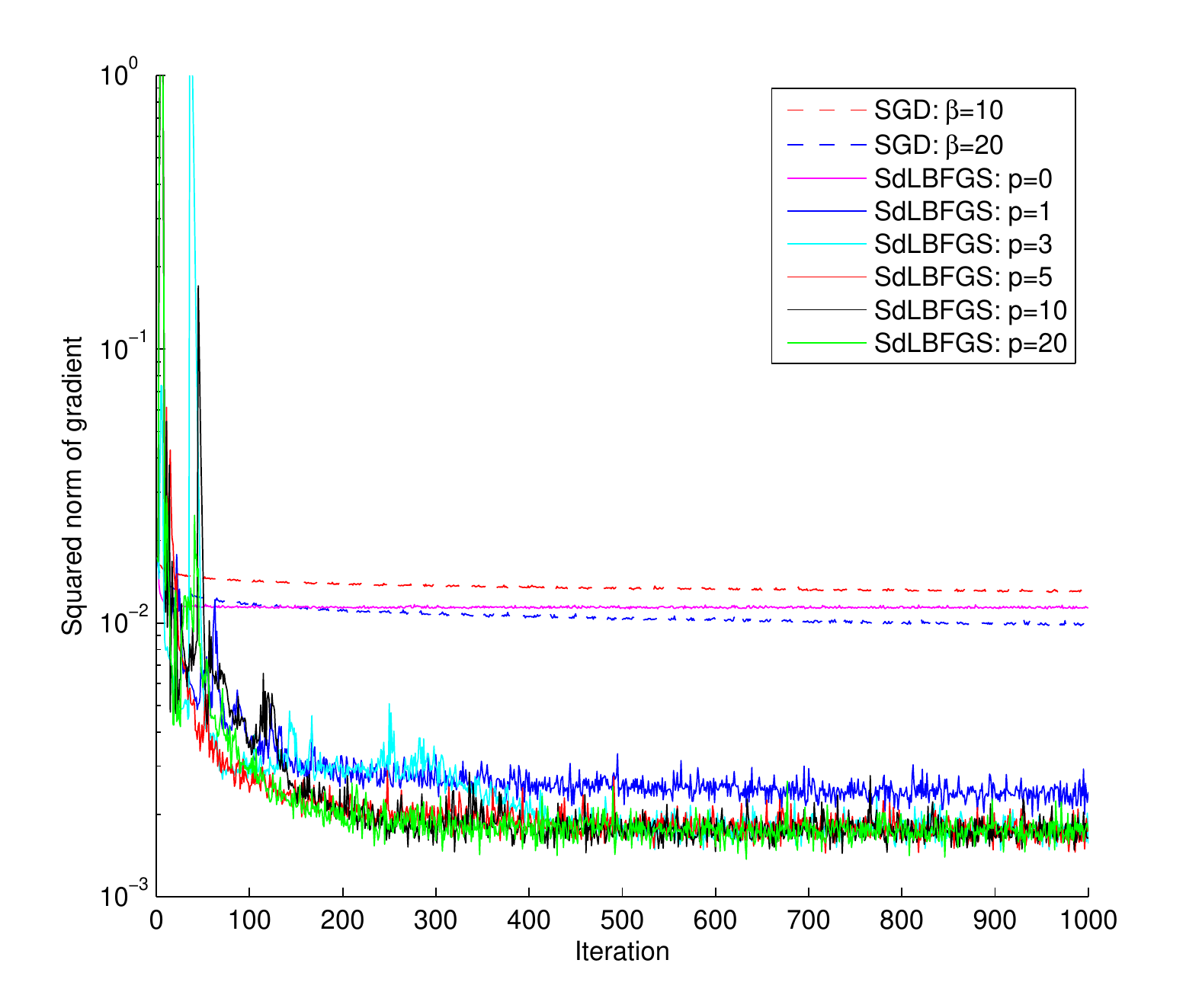}
 \includegraphics[scale=0.4]{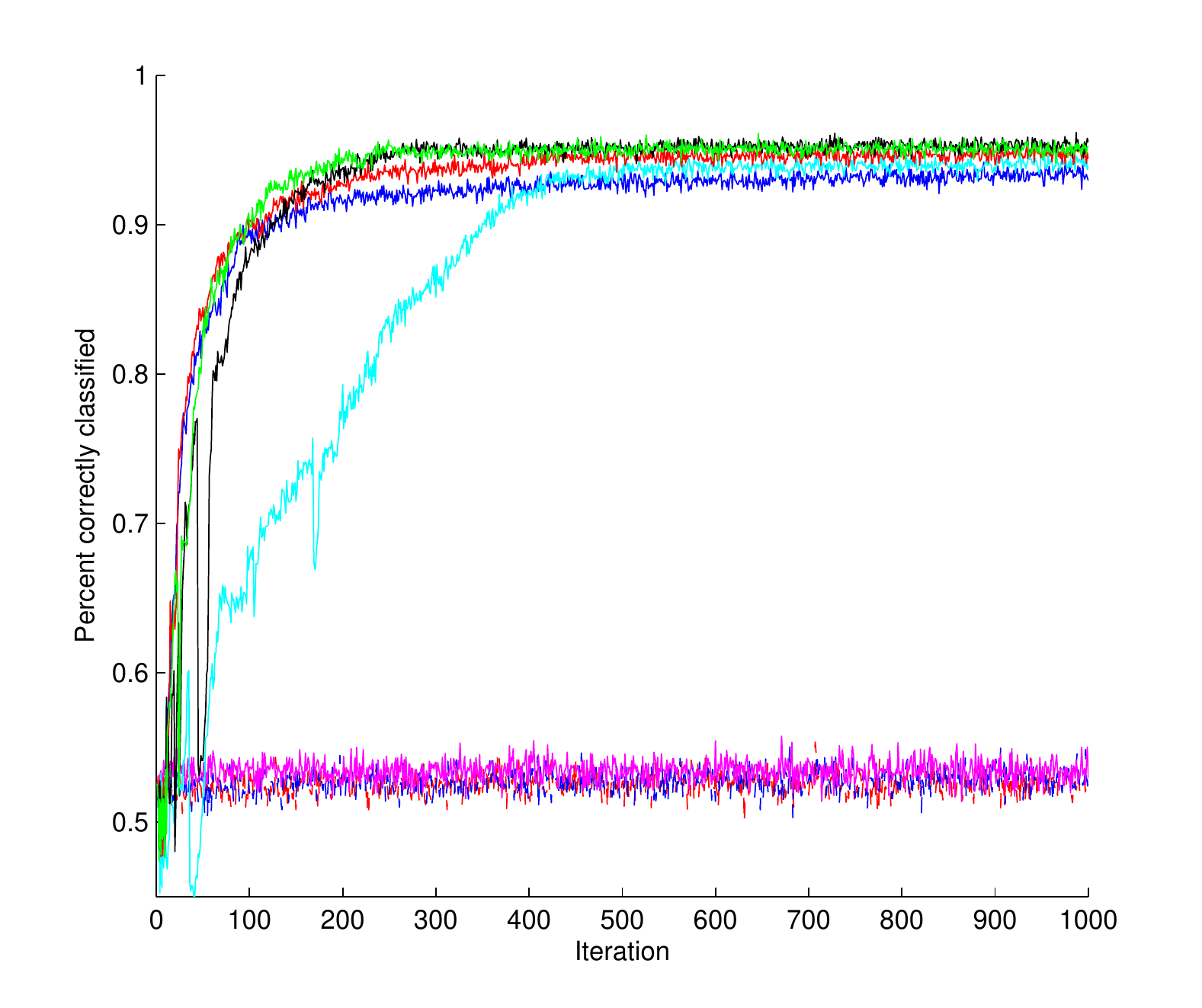}
}
\caption{Comparison on a synthetic data set of SGD and SdLBFGS variants with different memory size $p$, with respect to the squared norm of the gradient and the percent of correctly classified data, respectively. A stepsize of $\alpha_k = 10/k$ and batch size of $m=100$ was used for all SdLBFGS variants. Step sizes of $10/k$ and $20/k$ were used for SGD.}
\label{fig1}
\end{figure}

\begin{figure}[H]
\centering{\vspace{-8mm}
 \includegraphics[scale=0.4]{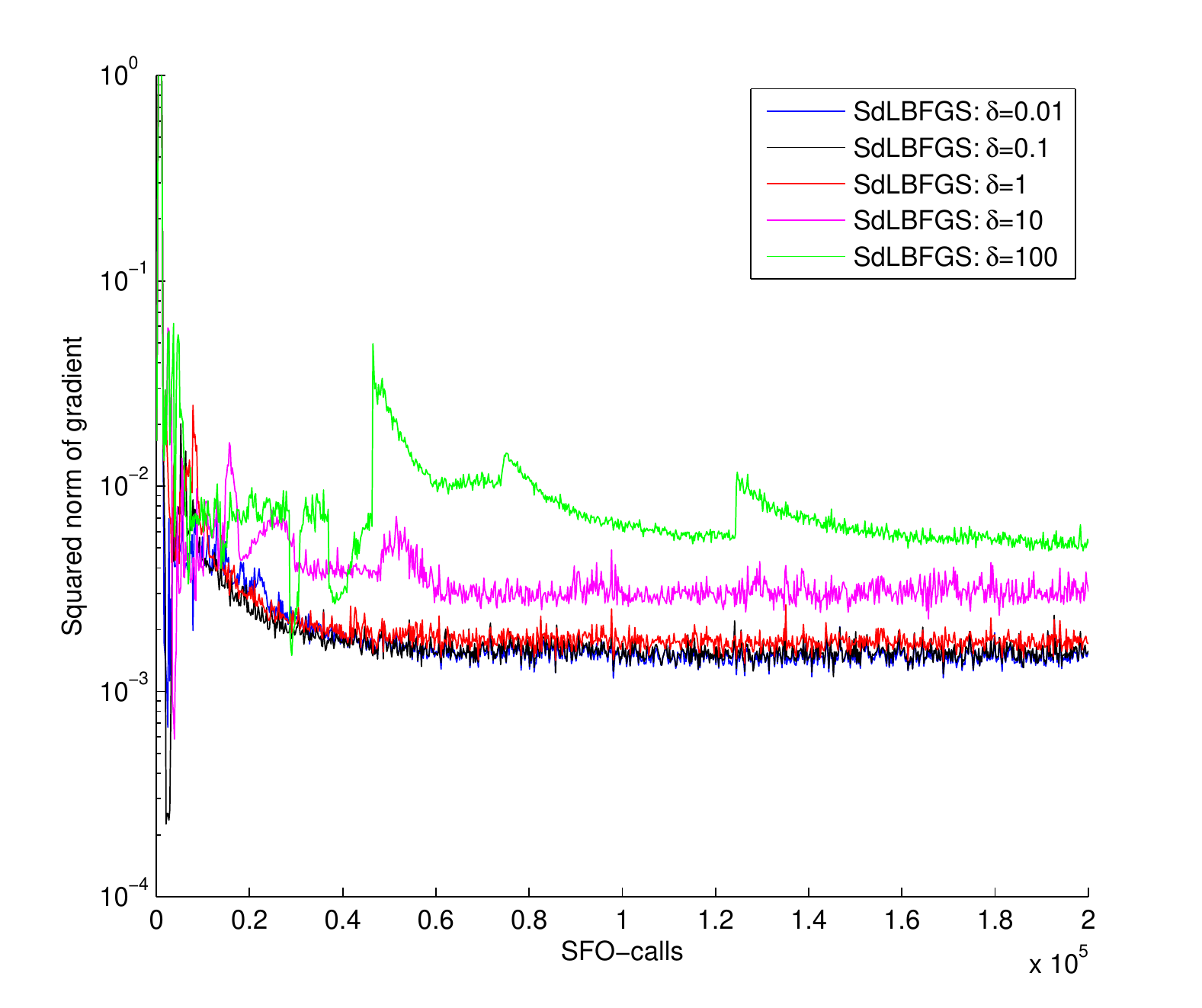}
\includegraphics[scale=0.4]{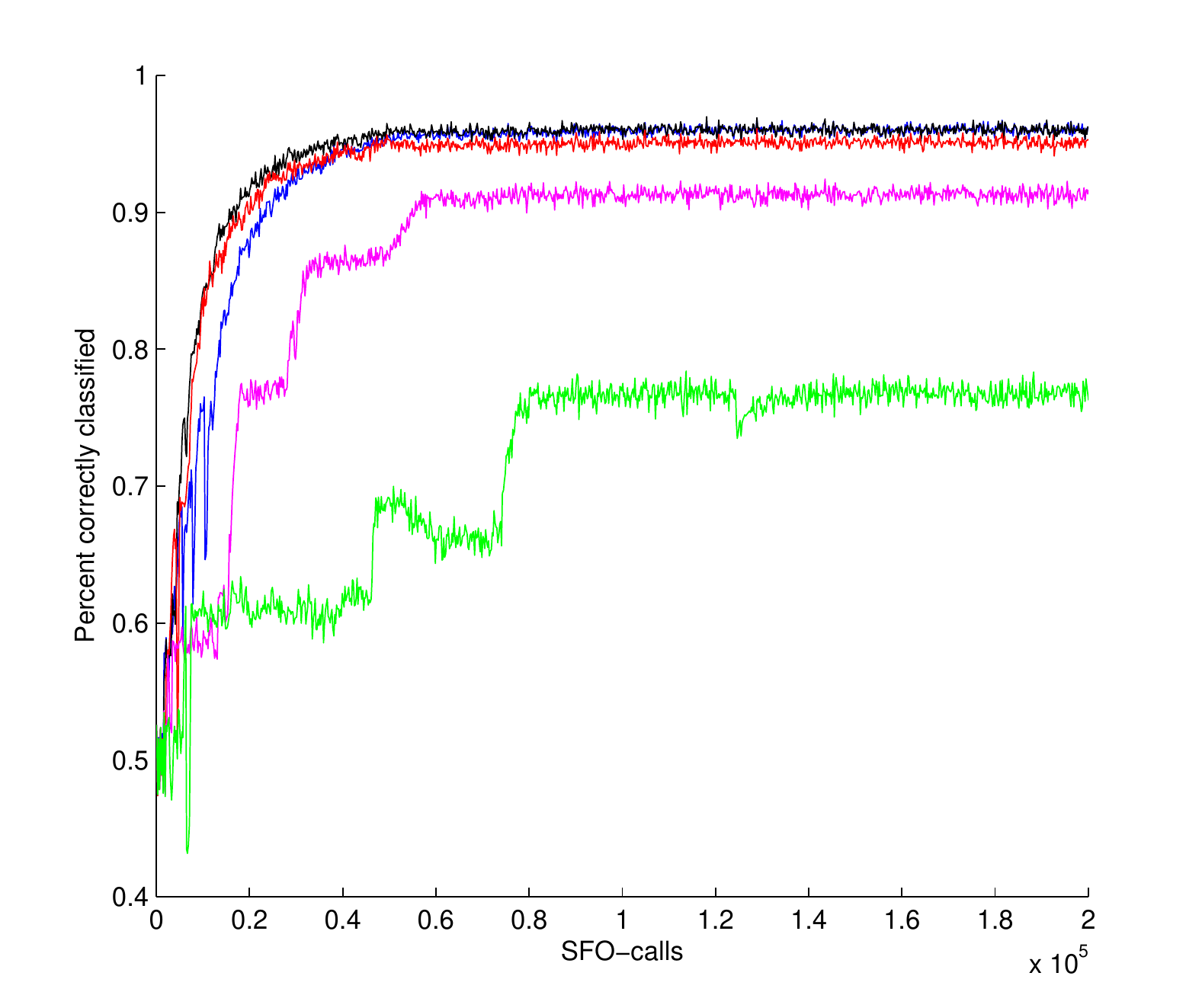}
}
\caption{Comparison on a synthetic data set of SdLBFGS variants with different initial Hessian approximations, by varying the value of $\delta$, with respect to the squared norm of the gradient and the percent of correctly classified data, respectively. A step size of $\alpha_k = 10/k$, memory size of $p=20$ and batch size of $m=100$ was used for all SdLBFGS variants.}
\label{fig-delta}
\end{figure}

\begin{figure}[H]
\centering{\vspace{-8mm}
\includegraphics[scale=0.4]{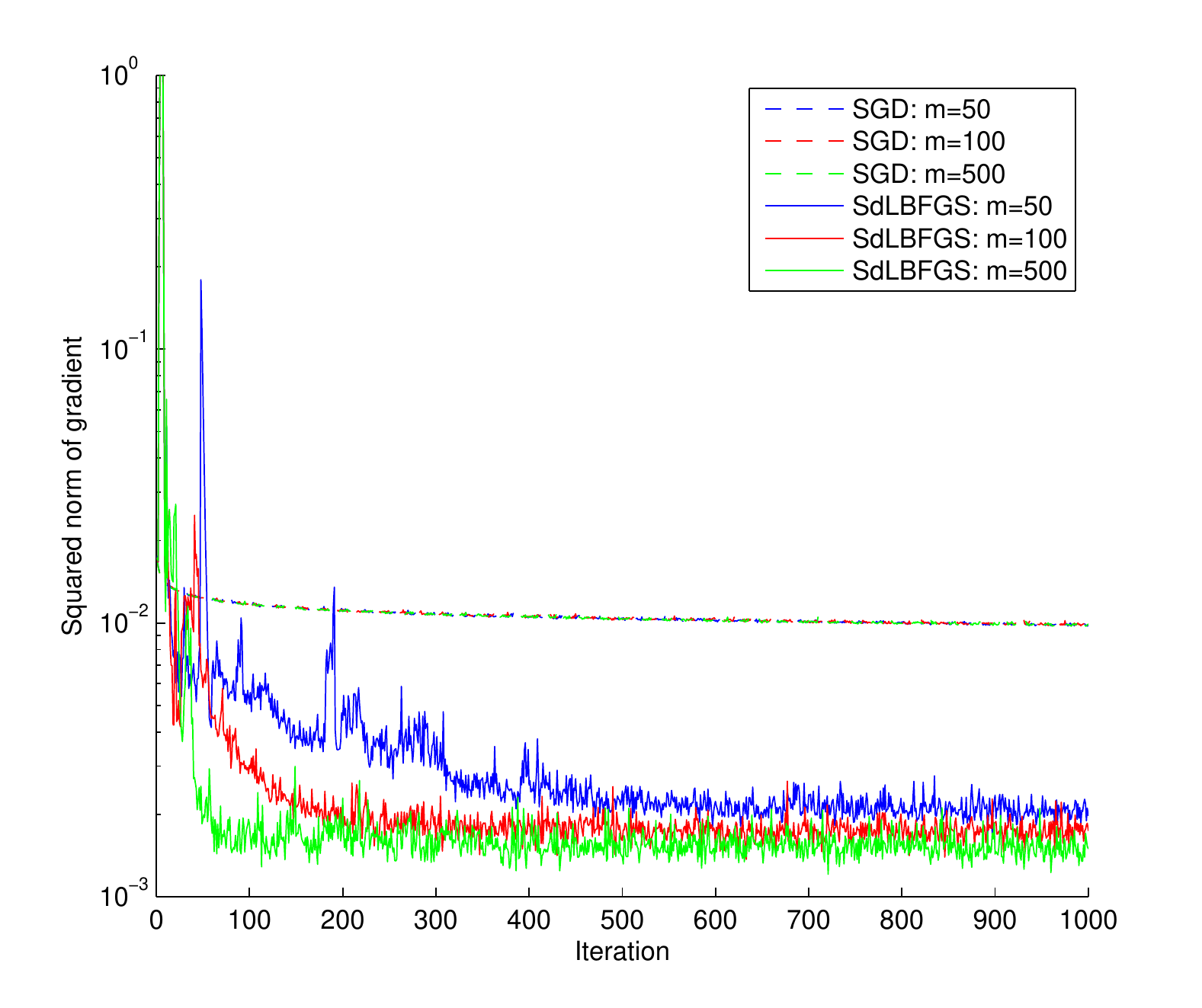}
 \includegraphics[scale=0.4]{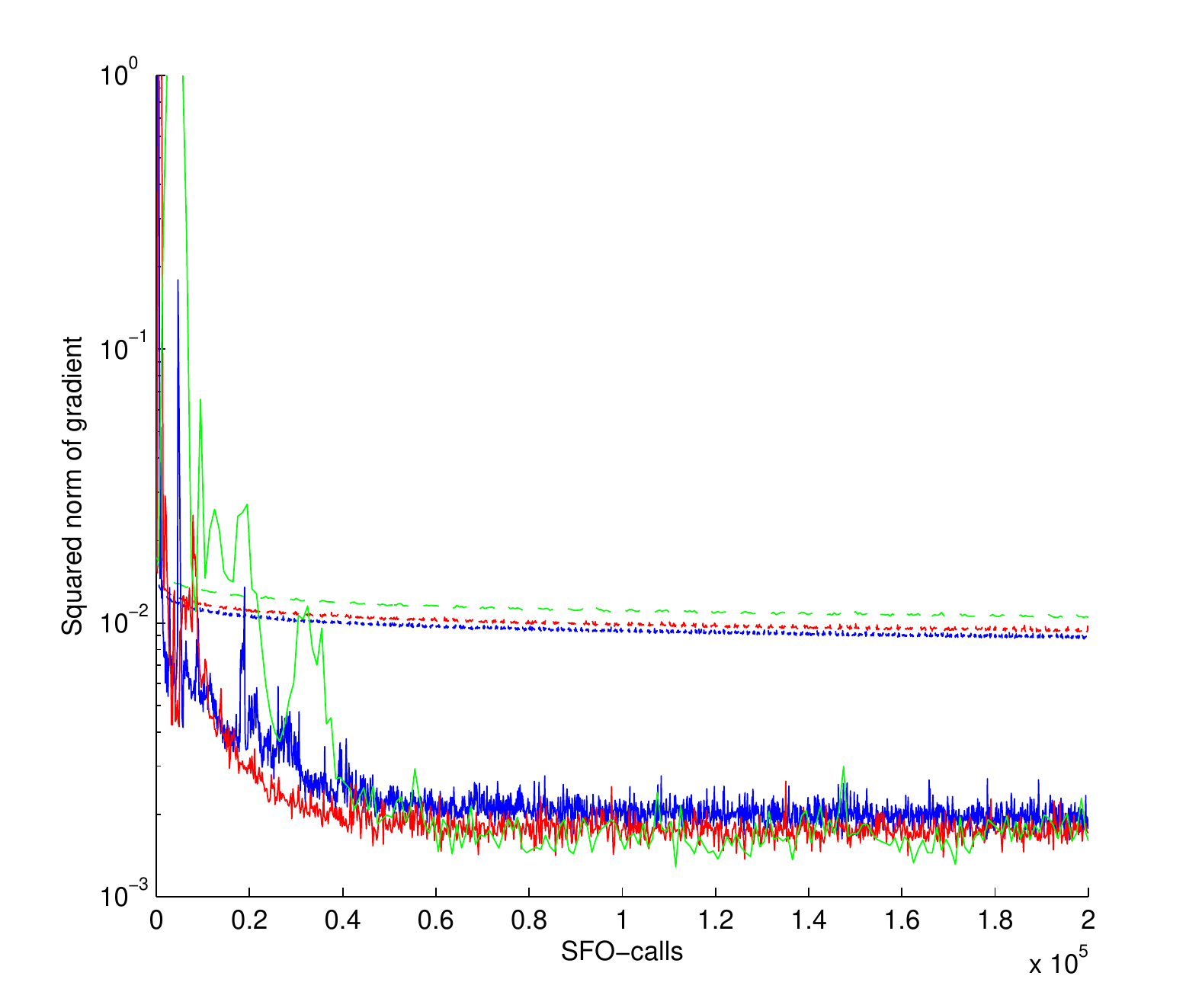}
}
\caption{Comparison of SGD and SdLBFGS variants with different batch size $m$ on a synthetic data set. The memory size of SdLBFGS was $p=20$ and the step size of SdLBFGS was $\alpha_k=10/k$, while the step size of SGD was $\alpha_k=20/k$.}
\label{fig1-batch}
\end{figure}

\begin{figure}[H]
\centering{\vspace{-8mm}
 \includegraphics[scale=0.4]{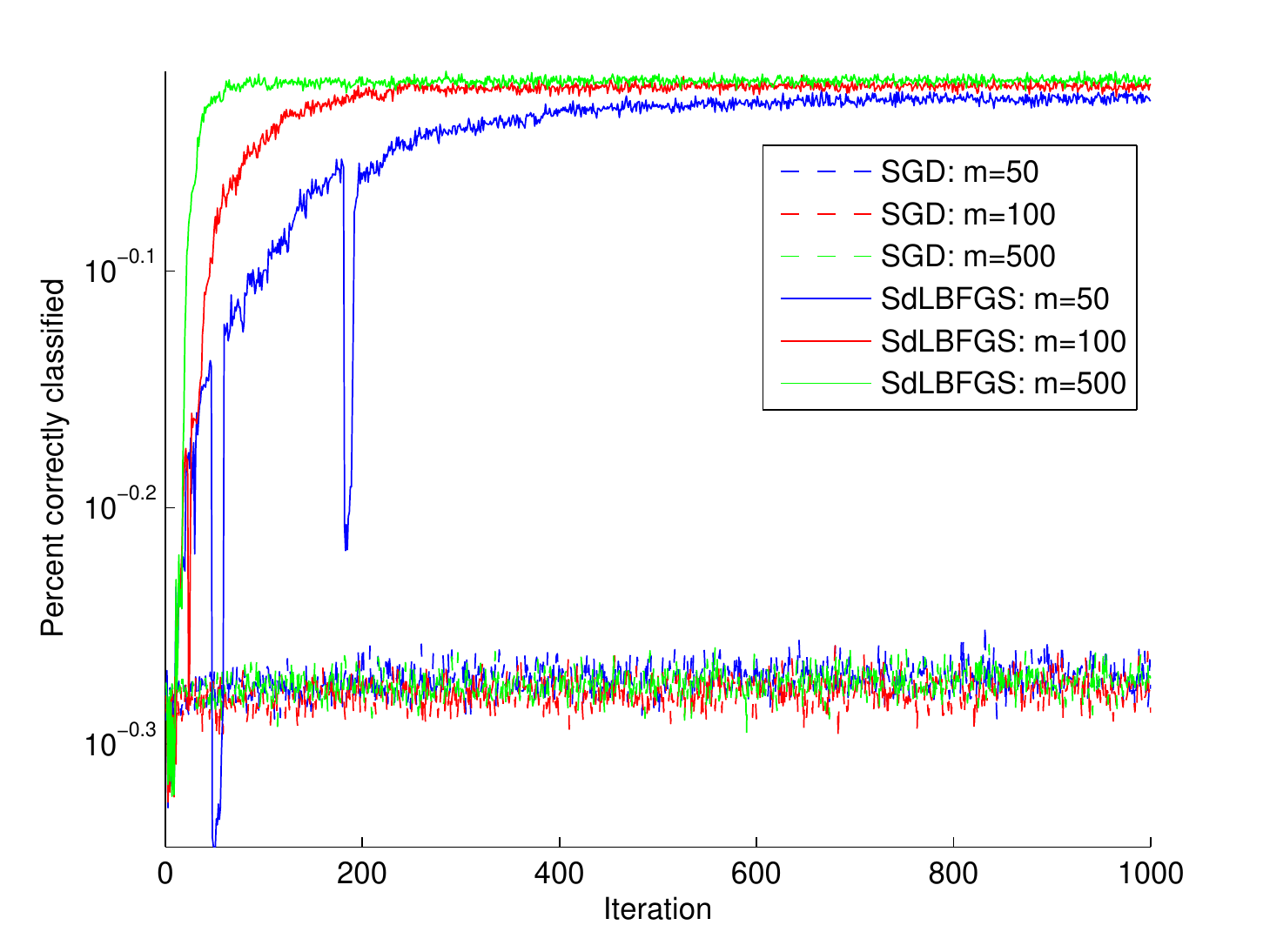}
}
\caption{Comparison of correct classification percentage on a synthetic dataset by SGD and SdLBFGS with different batch sizes. The memory size of SdLBFGS was $p=20$ and the stepsize of SdLBFGS was $\alpha_k = 10/k$. The step size of SGD was $\alpha_k=20/k$.}
\label{fig4-cls}
\end{figure}

\begin{figure}[H]
\centering{\vspace{-8mm}
 \includegraphics[scale=0.4]{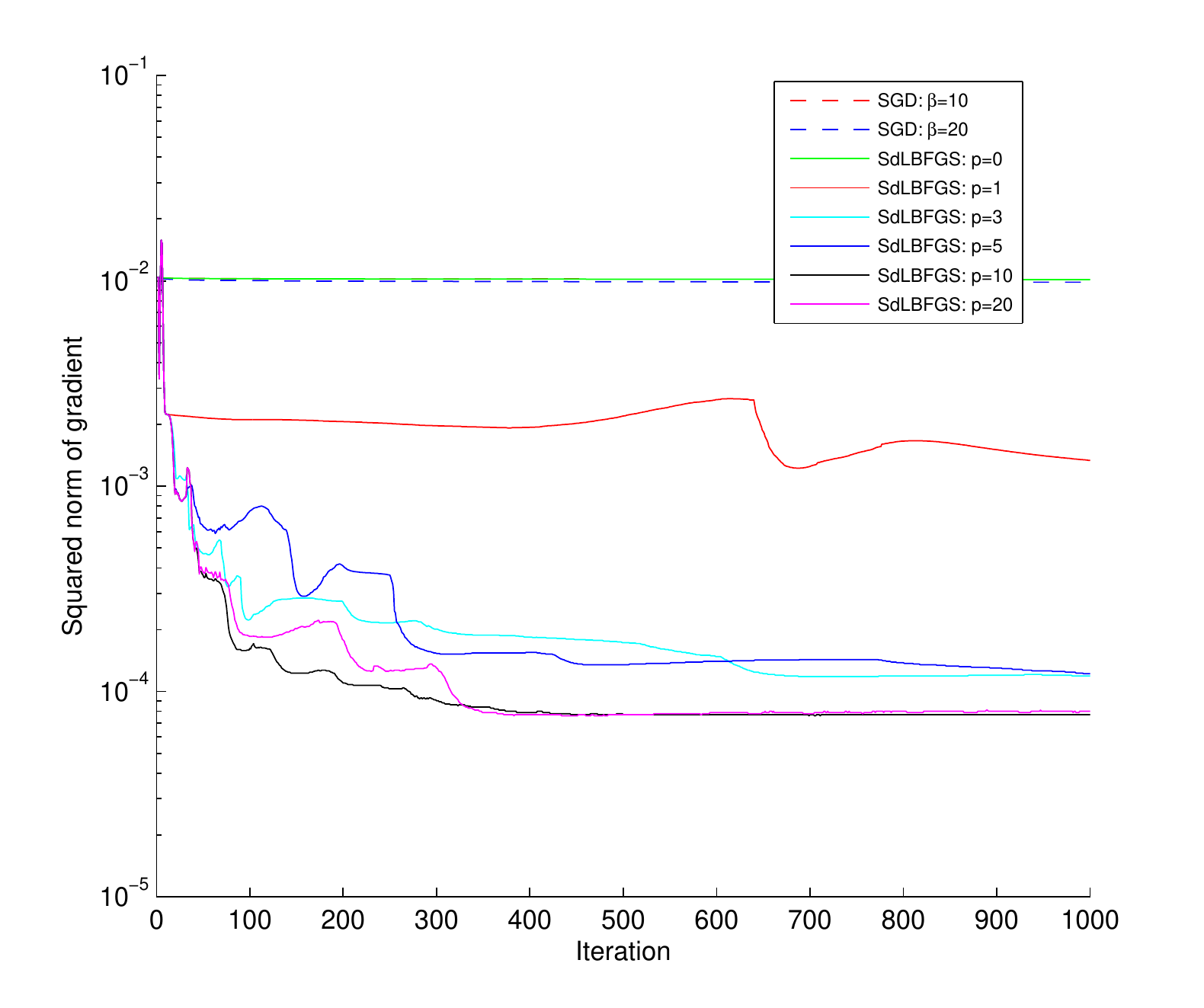}
 \includegraphics[scale=0.4]{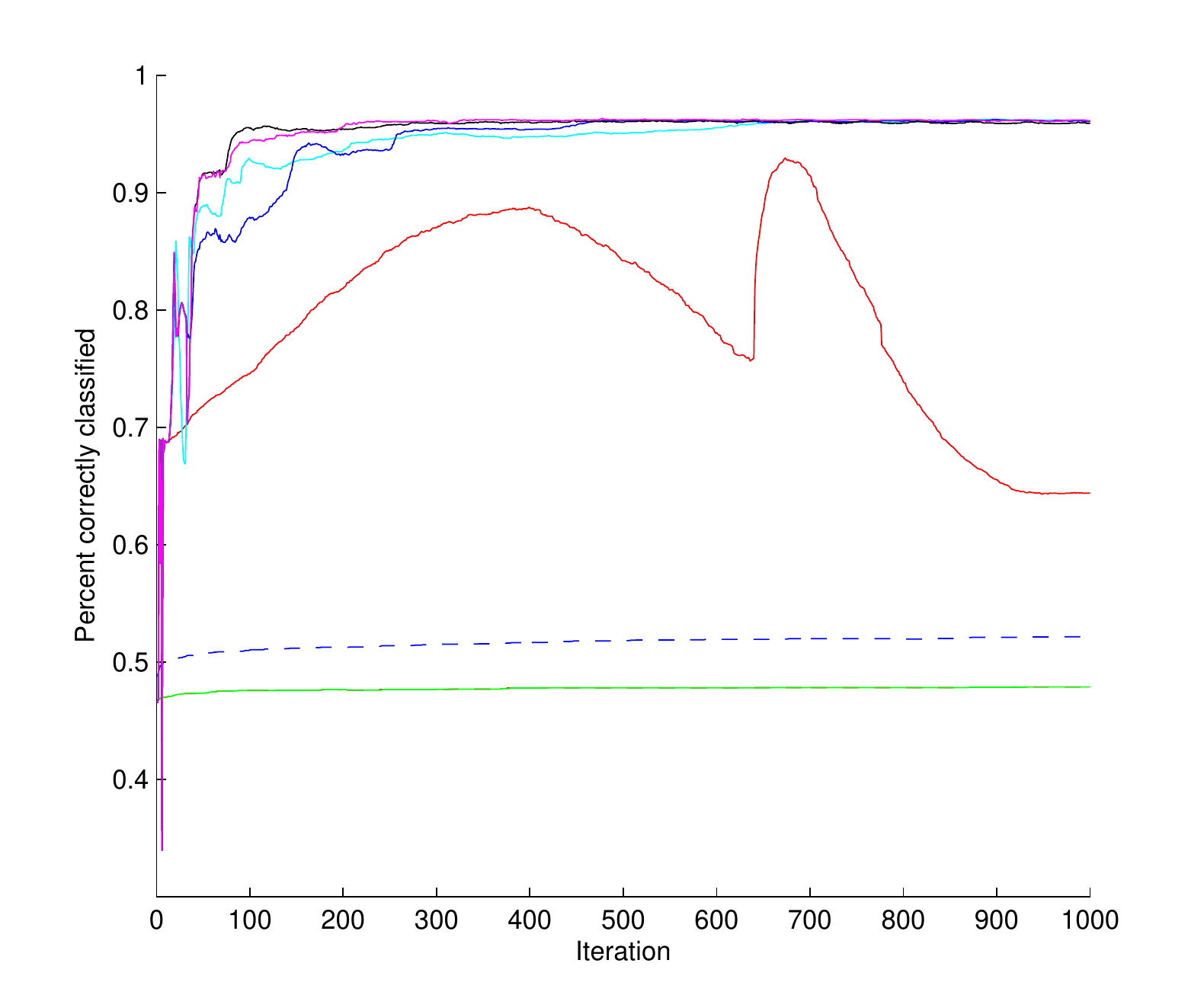}
}
\caption{Comparison of SdLBFGS variants with different memory sizes on the RCV1 dataset. The step size of SdLBFGS was $\alpha_k=10/k$ and the batch size was $m=100$.}
\label{fig2-RCV}
\end{figure}

\begin{figure}[H]
\centering{\vspace{-8mm}
 \includegraphics[scale=0.4]{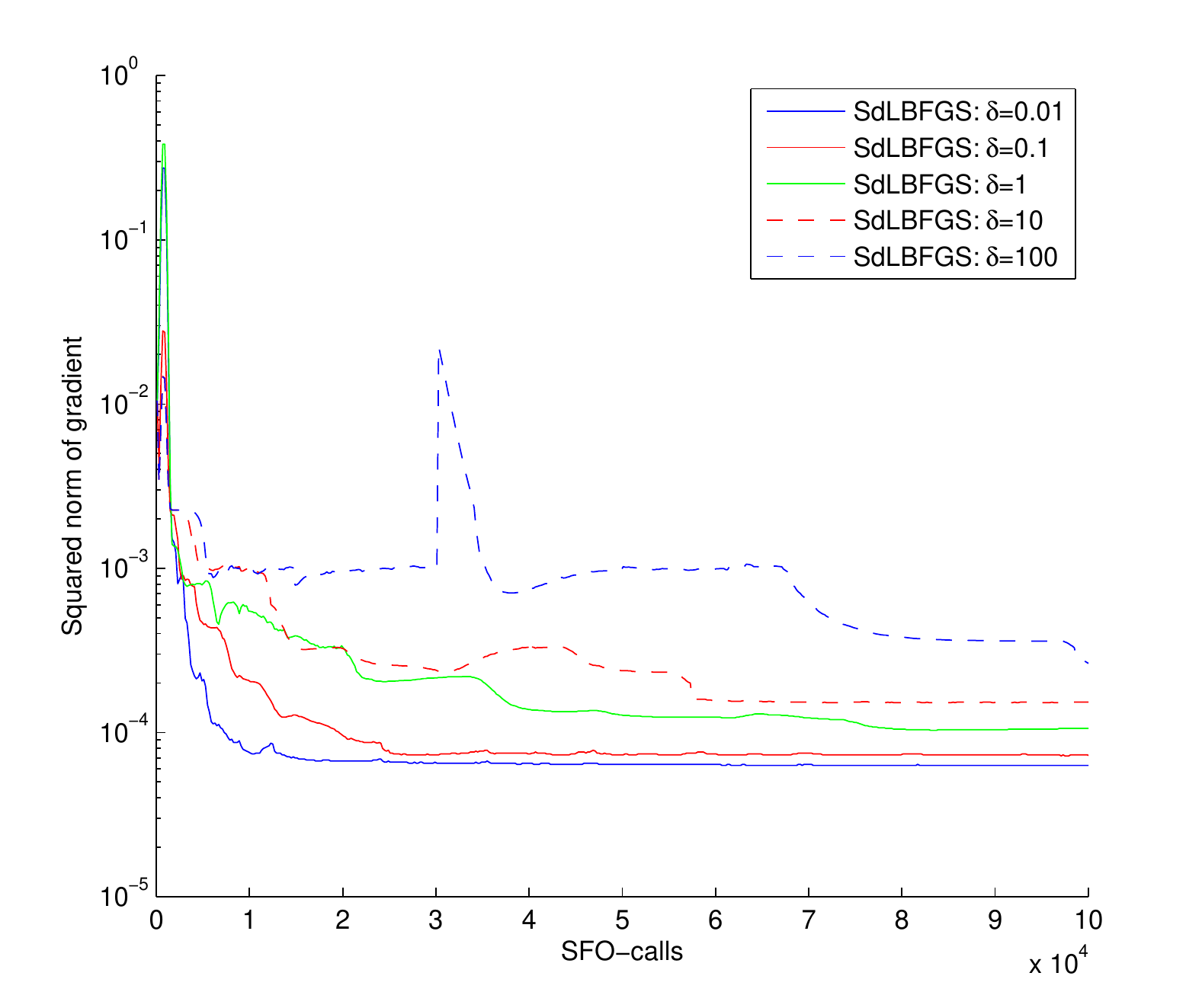}
\includegraphics[scale=0.4]{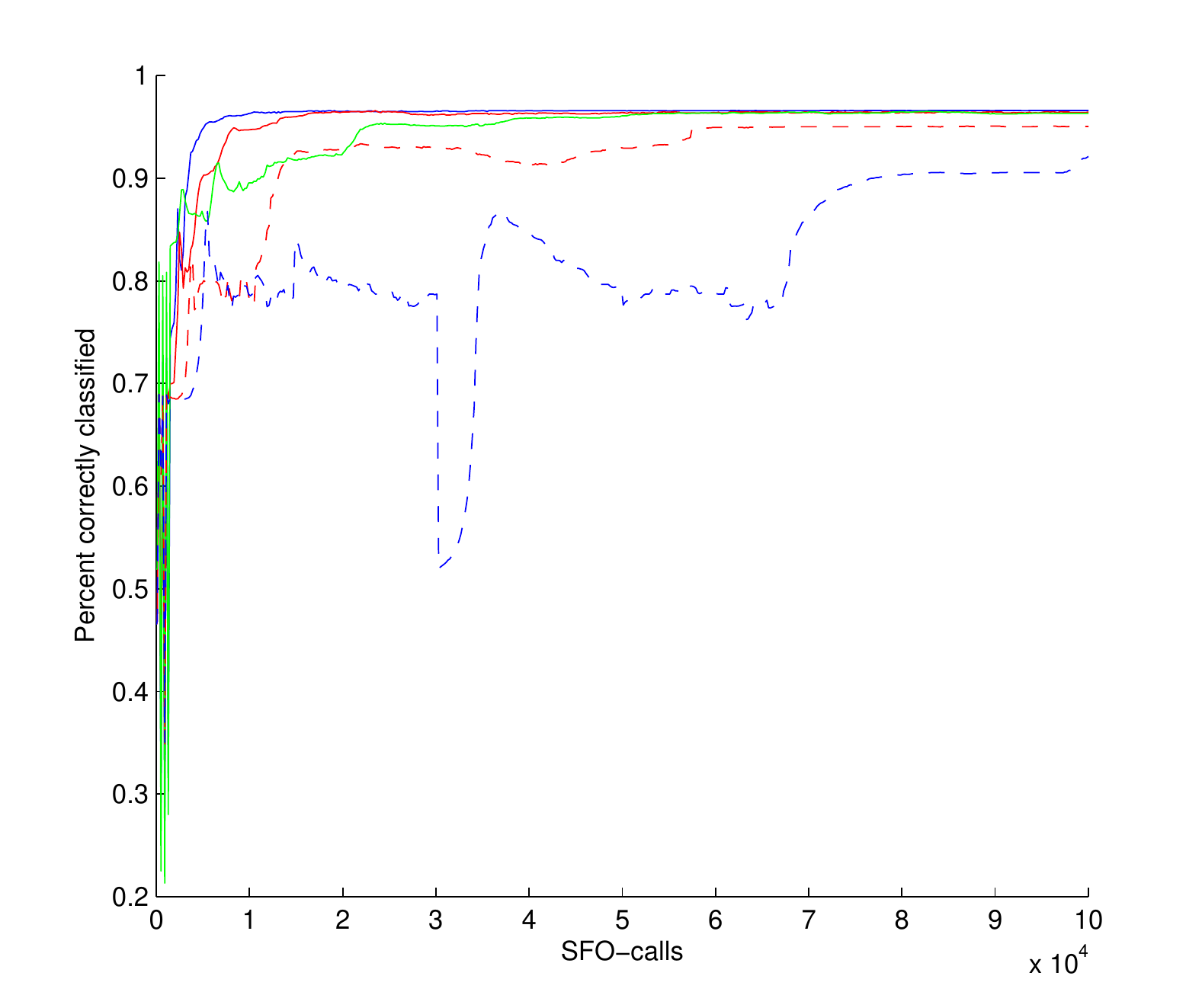}
}
\caption{Comparison on a RCV data set of SdLBFGS variants with different initial Hessian approximations, by varying the value of $\delta$, with respect to the squared norm of the gradient and the percent of correctly classified data, respectively. A step size of $\alpha_k = 10/k$, memory size of $p=10$ and batch size of $m=100$ was used for all SdLBFGS variants.}
\label{RCV-fig-delta}
\end{figure}

\begin{figure}[H]
\centering{\vspace{-8mm}
 \includegraphics[scale=0.4]{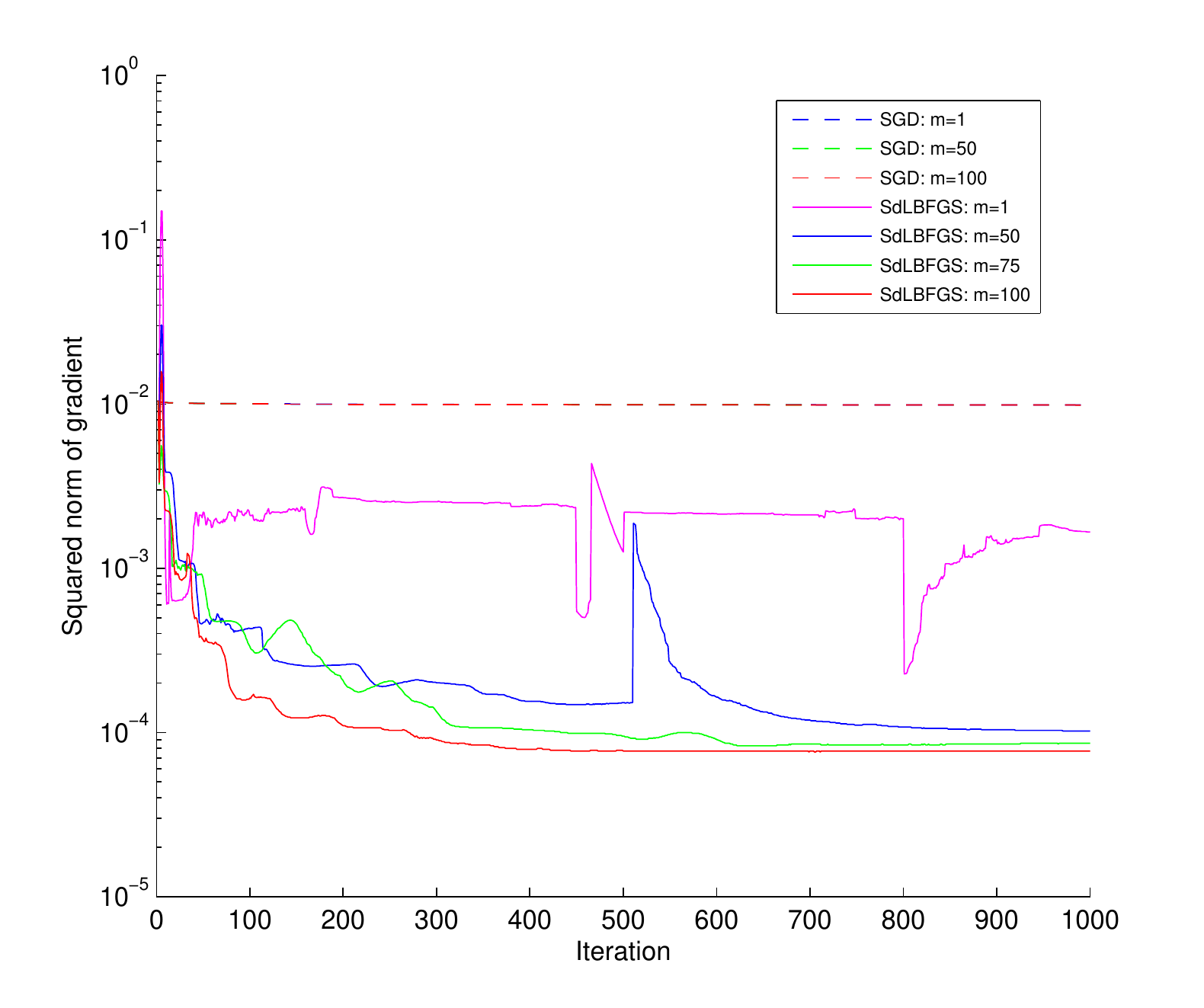}
 \includegraphics[scale=0.4]{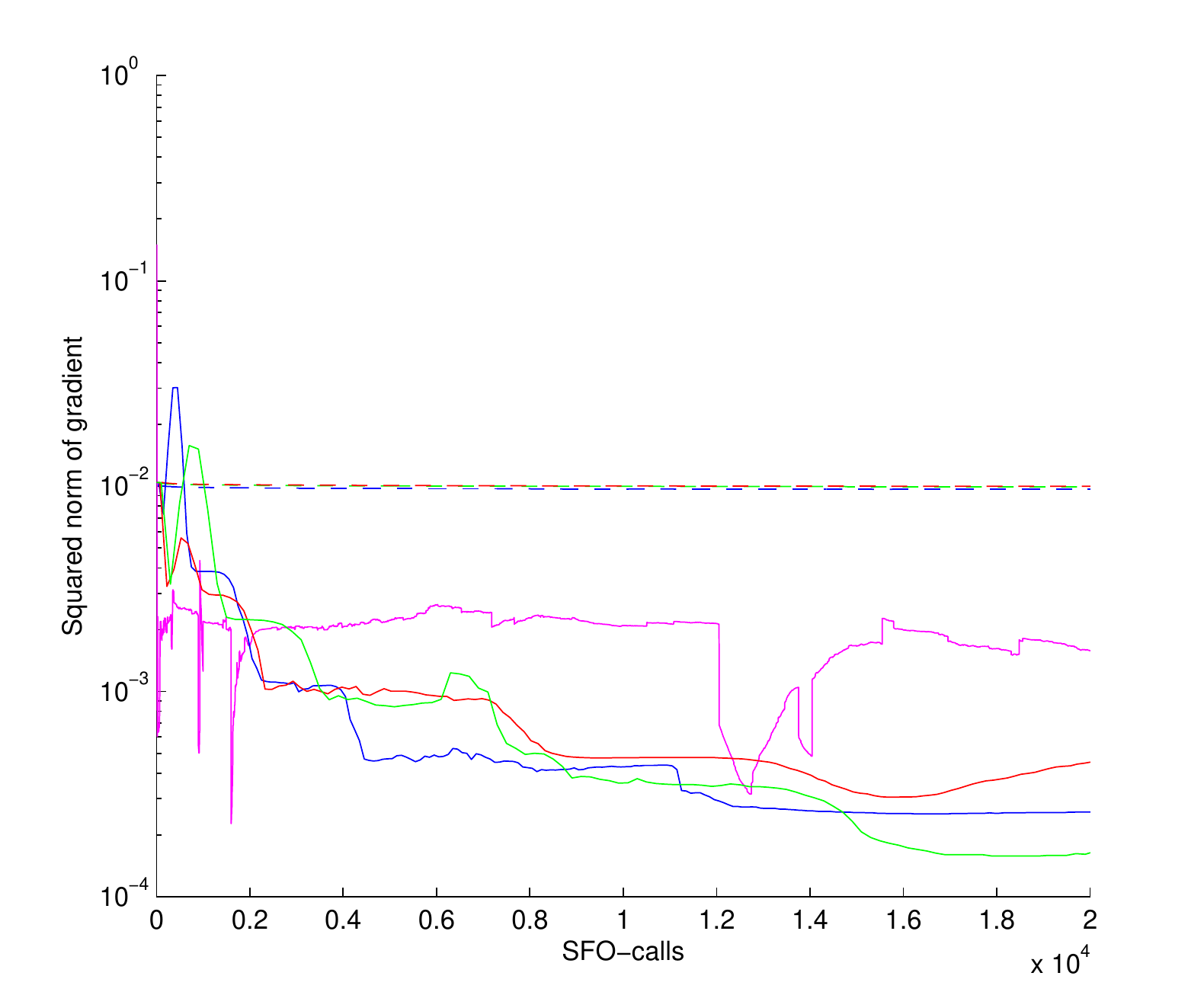}
}
\caption{Comparison of SGD and SdLBFGS with different batch sizes on the RCV1 dataset. For SdLBFGS the step size was $\alpha_k=10/k$ and the memory size was $p=10$. For SGD the step size was $\alpha_k=20/k$.}
\label{fig3-RCV}
\end{figure}

\begin{figure}[H]
\centering{\vspace{-8mm}
 \includegraphics[scale=0.4]{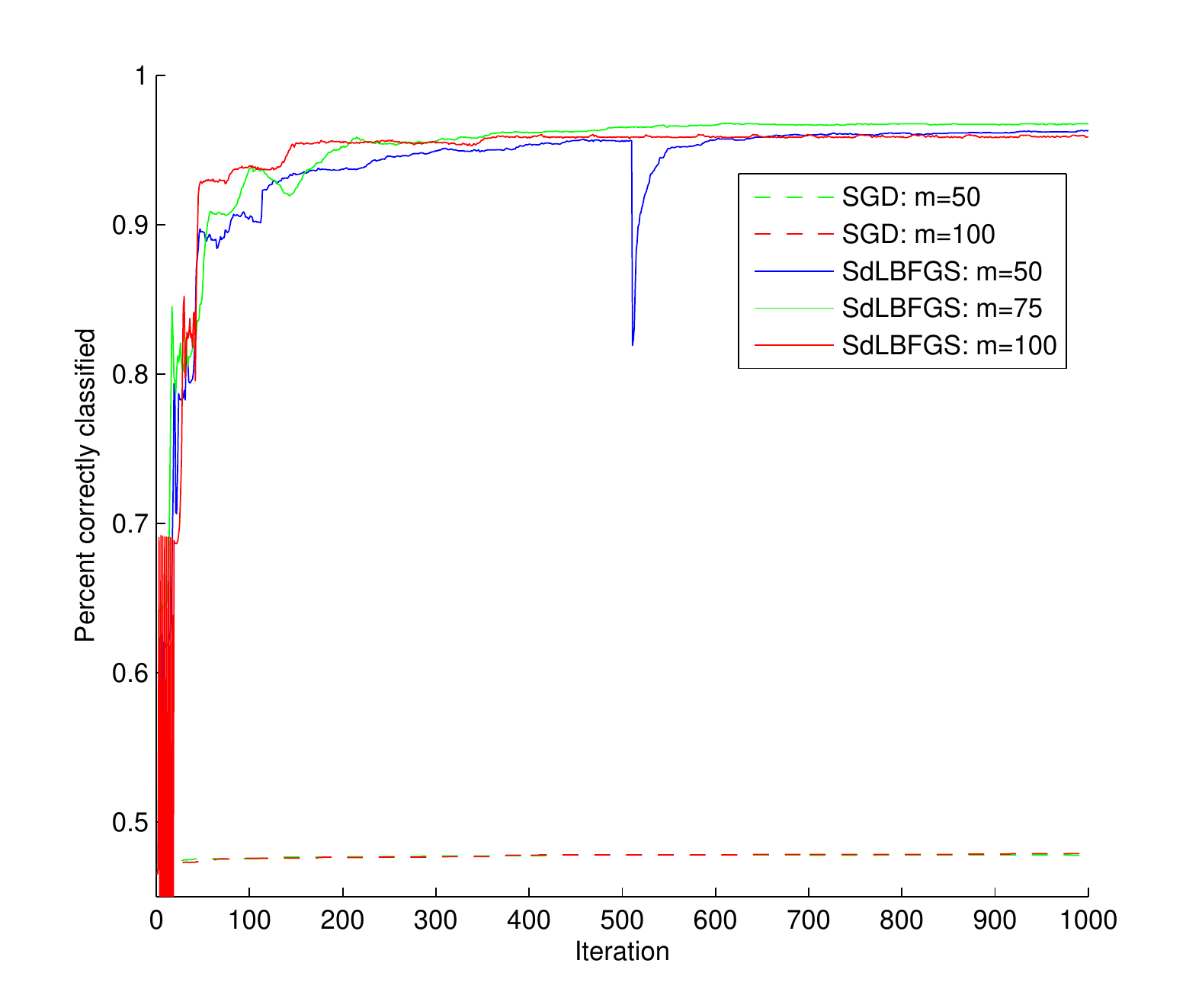}
}
\caption{Comparison of correct classification percentage by SGD and SdLBFGS with different batch sizes on the RCV1 dataset. For SdLBFGS the step size was $\alpha_k=10/k$ and the memory size was $p=10$. For SGD the step size was $\alpha_k=20/k$.}
\label{fig4-RCV}
\end{figure}

\begin{figure}[H]
\centering{\vspace{-8mm}
 \includegraphics[scale=0.4]{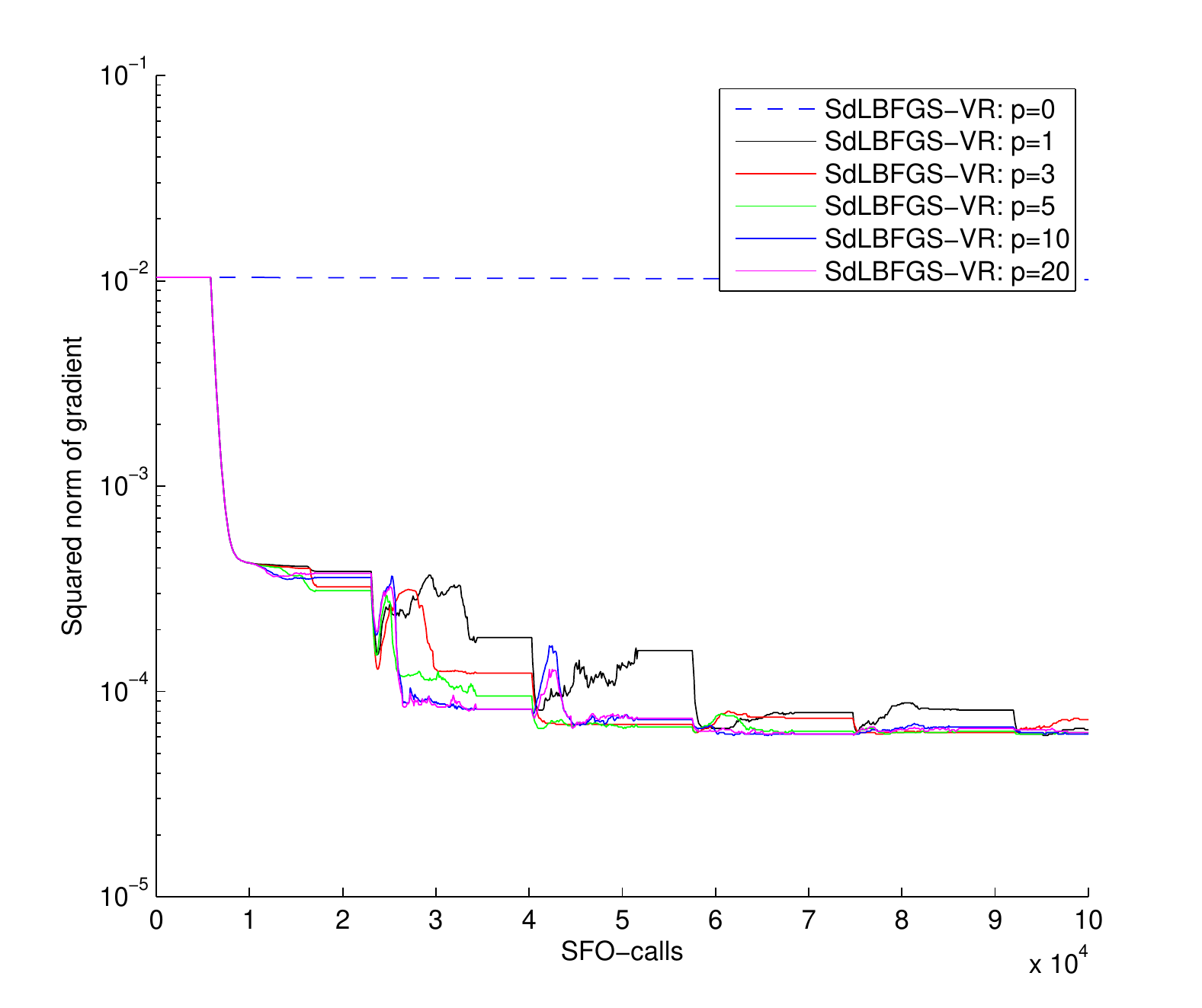}
 \includegraphics[scale=0.4]{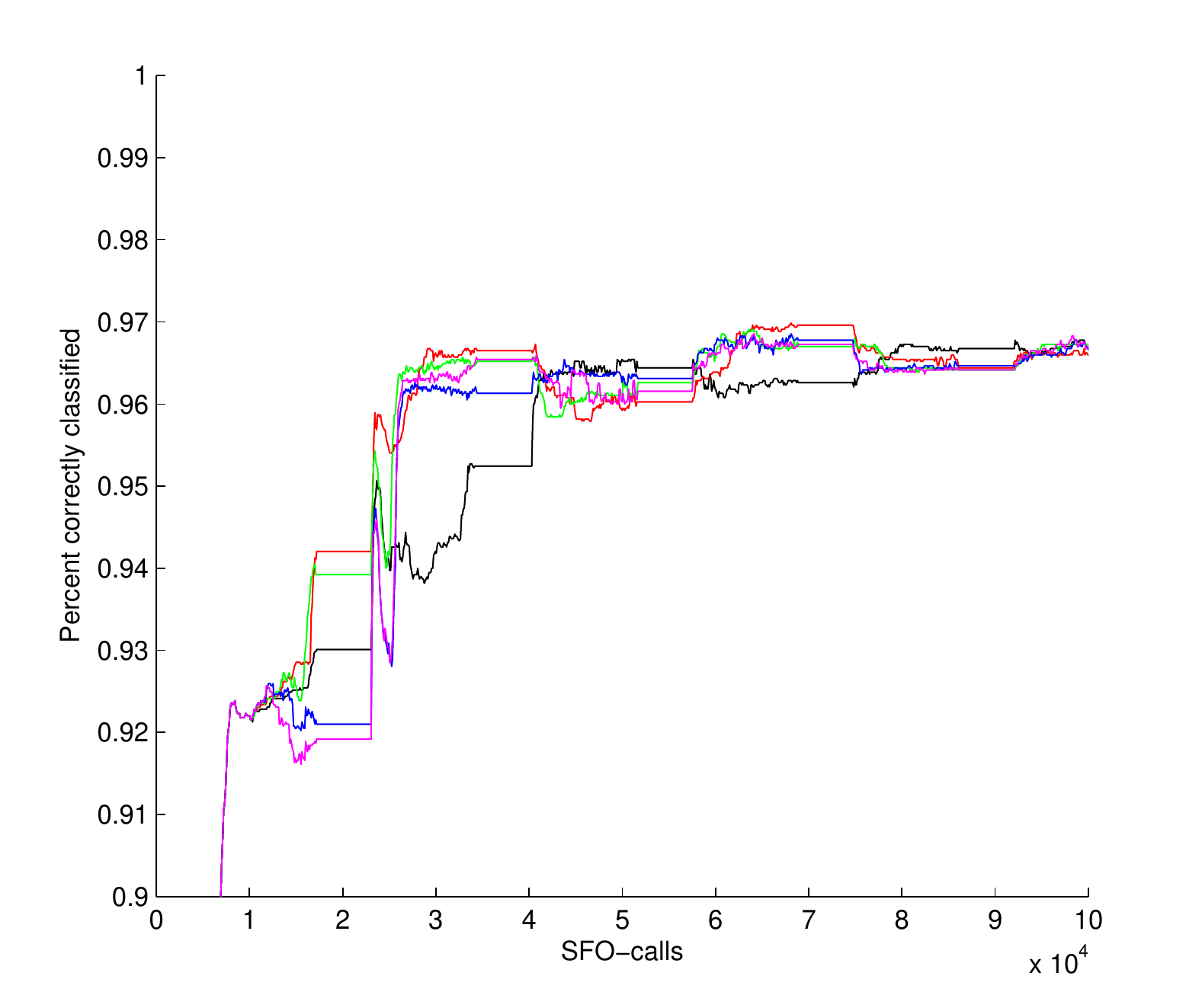}
}
\caption{Comparison of SdLBFGS-VR with different memory sizes $p$ on RCV1 data set. The step size is set as $\alpha=0.1$. We set $q=115$, and batch size $m=50$, so that $T\approx qm$.}
\label{RCV_diff-p}
\end{figure}

\begin{figure}[H]
\centering{\vspace{-8mm}
 \includegraphics[scale=0.4]{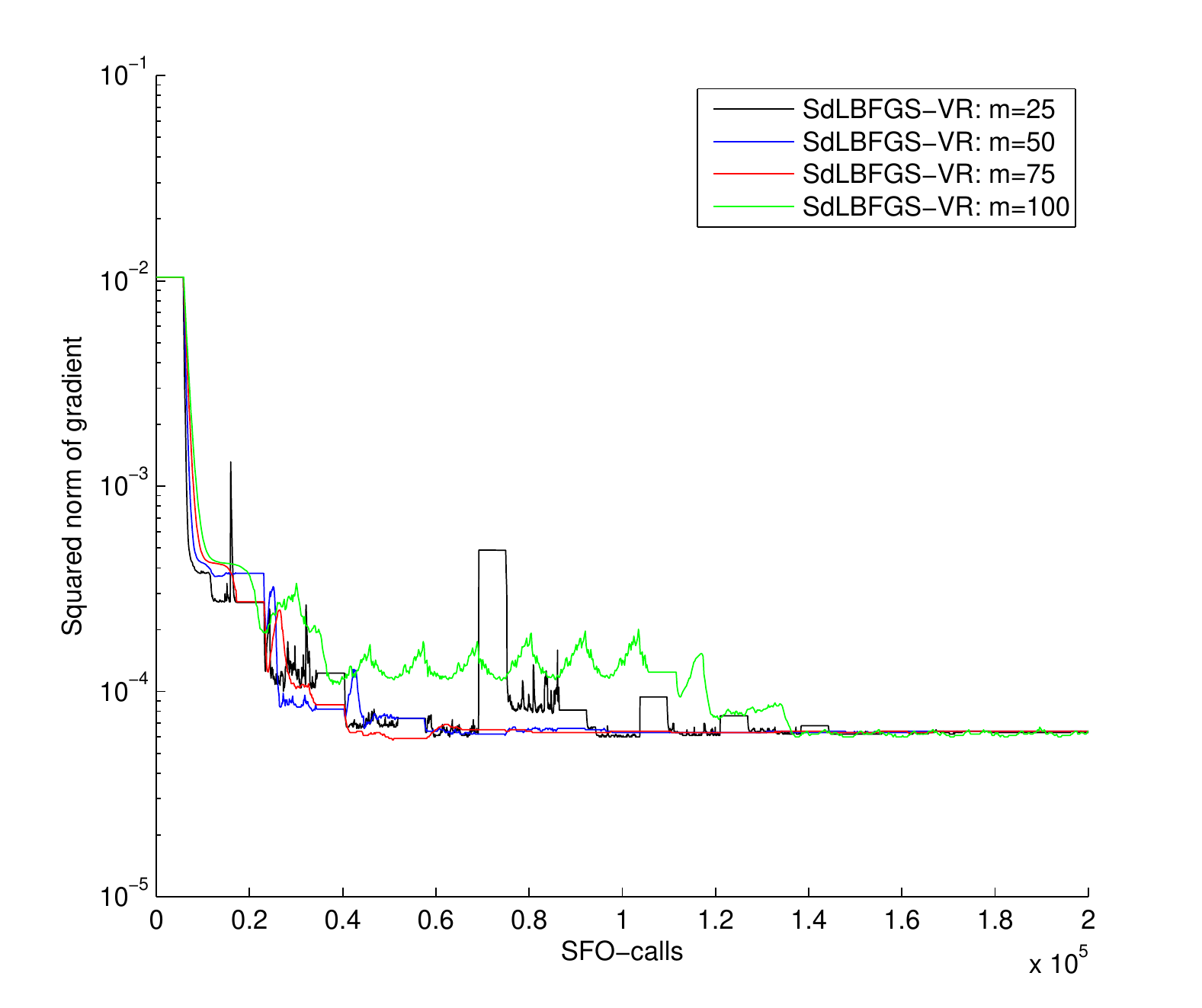}
 \includegraphics[scale=0.4]{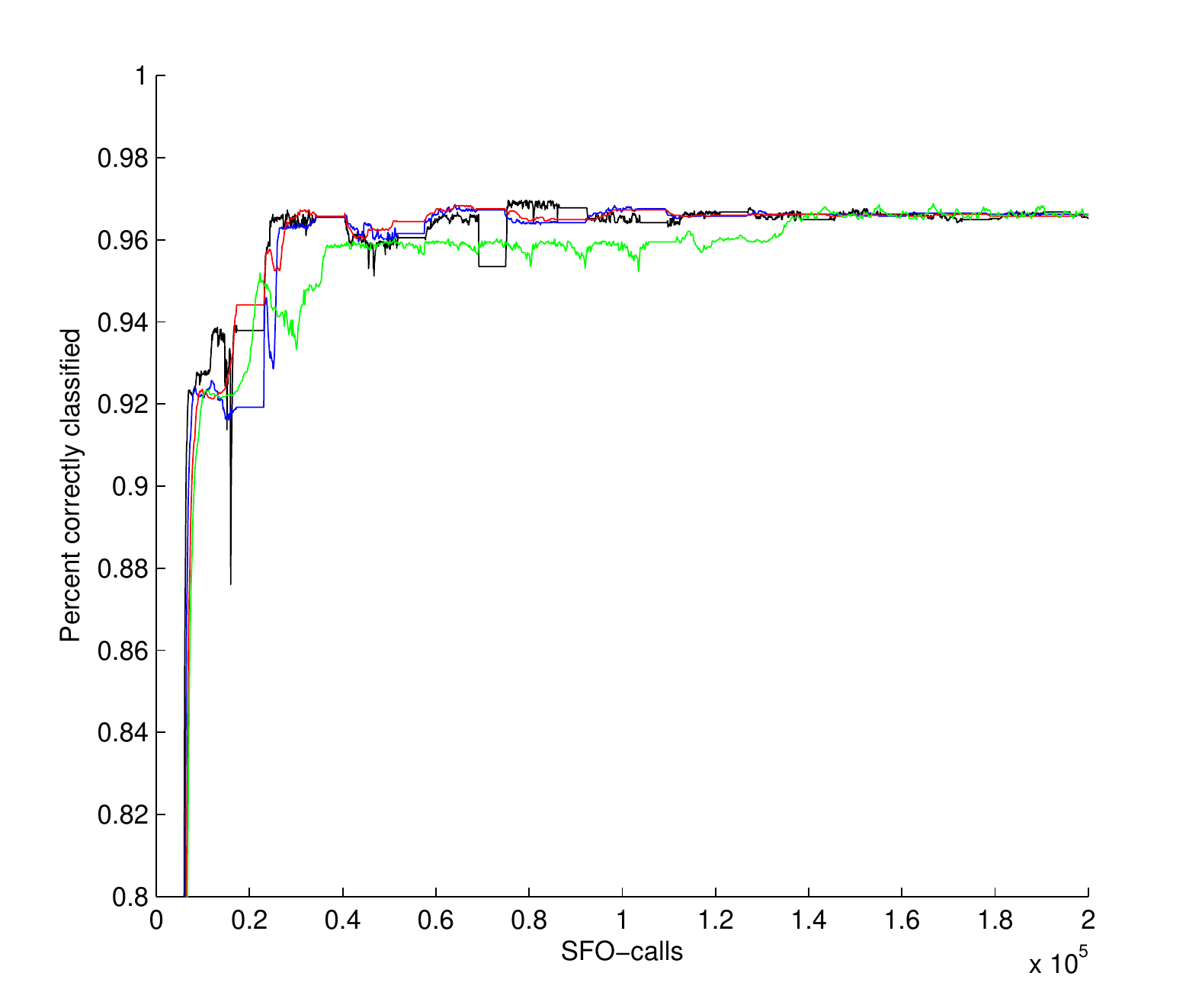}
}
\caption{Comparison of SdLBFGS-VR with different batch sizes on RCV1 data set. We set the step size as $\alpha=0.1$ and the memory size as $p=20$.}
\label{RCV-VR-diff-m}
\end{figure}

\begin{figure}[H]
\centering{\vspace{-8mm}
 \includegraphics[scale=0.4]{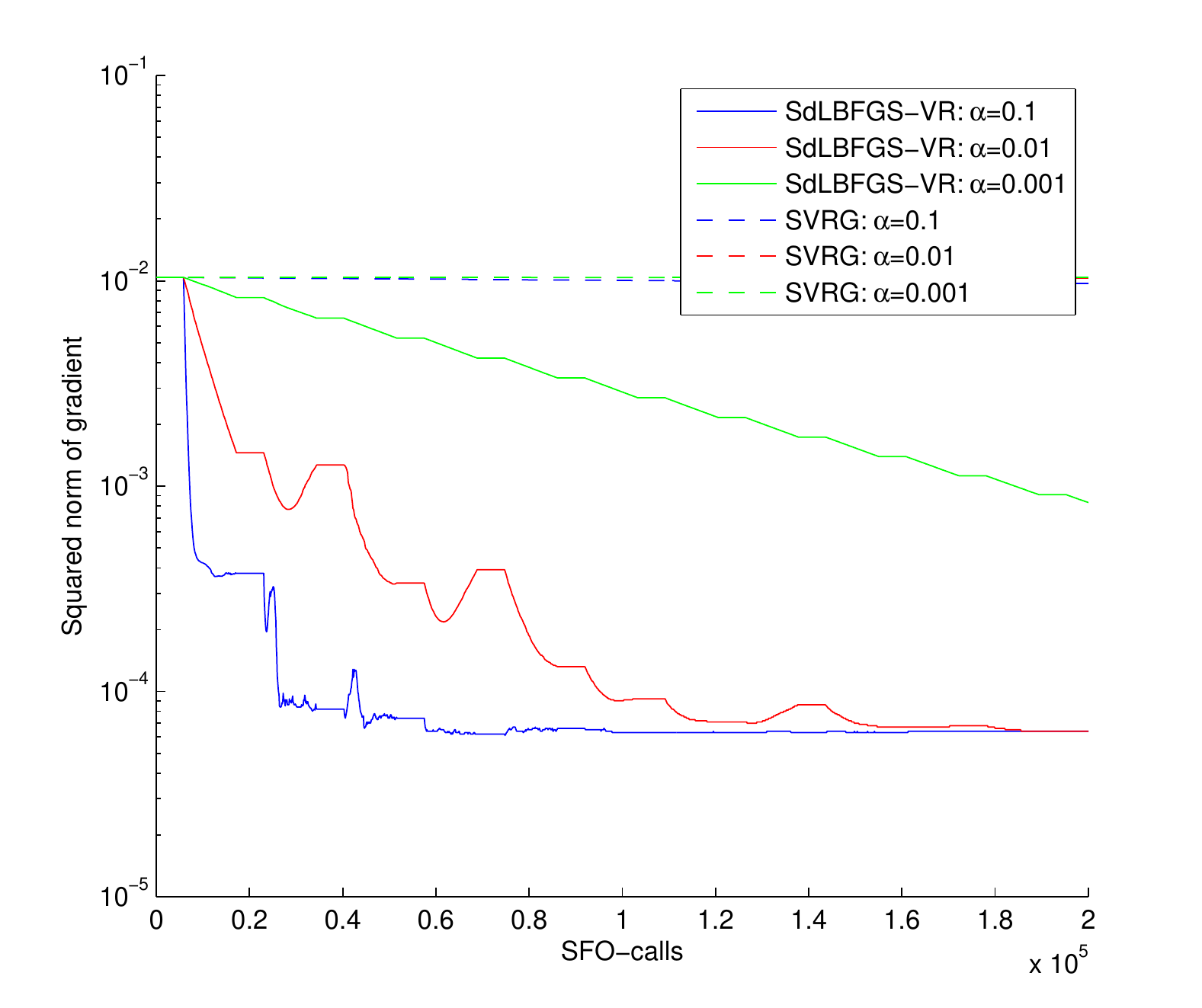}
 \includegraphics[scale=0.4]{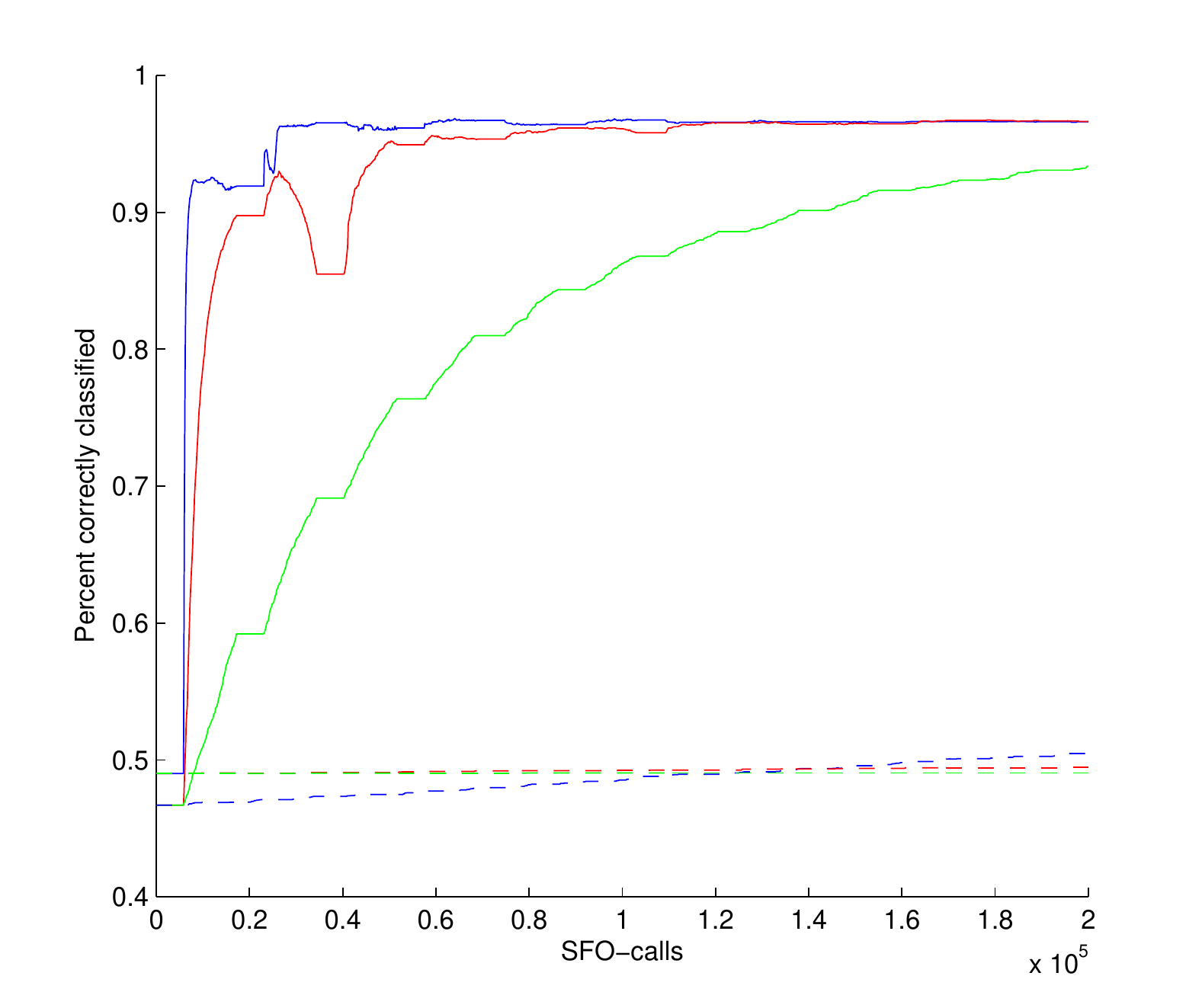}
}
\caption{Comparison of SdLBFGS-VR and SVRG on RCV1 data set with different step sizes. The memory size of SdLBFGS-VR is set as $p=20$. For both algorithms, we set the inner iteration number $q=115$ and the batch size $m=50$.}
\label{RCV_diff-beta2}
\end{figure}

\begin{figure}[H]
\centering{\vspace{-8mm}
\includegraphics[scale=0.4]{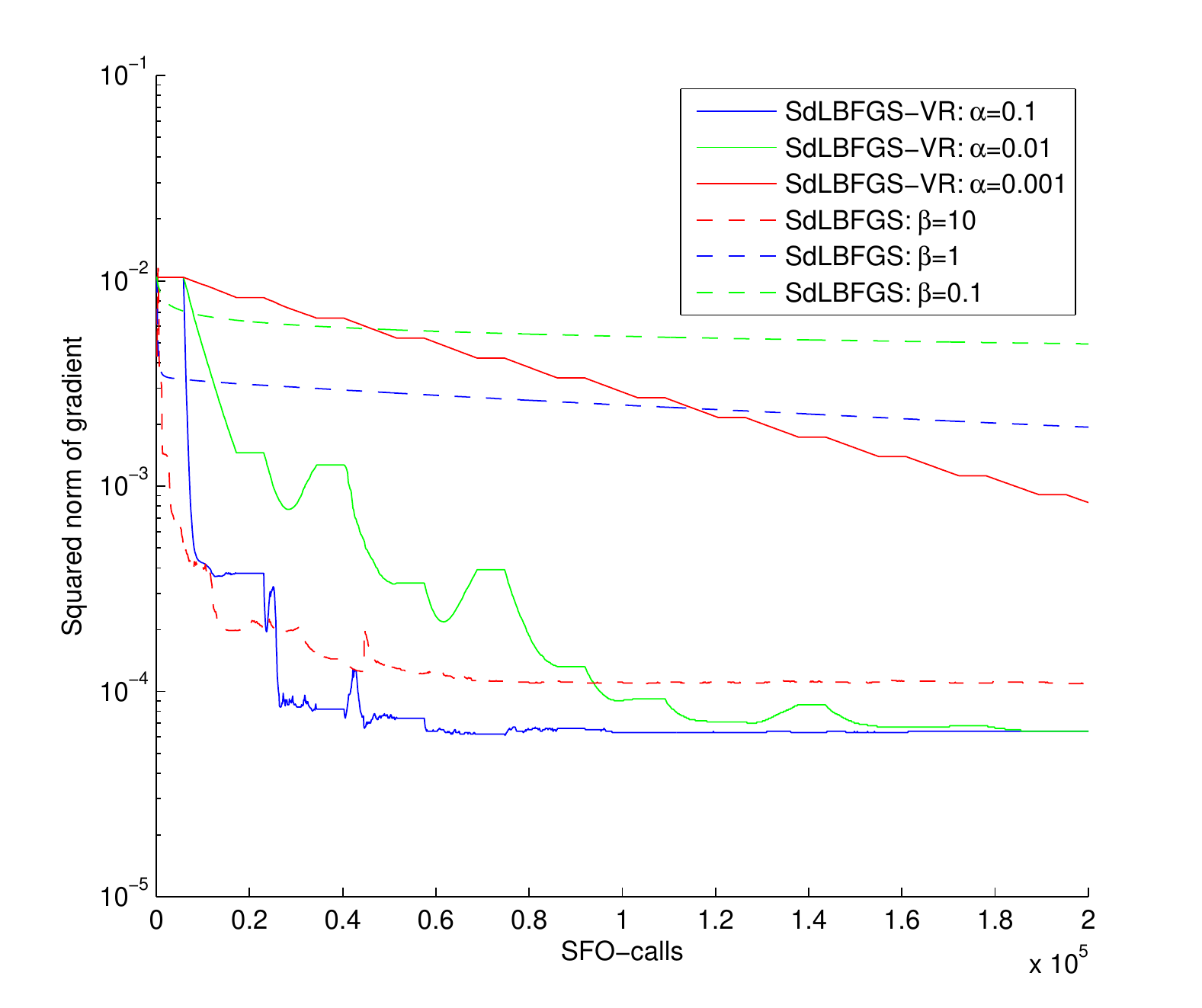}
 \includegraphics[scale=0.4]{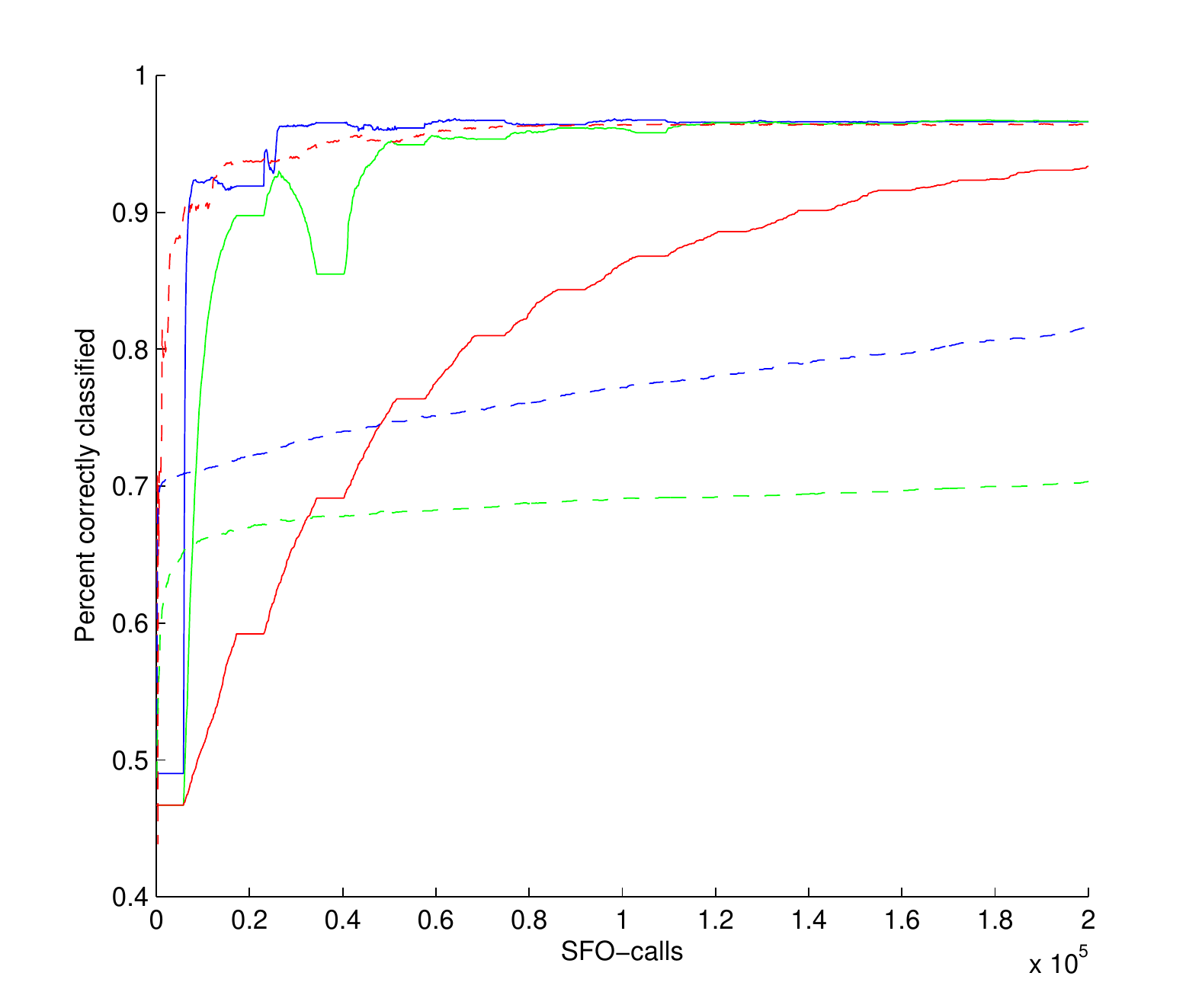}
}
\caption{Comparison of SdLBFGS-VR and SdLBFGS on RCV1 dataset. For both algorithms, the memory size is $p=20$ and batch size is $m=50$. For SdLBFGS-VR, the inner iteration number is $q=115$.}
\label{RCV-VR-diff-beta}
\end{figure}

\begin{figure}[H]
\centering{\vspace{-8mm}
 \includegraphics[scale=0.3]{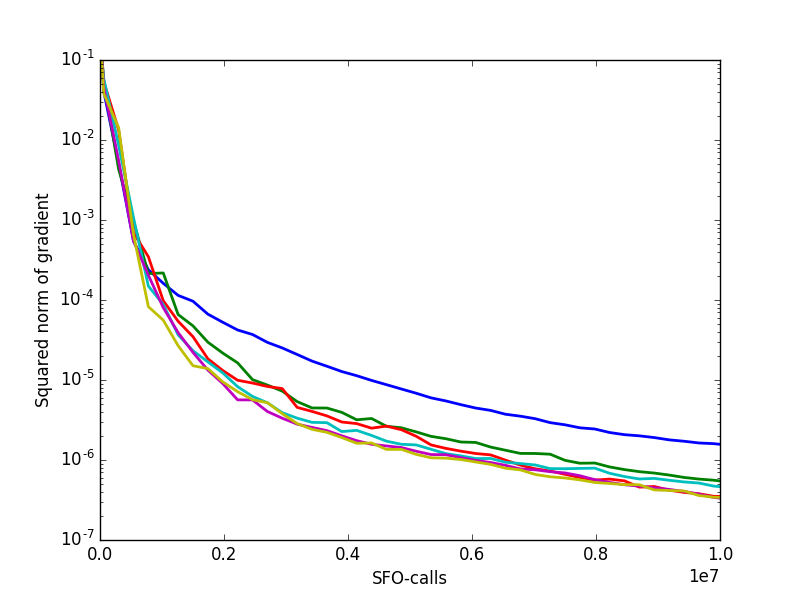}
 \includegraphics[scale=0.3]{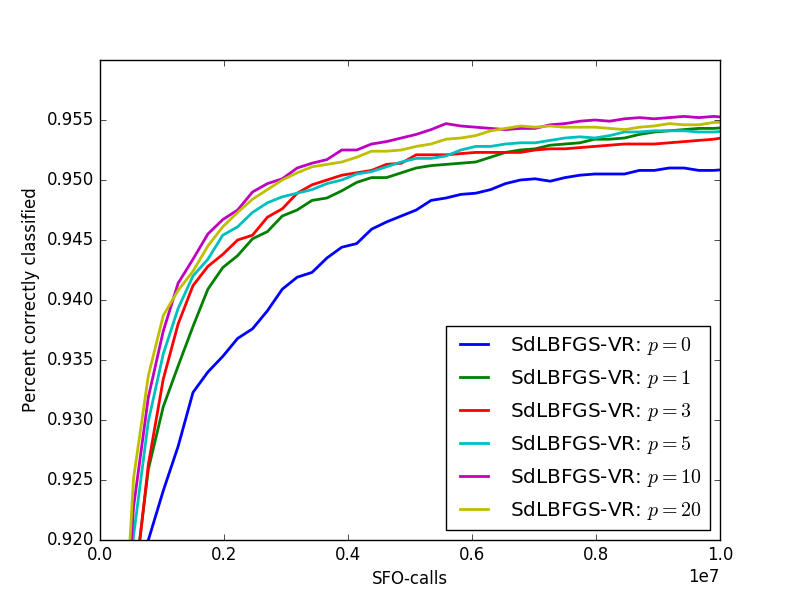}
}
\caption{Comparison of SdLBFGS-VR with different memory sizes $p$ on MNIST dataset. The step size is set as $\alpha=0.1$. We set $q=1200$, and batch size $m=50$, so that $T=qm$.}
\label{MNIST_diff-p}
\end{figure}

\begin{figure}[H]
\centering{\vspace{-8mm}
 \includegraphics[scale=0.3]{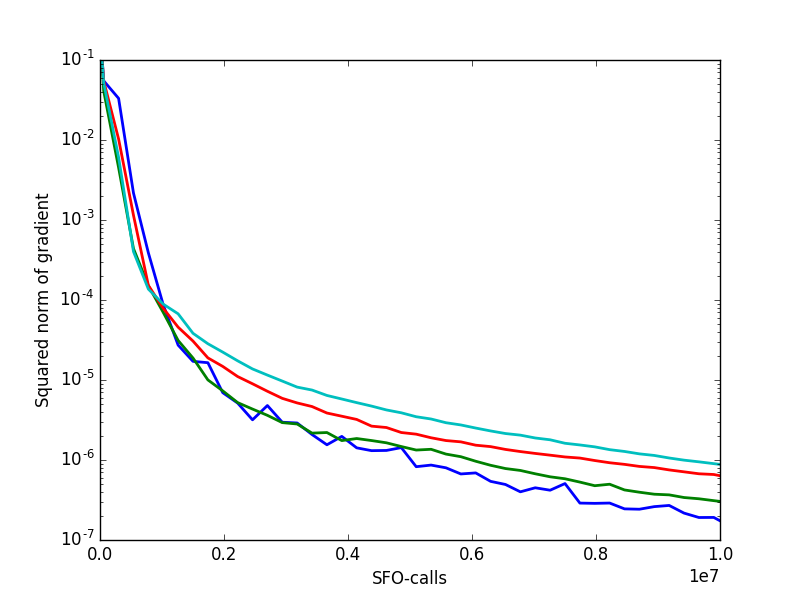}
 \includegraphics[scale=0.3]{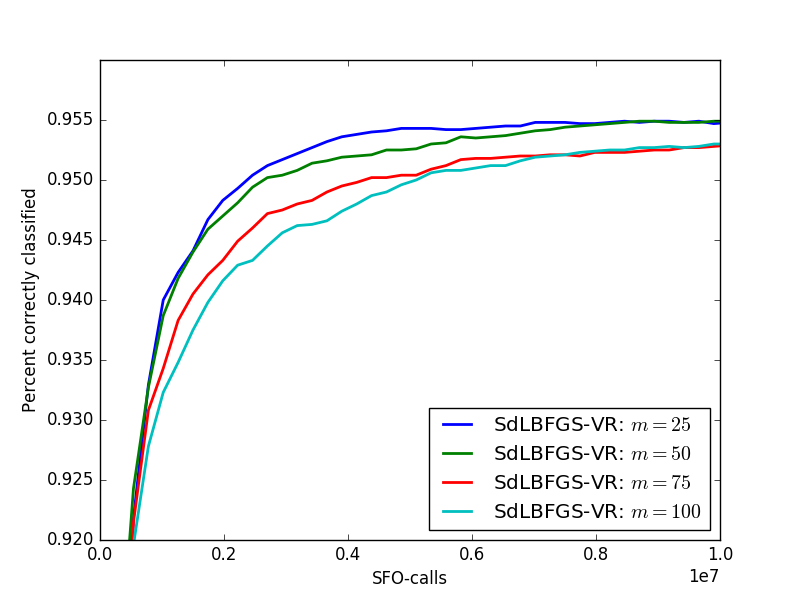}
}
\caption{Comparison of SdLBFGS-VR with different batch sizes on MNIST dataset. We set the step size as $\alpha=0.1$ and the memory size as $p=20$.}
\label{MNIST-VR-diff-m}
\end{figure}

\begin{figure}[H]
\centering{\vspace{-8mm}
 \includegraphics[scale=0.3]{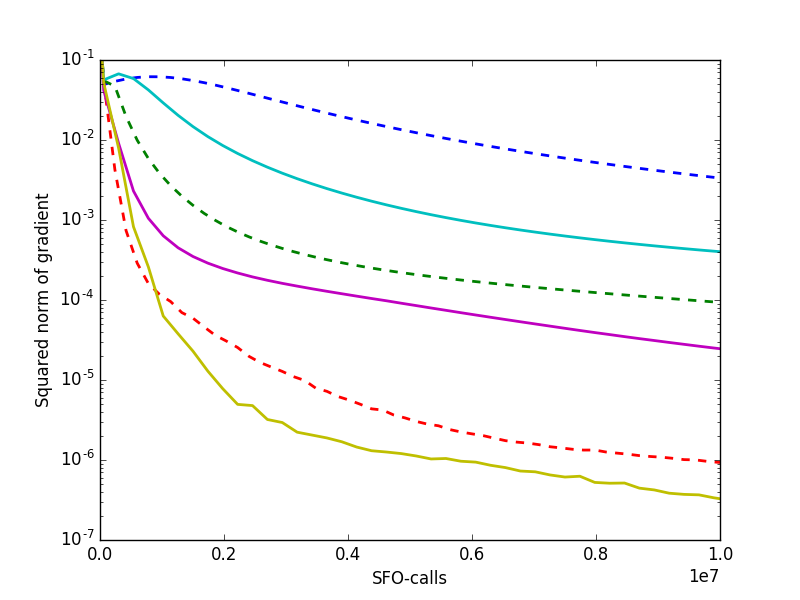}
 \includegraphics[scale=0.3]{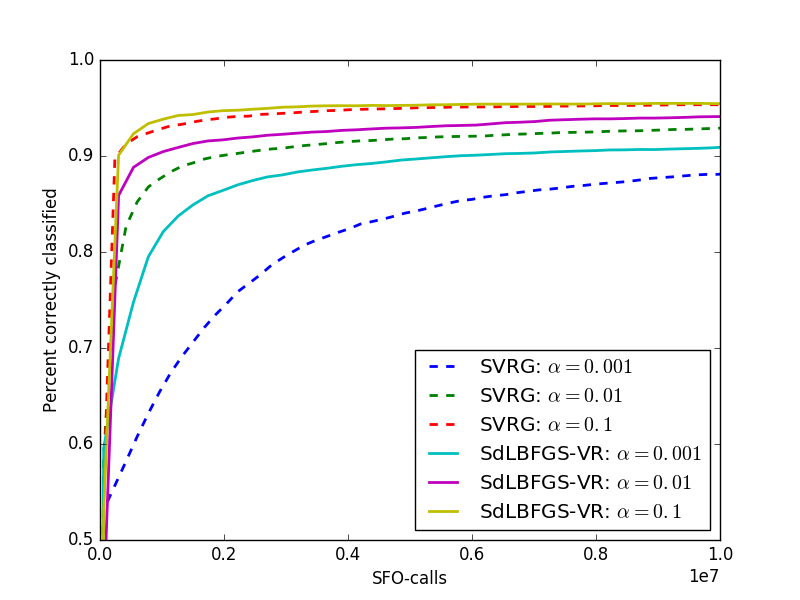}
}
\caption{Comparison of SdLBFGS-VR and SVRG on MNIST dataset with different step sizes. The memory size of SdLBFGS-VR is set as $p=20$. For both algorithms, we set the inner iteration number $q=1200$ and the batch size $m=50$.}
\label{MNIST_diff-beta2}
\end{figure}

\begin{figure}[H]
\centering{\vspace{-8mm}
\includegraphics[scale=0.3]{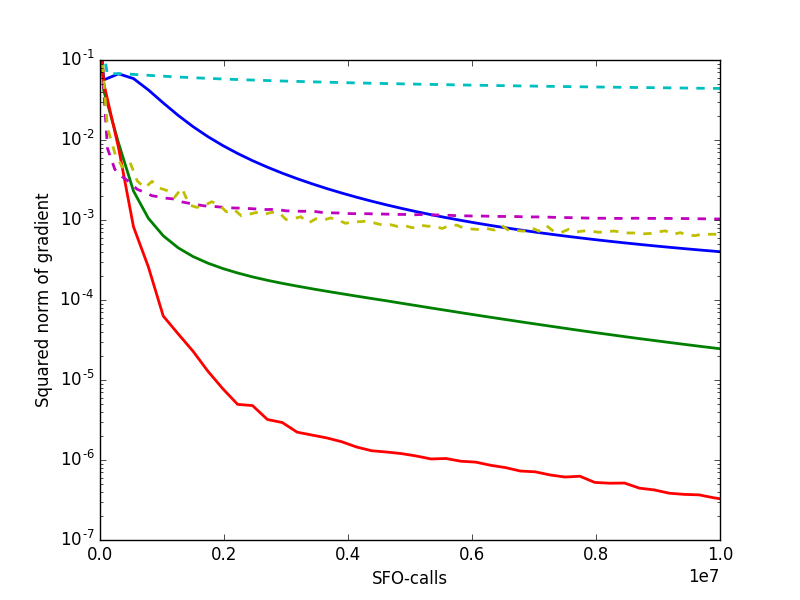}
\includegraphics[scale=0.3]{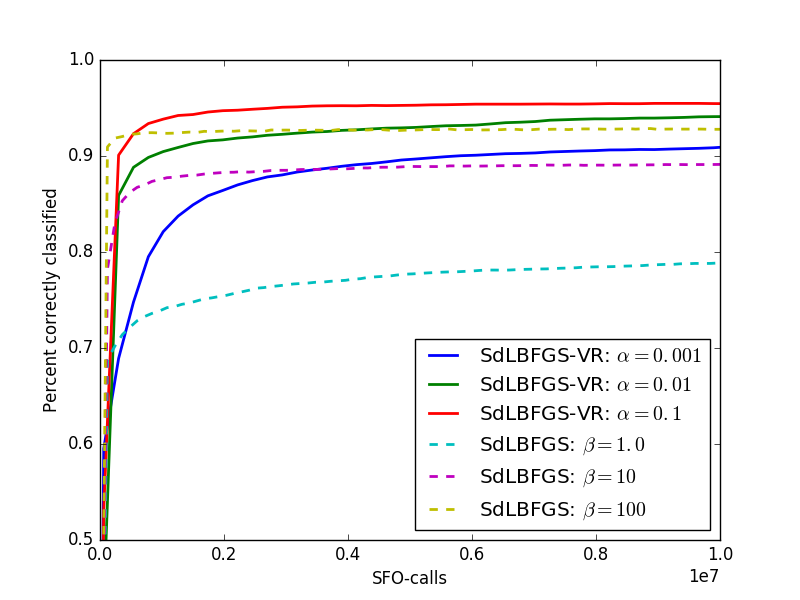}
}
\caption{Comparison of SdLBFGS-VR and SdLBFGS on MNIST dataset. For both algorithms, the memory size is $p=20$ and batch size is $m=50$. For SdLBFGS-VR, the inner iteration number is $q=1200$.}
\label{MNIST-VR-diff-beta}
\end{figure}


\end{document}